\documentclass[11pt,a4paper,reqno]{amsart}
\usepackage{dutchcal}
\usepackage[T1]{fontenc}
\usepackage{phaistos}
\usepackage{protosem}
\usepackage{amsmath}
\usepackage{amsthm}
\usepackage{amsfonts}
\usepackage{amssymb}
\usepackage{graphicx}
\usepackage{color}
\usepackage{amsbsy}
\usepackage{mathrsfs}
\usepackage{bbm}
\usepackage{verbatim}
\usepackage[colorlinks=true,linkcolor=black,citecolor=blue]{hyperref}
\newtheorem{theorem}{Theorem}[section]
\newtheorem{lemma}[theorem]{Lemma}

\newtheorem{proposition}[theorem]{Proposition}

\theoremstyle{remark}
\newtheorem{remark}[theorem]{\it \bf{Remark}\/}

\numberwithin{equation}{section}
\catcode`@=11
\def\section{\@startsection{section}{1}%
  \z@{1.5\linespacing\@plus\linespacing}{.5\linespacing}%
  {\normalfont\bfseries\large\centering}}
\catcode`@=12

\newcommand{\be}{\begin{equation}}
\newcommand{\ee}{\end{equation}}
\newcommand{\bea}{\begin{eqnarray}}
\newcommand{\eea}{\end{eqnarray}}
\newcommand{\bee}{\begin{eqnarray*}}
\newcommand{\eee}{\end{eqnarray*}}
\def\div{{\rm div \;}}

\def\pa{\partial}
\def\pr{\partial}

\DeclareMathOperator{\peye}{\textproto{o}}
\newcommand{\eye}{r_{\hskip -.1pc\peye}}

\def\RR{\mathbb{R}}

\def\d{\mathcal d}

\def\de{\delta}

\def\uh{\hat{u}}

\catcode`@=11
\def\supess{\mathop{\operator@font Sup\,ess}}
\catcode`@=12

\def\div{{\rm div \;}}

\def\RR{\mathbb{R}}

\def\e{\varepsilon}

\def\bar#1{{\overline #1}}

\def\R2+{\RR ^2_+}

\def\A{\mathcal A}

\def\pa{\partial}

\def\lim{\mathop{\rm lim}}

\def\l{\lambda}

\def\log{{\rm log}}

\def\e{\mathcal e}

\def\ut{\tilde{u}}

\def\rhoh{\hat{\rho}}
\def\Psih{\hat{\Psi}}

\def\T{\Theta}

\def\pa{\partial}

\def\rhot{\tilde{\rho}}
\def\Psit{\tilde{\Psi}}

\def\Lamdba{\Lambda}

\def\pa{\partial}
\def\la{\langle}
\def\matchal{\mathcal}
\def\ra{\rangle}

\def\NL{\textrm{NL}}

\def\Psipx{\Psi_{P}}
\def\rhopx{\rho_{P}}
\def\rhox{\rho}
\def\Psix{\Psi}
\def\qx{Q}

\def\rhox{\rho}
\def\Psix{\Psi}
\def\Phix{\Phi}
\def\rhoh{\hat{\rho}}
\def\Psih{\hat{\Psi}}

\def\Et{\tilde{\mathcal E}}

\def\Ht{\tilde{H}}
\def\ph{\hat{p}}
\begin{document}

\title[]{On the implosion of a three dimensional compressible fluid}
\author[F. Merle]{Frank Merle}
\address{AGM, Universit\'e de Cergy Pontoise and IHES, France}
\email{merle@math.u-cergy.fr}
\author[P. Rapha\"el]{Pierre Rapha\"el}
\address{Department of Pure Mathematics and Mathematical Statistics, Cambridge, UK}
\email{pr463@cam.ac.uk}
\author[I. Rodnianski]{Igor Rodnianski}
\address{Department of Mathematics, Princeton University, Princeton, NJ, USA}
\email{irod@math.princeton.edu}
\author[J. Szeftel]{Jeremie Szeftel}
\address{CNRS $\&$ Laboratoire Jacques Louis Lions, Sorbonne Universit\'e, Paris, France}
\email{jeremie.szeftel@upmc.fr}
\begin{abstract} 
We consider the compressible three dimensional Navier Stokes and Euler equations. In a suitable regime of barotropic laws,  we construct a set of finite energy smooth initial data for which the corresponding solutions to both equations implode (with infinite density) at a later time at a point, and completely describe the associated formation of singularity. Two essential steps of the analysis are the existence of $\mathcal C^\infty$ smooth self-similar solutions to the compressible Euler equations for quantized values of the speed and the derivation of spectral gap estimates for the associated linearized flow which are addressed in the companion papers  \cite{MRRSprofile, MRRSdefoc}. All blow up dynamics obtained for the Navier-Stokes problem are of type II (non self-similar).

\end{abstract}

\maketitle

%%%%%%%%%%%%%%%%%%%%%%%%%%%%%%%%%%%%%%%%%%%%%%%%%%
%%%%%%%%%%%%%%%%%%%%%%%%%%%%%%%%%%%%%%%%%%%%%%%%%%

\section{Introduction}

%%%%%%%%%%%%%%%%%%%%%%%%%%%%%%%%%%%%%%%%%%%%%%%%%%
%%%%%%%%%%%%%%%%%%%%%%%%%%%%%%%%%%%%%%%%%%%%%%%%%%

%%%%%%%%%%%%%%%%%%%%%%%%%%%%%%%%%%%%%%%%%%%%%%%%%%

\subsection{Setting of the problem}

%%%%%%%%%%%%%%%%%%%%%%%%%%%%%%%%%%%%%%%%%%%%%%%%%%

We consider the three dimensional barotropic compressible Navier-Stokes equation:
\be
\label{NScomp}
(\rm{Navier-Stokes})\ \ \left|\begin{array}{lll}\pa_t\rho+\nabla\cdot(\rho u)=0\\
\rho\pa_tu+\rho u\cdot\nabla u+\nabla \pi =\mu\Delta u+\mu'\nabla\div u\\
\pi=\frac{\gamma-1}{\gamma}\rho^\gamma\\
(\rho_{|t=0},u_{|t=0})=(\rho_0(x),u_0(x))\in \Bbb R^*_+\times \Bbb R^3
\end{array}\right.
\ee 
for $\gamma>1$, as well as the compressible Euler equations:
\be
\label{eulercomp}
(\rm{Euler})\ \ \qquad\qquad\left|\begin{array}{lll}\pa_t\rho+\nabla\cdot(\rho u)=0\\
\rho\pa_tu+\rho u\cdot\nabla u+\nabla \pi=0\\
\pi=\frac{\gamma-1}{\gamma}\rho^\gamma\\
(\rho_{|t=0},u_{|t=0})=(\rho_0(x),u_0(x))\in \Bbb R^*_+\times \Bbb R^3
\end{array}\right.
\ee 
with non-vanishing density $\rho>0$, but possibly decaying at $+\infty$ 
\be
\label{vneivenoenenevnove}
\lim_{|x|\to +\infty}\rho(t,x)=0.
\ee 
The problem of understanding global dynamics of classical solutions of compressible fluid dynamics is notoriously difficult, 
as was already observed in the 1-dimensional inviscid case by Challis, \cite{challis}. It becomes even more complicated in higher 
dimensions, including a physically relevant 3-dimensional problem, and in the viscous case due to the lack of access to the 
method of characteristics.
\subsection{Breakdown of solutions for compressible fluids}
For non-vanishing densities, smooth initial data satisfying appropriate fall-off conditions at infinity yield unique local in time strong solutions, \cite{nash,majda,mku,chemin,choe}. However,  for the Euler equations, it has been known since the pioneering work of Sideris \cite{sideris},  that well chosen initial data (with density which is constant outside of a large ball) cannot be continued for all times as strong solutions. The result applies to both large and ``small'' data and holds for all $\gamma>1$. 
Similarly, for the Navier-Stokes equations, there are regimes in which strong solutions to \eqref{NScomp} 
can not be continued, however such results require vanishing conditions on the data. 
It was first shown in \cite{xin} for all compactly supported data
and then in \cite{rozanova} for non-vanishing (density) but decaying at infinity data for $\gamma>\frac 65.$ 
In both Euler and Navier-Stokes cases the underlying convexity arguments give no insight into the nature of the singularity
and while for the compressible Euler equations subsequent work (see below) produced complete description of singularity (shock) formation (at least in the small data near constant density regime), the questions about quantitative singularity formation in
Navier-Stokes and in other Euler regimes remained open.\\

In this paper we address the classical problem of  singularity formation in compressible fluids 
arising from smooth well localized initial data with non-vanishing density. We study both the three dimensional Navier-Stokes equations and its inviscid Euler limit. For a suitable range of equations of state,  we exhibit a  class (finite co-dimension 
manifold in the moduli space) of smooth, well localized (without vacuum) initial data for which the corresponding solutions blow up in finite time at a point and completely describe the associated formation of singularity. The results also extend to the two dimensional Euler equations.
These solutions describe self-{\it implosion} of a fluid/gas in which smooth well localized (in particular finite energy) distribution of matter collapses
upon itself (with {\it infinite density}) in finite time while remaining smooth (in particular, free of shocks) until then. At the collapse time, remaining matter assumes a certain {\it universal} form.\\

With the focus on both the Navier-Stokes and the Euler equations we examine the question of failure of classical solutions to 
be continued globally in time. Specifically, we study the {\it first time singularity problem}, identifying the first time that solutions stop being classical and the singular set on which it happens. In the Navier-Stokes case such results are completely new. For the Euler equations in two and three dimensions such results are connected with the more general singularity formation in quasilinear hyperbolic equations and originate in the works of John \cite{john} and Alinhac \cite{alinhac,alinhac1}. In the Euler case, due to the hyperbolic nature of the equations, one can also study a richer problem of {\it shock formation} which in particular addresses the structure of the {\it full} singular set of the solutions.

%%%%%%%%%%%%%%%%%%%%%%%%%%%%%%%%%%%%%%%%%%%%%%%%%%

\subsection{Quantitative theory of singularity formation for the compressible Euler}

%%%%%%%%%%%%%%%%%%%%%%%%%%%%%%%%%%%%%%%%%%%%%%%%%%
 
We (mostly) limit our discussion to the three dimensional case and completely bypass the rich and storied narrative of the one dimensional case, see e.g. \cite{dafermos}. Shock formation for the three dimensional Euler equations was shown 
in the
work of Christodoulou \cite{christodoulou} in both the relativistic and non-relativistic cases (see also \cite{chm}.) 
The work covered a small data regime of near constant density and small velocities, with the shock forming in the 
irrotational part of the fluid, and provided a complete geometric description at the shock. One of the key features of the 
work and the reason why the result may be called  ``shock formation" is that it constructed and  showed a particular structure of the {\it maximal Cauchy development} of solutions. Such a maximal Cauchy development possesses a boundary $\partial \mathcal H\cup\mathcal H\cup \mathcal C$, part of which -- a smooth null 
3-d hypersuface $\mathcal H$ and 2-d sphere $\partial \mathcal H$ -- is the singular set of the solution. The past endpoints   
$\mathcal H$ are precisely the set $\partial \mathcal H$ --  the {\it first singularity} of the solution. It is also that aspect of the construction that later allowed Christodoulou to study the (restricted) 
{\it shock development} problem, \cite{christ}. \\

While shock formation and shock development problems require studying the maximal Cauchy development and 
the associated first singularity, one could, especially in the non-relativistic setting where the time variable $t$ is well defined,
investigate the problem of the {\it first time singularity.} That problem amounts to understanding a singular set of the solution at the first time $T$ when it becomes singular. In the setting described above, this would be the set 
$T\cap \partial \mathcal H\cup\mathcal H\subset \partial \mathcal H\cup  \mathcal H$ which a priori may not coincide with the first singularity set $\partial \mathcal H$ (or even have the same dimension). On the other hand, just 
the knowledge of the first time singular set provides no information about the maximal Cauchy development, the full singular set or shock formation. In fact, in principle, it may be completely consistent with the full singular set being a 3-d {\it space-like} hypersurface, rather than the null $\mathcal H\cup\partial \mathcal H$, and thus be incompatible with 
the shock development picture. (For the multi-dimensional semilinear wave equations examples of singular sets have been considered and analyzed in 
e.g. \cite{caf,kichen,mz,mz1}.)
\\

Having drawn a distinction between the first singularity (shock formation) and first time singularity formation, we should recall again that the latter 
problem for multi-dimensional compressible Euler equations
had been studied in the works of Alinhac (with a precursor in John, \cite{john}) in two and three dimensions for a more general (quasilinear hyperbolic) 
class of equations, 
including Euler, \cite{alinhac,alinhac1}, in the small data regime and was tied to the failure of Klainerman's 
null condition, \cite{kl}, and to a {\it 1-dimensional Burgers mechanism} of singularity formation. Recently, this has been extended in \cite{speck};
open set of data leading to solutions of the Euler equations with 
{\it non-trivial vorticity} at the first time singularity have been constructed in \cite{ls} and later, in different regimes in \cite{BSV,BSV1}. The 1-dimensional Burgers phenomenon has been lifted to higher dimensions also recently in \cite{cmt} 
for the Burgers equation with transverse viscosity.\\

\subsection{Results}
We now contextualize our results. Once again we limit our discussion to the three dimensional case. There are three 
critical issues. 

First, since in this paper we study the Navier-Stokes and Euler problems {\it simultaneously}, we can 
not even define maximal Cauchy development, which is associated with hyperbolic PDE's, and thus properly speak about 
shock formation. Ours is a first time singularity result. 

Secondly, shock formation  and development for the 
three dimensional compressible Euler equations has been shown only in the small data, near constant density, regime. 
For such data Navier-Stokes solutions remain global, \cite{mn}. Our solutions to both Navier-Stokes and Euler belong to a very 
different, large data regime. For Navier-Stokes, in this regime the density decays at infinity. For Euler, in view of the domain of dependence principle, behavior at infinity, in principle,  is less important. See also comment 5. after the statement of the main theorem.

Lastly, the first time singularities constructed in this paper occur at one point. We do not speculate about the structure of the 
full singular set.  However, we emphasize two important issues. One is that at a singular point {\it all} directions are singular,
unlike the picture established in \cite{christodoulou} where each point of the singular set $\partial \mathcal H$ possesses 
3 regular tangential direction (along $\mathcal H$).  The other one is perhaps the most important 
point: in formation of shock type singularities one expects to maintain boundedness of both density and velocity 
(with their first derivatives blowing up). In solutions constructed in this paper both density and velocity {\it blow up} at the 
singularity. This is a new phenomenon of formation of {\it strong singularities}. It relies on the existence of appropriate self-similar solutions to the 
Euler equations and makes no connection to the link between the Euler equation and explicit solutions of the Burgers equation.

 %%%%%%%%%%%%%%%%%%%%%%%%%%%%%%%%%%%%%%%%%%%%%%%%%%

%%%%%%%%%%%%%%%%%%%%%%%%%%%%%%%%%%%%%%%%%%%%%%%%%%

\subsection{Statement of the result}

%%%%%%%%%%%%%%%%%%%%%%%%%%%%%%%%%%%%%%%%%%%%%%%%%%

We recall that $\gamma$ is the parameter describing the equation of state and define the following additional parameters:
\be
\label{assumptionnonlinearity}
\left|\begin{array}{l}
\ell=\frac{2}{\gamma-1}\\
r^*(d,\ell)=\frac{d+\ell}{\ell+\sqrt{d}},\\
r_+(d,\ell)=1+\frac{d-1}{(1+\sqrt{\ell})^2}\\
\eye(d,\ell)=\left|\begin{array}{l} r^*(d,\ell)\ \ \mbox{for}\ \ \ell<d\\
r_+(d,\ell)\ \ \mbox{for}\ \ \ell>d.
\end{array}\right.
\end{array}\right.
\ee
The standard physical restrictions on the viscosity parameters $\mu,\mu'$ is given by 
$$
\mu\ge 0,\quad 3\mu'-\mu\ge 0
$$
 In what follows, we will allow the weaker assumption 
 $$
\mu+\mu'\ge 0
$$

\begin{theorem}[Implosion for a three dimensional compressible fluid]
\label{thmmain}
There exists a (possibility empty) exceptional countable sequence $(\ell_n)_{n\in \Bbb N}$ whose accumulation points can only be at $\{0,3,+\infty\}$ such that the following holds.  Let $\ell$ be related to $\gamma$ according to \eqref{assumptionnonlinearity}, and assume
\be
\label{ranegeulerns}
\left|\begin{array}{l}
\ell \neq 3\\
\ell>\sqrt{3} \ \ \mbox{for}\ \ (\rm{Navier-Stokes})\\
\ell>0 \ \ \mbox{for}\ \ (\rm{Euler})
\end{array}\right.
\ee 
and $\ell$ avoids the countable values:  
\be
\label{admissibiblevlaur}
\ell\notin\{ \ell_n, n\in \Bbb N\}.
\ee
Then for each such admissible $\ell$,
there exists a discrete sequence of blow up speeds $(r_k)_{k\ge 1}$ with $$1<r_k<\eye(3,\ell), \ \ \lim_{k\to+\infty} r_k=\eye(3,\ell)$$ such that for each $k\ge 1$,  there exists a finite co-dimensional manifold (in the moduli space) of smooth spherically symmetric initial data $(\rho_0,u_0)\in \cap_{m\ge 0} H^m(\Bbb R^3,\Bbb R^*_+\times \Bbb R^3)$ such that the corresponding solutions to {\em both} \eqref{NScomp} and \eqref{eulercomp} in their respective regimes \eqref {ranegeulerns} blow up 
in finite time $0<T<+\infty$ at the center of symmetry with
\be
\label{vmempemempevp}
\|u(t,\cdot)\|_{L^\infty}=\frac{c_{u_0}(1+o_{t\to T}(1))}{(T-t)^{\frac{r_k-1}{r_k}}}\qquad 
\|\rho(t,\cdot)\|_{L^\infty}=\frac{c_{\rho_0}(1+o_{t\to T}(1))}{(T-t)^{\frac{\ell(r_k-1)}{r_k}}}
\ee
for some constants $c_{\rho_0},c_{u_0}> 0$. 
\end{theorem}

\begin{remark}
A corresponding statement holds for Euler in dimension 2 in the range $\ell>0$, $\ell\neq 2$, see the third comment of section \ref{sec:commentsonthemainresult}.
\end{remark}

\subsection{Comments on the result}\label{sec:commentsonthemainresult}

We begin our discussion by emphasizing the point that for the Navier-Stokes equations the results of Theorem \ref{thmmain}
{\it do not} describe a self-similar (type I) singularity formation. The blow up profile dominating the behavior on the approach to singularity is a {\it front} for the Navier-Stokes equations and obeys (one of) the Euler scalings\footnote{The Euler equations possess a 2-parameter family of scaling transformations containing a 1-parameter family of Navier-Stokes as a subfamily. The parameter $r$ -- what we call here {\it speed} --- labels a particular choice of a 1-parameter subfamily of the scaling transformations of the Euler equations.} rather than the Navier-Stokes one. The scaling is super-critical for the Navier-Stokes problem:  the scale invariant Sobolev norm\footnote{The Navier-Stokes scaling preserves the $\|\rho(t,\cdot)\|_{\dot H^{s_{NS}}}$ with $s_{NS}=1+\frac 1{2\gamma}$, while the Euler scaling used for the profile preserves the Sobolev norm with the exponent $s_c=\frac 12+\frac 1r$. The condition \eqref{conditione} $\e>0$ which dictates the compatibility of Eulerian regimes with Navier-Stokes is precisely $s_c<s_{NS}$, which means that the scale invariant Navier-Stokes Sobolev norm blows up.} blows up 
at the singular time. Blow up is therefore of type II similar to our previous work \cite{MRRnls}.\\

\noindent{\em 1. Inviscid limit}. The results of Theorem \ref{thmmain} are {\it uniform} relative to the viscosity parameters $\mu, \mu'$ of the Navier-Stokes equations. Neither the sequence $\ell_n$ nor the blow up speeds $r_k$ 
depend on $\mu,\mu'$. Moreover, the singularity formation in Navier-Stokes survives in the inviscid limit. To describe that we note that the finite co-dimensional manifold of data, for which our results hold, 
are constructed as a pair of elements $(x,x')$, where $x$ belongs to a small ball (in appropriate topology) in the 
linear space $U$, while $x'=\Phi_{\mu,\mu'}(x)$ lies in $V$. 
The linear stable space $U$ and (finite-dimensional) unstable space $V$ are the {\it same} for all viscosity parameters $\mu,\mu'$
and constitute a decomposition of the full Hilbert space $H=U\bigoplus  V$.
The nonlinear map $\Phi_{\mu,\mu'}(x)$ is uniformly bounded with respect to the parameters $\mu,\mu'$ and admits a limit in the regime $\mu>0$, $\mu'>0$  and $\mu, \mu'\to 0$. 

As a consequence, singularity formation in the Euler equations
in this paper falls into two categories: in the Navier-Stokes regime $\ell>\sqrt{3}$ singular solutions of the Euler equations also correspond to (and arise as limits of) singular solutions of Navier-Stokes; in the remaining allowed range $\ell<\sqrt{3}$ singular solutions of the Euler equations do not have their viscous analogs. We should however stress 
that both in the Navier-Stokes regime and the ``pure'' Euler regimes blow up  occurs via a self-similar 
 {\it Euler} profile.
 \\

\noindent{\em 2. The range \eqref{ranegeulerns}}. The value $\ell=3$ or $\gamma=\frac 53$, which corresponds to the law for a 
monoatomic ideal gas, is exceptional and signals a phase transition from the blow up rate $r^*(3,\ell)$ for $\ell<3$ to 
$r_+(3,\ell)$ for $\ell>3$.
The nature of the phase portrait underlying the existence of suitable blow up profiles for Euler degenerates dramatically for $\ell=3$ with the formation of a critical triple point, \cite{MRRSprofile}. In the general dimension $d$ this phenomenon happens at $\ell=d$. The lower bound restriction $\ell>\sqrt{3}$ for the Navier-Stokes problem is also essential and sharp and measures the compatibility of the Euler-like blow up with the dissipation term in the Navier-Stokes equations. Viewing dimension $d$ as a parameter, this compatibility can be sharply measured by the condition, see \eqref{conditione}:\\
\underline{$\ell<d$}:
\be\label{rdl}
r^*(d,\ell)=\frac{d+\ell}{\ell+\sqrt{d}}>\frac{2+\ell}{1+\ell}\Leftrightarrow \ell>\ell_0(d)=\frac{2\sqrt{d}-d}{d-1-\sqrt{d}}
\ee
which  {\em always} holds  for $d\ge 4$ (all terms $\ge 0$), {\em never} holds for $d=2$ (all terms $<0$), and for $d=3$ demands $\ell>\sqrt{3}$, this is the lower bound \eqref{ranegeulerns}.\\
\underline{$\ell>d$}: 
$$
r_+(d,\ell)=1+\frac{d-1}{(1+\sqrt{\ell})^2}>\frac{2+\ell}{1+\ell}
$$
also never holds for $d=2$ but always holds for $d=3$, $\ell>3$.\\
This shows the fundamental influence of {\em both} the dimension and the blow up speed, attached to the Eulerian regime, on the strength of dissipation for fluid singularities.\\

\noindent{\em 3. The Euler case}. Our theorem also holds 
for the two dimensional Euler equations in the range $\gamma>1$ and $\gamma\ne 2$. Both the inviscid limit statement and the validity of the ``pure'' Eulerian regimes $(d=3, \ell<\sqrt{3})$, $(d=2,\ell>0)$, arise from the proof of the theorem. Let us note that in the case of Euler, a direct analysis of the dynamical system governing the self similar dynamics \cite{guderley, MRRSprofile} easily produces a continuum of self-similar solutions which, in principle, using the finite speed of propagation, one could try to localize, to produce finite energy self similar blow up solutions. These solutions however arise from the data of limited regularity, see section \ref{veniovninenein}. This procedure cannot be applied in the Navier-Stokes case, and, more generally, our understanding of the {\em finite co-dimensional stability} of these self similar solutions is directly linked to the $\mathcal C^\infty$ regularity. \\

\noindent{\em 4. The sequence $\ell_n$}. The discrete sequence $\ell_n$ of possibly non admissible equations of state is related to the existence of $\mathcal C^\infty$ self similar solutions to the compressible Euler. We proved in \cite{MRRSprofile} that for all $d\ge 2$, such profiles exist for discrete values of the blow up speed in the vicinity of the limiting speed $\eye(d,\ell)$ provided a certain non vanishing condition $S_\infty(d,\ell)\neq 0$ holds. The function $S_\infty(d,\ell)$ is given by an explicit series and is holomorphic in $\ell$ (in a small complex neighborhood of each interval $(0,d)$ and $(d,\infty)$). We do not know how to check the non vanishing condition analytically, but we can prove that the possible zeroes of $S_\infty(d,\cdot)$ are isolated and possibly accumulate only at $\ell\in\{0,d\}$. For small $\ell$, this condition can easily be checked numerically, but the series becomes exceedingly small as $\ell\to d$ and hence the numerical check of a given value becomes problematic, see \cite{MRRSprofile}. We do not know whether the condition $S_\infty(d,\ell)\neq 0$ is necessary for the existence of $\mathcal C^\infty$ self-similar profiles, understanding this would require revisiting the asymptotic analysis $r\uparrow \eye(d,\ell)$ performed in \cite{MRRSprofile} in the degenerate case.\\

\noindent{\em 5. Behavior at infinity \eqref{vneivenoenenevnove} and other domains}. In this paper 
our results apply to the solutions $(\rho,u)$ which decay at infinity. As such,
the solutions have finite energy. However, from that point of view it is unnecessary for both $\rho$ and $u$ to decay.
A particularly interesting case is when $\rho$ approaches a constant at infinity and $u$ vanishes appropriately. For Navier-Stokes such solutions are specifically excluded even from qualitative arguments in \cite{xin,rozanova}.  Our analysis
begins with a construction of self-similar Euler profiles which decay rather slowly. In particular, $\rho\sim {|x|^{-2\frac{r_k-1}{\gamma-1}}}$. For $|x|>5$ we then reconnect our profiles to rapidly decaying functions and consider 
similarly rapidly decaying perturbations. The reconnection procedure is not subtle and its main goal is to create solutions 
of finite energy. One could, in principle, be able to reconnect the profile to one with constant density for large $x$ and 
rapidly decaying velocity, instead. This should lead to a singularity formation result for Navier-Stokes for solutions with 
constant density at infinity. 
Even more generally, the analysis should be amenable to other boundary conditions and domains, e.g. Navier-Stokes and
Euler equations on a torus. An example of such adaptation in the context of a nonlinear heat equation and a domain with Dirichlet boundary condition is provided by \cite{collot}.\\

\noindent{\em 6. Spherical symmetry assumption.} Theorem \ref{thmmain} is proved for spherically symmetric initial data. 
The symmetry is used in a very soft way, and we expect that the blow up  of Theorem \ref{thmmain} is stable modulo finitely many instabilities for non symmetric perturbations, including in particular solutions with non trivial vorticity.\\

\noindent{\em 7. Blow up profile.} The proof of Theorem \ref{thmmain} involves a much more precise description of the blow up  \eqref{vmempemempevp}. In particular, we prove that, after renormalization, the blow up profile  is given by a suitable self-similar solution to the compressible Euler flow, and that singularity occurs at the origin only, with a universal blow up profile away from the singularity, as is also the case in some examples of blow up for the Schr\"odinger equations, see e.g. \cite{MRprofile}. The proof of our main result also implies the existence of the limits for the density $\rho(t,x)$ and velocity $u(t,x)$
as $t$ converges to the blow up time $T$ and $|x|>0$. One can show that for any $x$: $0<|x|<5$,
\be\label{lim}
\lim_{t\uparrow T}\rho(t,x)=\frac {\rho_*}{|x|^{2\frac{r_k-1}{\gamma-1}}}+O(1),\ \ \lim_{t\uparrow T}u(t,x)= \frac {u_*}{|x|^{(r_k-1)}}+O(1).
\ee
for some (universal) constants $\rho_*>$ and $u_*$. Note that the limiting profile \newline$(\frac {\rho_*}{|x|^{2\frac{r_k-1}{\gamma-1}}}, \frac {u_*}{|x|^{(r_k-1)}})$ is {\it not} a solution of the Euler equations. We should emphasize, that in contrast to the previously studied (in mathematical literature) singularity and shock formation for the two and three dimensional Euler equations where solutions remain bounded up to and including the first 
singularity, both the density and velocity of our solutions blow up at the first singularity.\\

\noindent{\em 8. The stability problem.} The results of Theorem \ref{thmmain} hold for a ball in the moduli space of initial data around the self similar profile modulo
a  finite number of unstable directions, possibly none. The proof comes with a complete understanding of the associated linear spectral problem. Providing a {\it precise} count for (non real valued) eigenvalues analytically does not seem obvious, but clearly this problem can be addressed numerically since the radial nature of the self-similar profile allows one to reduce the problem to standard ode's. This remains to be done.\\

\noindent{\em 9. Weak solutions}. 
Solutions to the compressible Navier-Stokes equations constructed in this paper coexist, in principle, with the theory of {\it weak} 
global solutions of P.-L. Lions \cite{lions} and its extension in \cite{fr}. 
Existence of weak global solutions is asserted under finite energy assumptions and
in the range $\gamma>3/2$ (originally, $\gamma\ge 9/5$) in dimension three. These solutions, in particular,
 have the property that for any $T<\infty$, $\rho\in L^\infty([0,T]; L^\gamma(\Bbb R^3))$. On the other hand, from \eqref{lim}, we see that solutions considered in this paper failed to obey a uniform bound in the space $L^{\frac{3(\gamma-1)}{2(r_k-1)}}(\Bbb R^3)$ on the approach to the singular time $T$:
 $$
 \rho\not\in L^\infty([0,T]; L^{\frac{3(\gamma-1)}{2(r_k-1)}}(\Bbb R^3))
 $$
with $r_k$ chosen to be close to the value $\eye(d,\ell)$ from \eqref{rdl}.\\

\subsection{Connection to the blow up for the semilinear Schr\"odinger equation}

%%%%%%%%%%%%%%%%%%%%%%%%%%%%%%%%%%%%%%%%%%%%%%%%%%
 \begin{comment}
{There has been an intense activity for the past twenty years related to the study of blow up bubbles in canonical model like the semilinear heat, wave and Schr\"odinger equations, in both critical and super critical settings, see e.g \cite{MRloglog, MRRnlsmap, RaphRodwave, RodSter, MRRnls, RSch, KST, HV, CRS,DS}. The typical model is the {\em focusing} nonlinear Schr\"odinger equation $$i\pa_tu+\Delta u+u|u|^{p-1}=0, \ \ x\in \Bbb R^d$$ in the energy super critical range $p>\frac{d+2}{d-2}$, $d\ge 3$. In some suitable regime of parameters, smooth decreasing initial data are exhibited which lead to finite time blow up solutions with a complete description of the associated singularity formation in terms of blow up rate and asymptotic profile, both in original and renormalized variables, and both in self similar and non self similar regimes. However the focusing nature of the equation is used in an essential way in the analysis, as it provides the existence of canonical blow up profiles through either soliton like solutions or self similar solutions which are the very starting point of the blow up analysis.}
\end{comment}
Somewhat surprisingly, the mechanism of singularity formation in compressible fluids exhibited in this paper turns out to
be connected with the singularity formation in defocusing super-critical Schr\"odinger equations.
In the companion paper \cite{MRRSdefoc}, we obtain the fist result on the existence of blow up solutions emerging from smooth well localized data for the energy  {\em super-critical defocusing} model
\be
\label{nlsdefoc}
\hskip -2pc({\text {NLS}})\ \ \qquad i\pa_tu+\Delta u-u|u|^{p-1}=0, \ \ x\in \Bbb R^d
\ee 
in a suitable energy super-critical range $p>p(d)$ and $d\ge 5$. Neither soliton solutions nor self-similar solutions are known for \eqref{nlsdefoc}, but we rely on a third blow up scenario, well known for the focusing non-linear heat equation, 
see e.g. \cite{BK,MZduke} and in more recent \cite{MRSheat, CMRcylindrique}: the front scenario. After passing to the hydrodynamical variables, which for (NLS) are the phase and modulus, the front renormalization maps  \eqref{nlsdefoc}
to leading order onto the compressible Euler flow \eqref{eulercomp} with the behavior at infinity given by \eqref{vneivenoenenevnove}. The analysis then follows three canonical steps. These steps run in parallel to the treatment of the Navier-Stokes equations in 
this paper, which is also approximated by the Euler dynamics. The description below applies to both.
\subsection{Strategy of the proof}
\subsubsection{Self-similar Euler profiles}
\label{veniovninenein}
We first derive the leading order blow up profile which corresponds here to self-similar solutions of \eqref{eulercomp}. Continuums of such solutions have been known since the pioneering 
works of Guderley \cite{guderley} and Sedov \cite{sedov}.
However, the rich amount of literature produced since then  is concerned with {\it non-smooth} self-similar 
solutions. This is partly due to the physical motivations, e.g. interests in solutions modeling implosion or detonation waves, where 
self-similar rarefaction or compression is followed by a shock wave (these are self-similar solutions which contain shock discontinuities 
already present in the data), and, partly due to the fact that, as it turns out, global solutions with the desired behavior at infinity and at the
center of symmetry are {\it generically} not $\mathcal C^\infty$. This appears to be a fundamental feature of the self-similar Euler dynamics
and, in the language of underlying acoustic geometry, means that {\it generically} such solutions are not smooth across the 
backward light (acoustic) cone with the vertex at the singularity.

The key of our analysis is the construction of those non-generic $\matchal C^\infty$ solutions and the discovery that regularity is {\em an essential} element in controlling suitable {\it repulsivity} properties of the associated linearized operator. 
This is at the heart of the control of the full blow up. In our companion paper \cite{MRRSprofile} we 
 construct a family of $\mathcal C^\infty$ spherically symmetric self-similar solutions to the compressible Euler equations with suitable behavior at infinity and at the center of symmetry for {\em discrete values of the blow up speed parameter $r$  in the vicinity of the limiting blow up speed $r_{\eye}(d,\ell)$ given by \eqref{assumptionnonlinearity}}.

\subsubsection{Linearized stability}
The second step is to understand how $\mathcal C^\infty$ regularity of the blow up profile is essential to control the associated linearized operator for the Euler problem \eqref{eulercomp} in renormalized variables. Here the problem is treated  as a quasilinear wave equation and we rely on spectral and energy methods to derive the local {\it linearized} asymptotic stability of the blow up profile. The local aspect of the analysis is manifest in the fact that it is only carried out in the region 
which includes, but only barely, the interior  of the backward acoustic cone (associated with the profile) emanating from the singular point. The statement of linear stability holds for a finite co-dimension subspace of initial data. This is ultimately responsible for the assertion that results of Theorem \ref{thmmain} hold for a finite co-dimensional manifold of the moduli space of initial data.
Full details of this analysis are given in \cite{MRRSdefoc}.

\subsubsection{Nonlinear stability}
The final step of our analysis is the proof of global nonlinear stability. Here, the details of the treatment of (NLS) and (NS)
are different. However,  one unifying feature is the {\em dominance} of the Eulerian regime. 
For Navier-Stokes it means that,
{\em in a suitable regime of parameters}, the dissipative term  involving the  Laplace operator $\Delta$ is treated perturbatively  all the way to the blow up time. The reason for this is that the {\it renormalized equations} take the form
(cf. \eqref{renormalizedflow})
\be
\label{renormalizedflow'}
\left|\begin{array}{l}
\pa_\tau \rho_T=-\rho_T\div u_T-\frac{\ell(r-1)}{2}\rho_T-(2u_T+ Z)\cdot\nabla\rho_T\\
 \rho_T^2\pa_\tau u_T=b^2(\mu\Delta u_T+\mu'\nabla\div u_T)-\left[2 u_T\cdot\nabla u_T+(r-2+d)u_T+Z\cdot\nabla u_T\right]\rho_T^2+\nabla\pi.
\end{array}\right.
\ee
Here, $\rho_T$ corresponds to the square root of the density. The blow up time corresponds to $\tau\to \infty$ 
and the point is that the renormalized viscosity factor is given by $b^2\sim e^{-2\mathcal e\tau}$ with the parameter 
\be\label{eq:e'}
\mathcal e=\frac{(1+\gamma)r-2\gamma}{2(\gamma-1)}=\frac12[\ell(r-1)+r-2].
\ee
The positivity of $\mathcal e$ for $r$ close enough to $r_{\rm eye}$, which makes the dissipative term decay as $\tau\to\infty$,  
is precisely the restriction on the upper bound for $\gamma$: 
$\gamma<(2+\sqrt 3)/\sqrt 3$. 

For the Schr\"odinger equations, similar but more subtle (not all the terms involving the original $\Delta$ disappear)  
considerations lead to the restrictions on the range
of the power $p$.

The key to our claim that the results hold uniformly in viscosity and apply directly to the Euler equations is that {\it all}
of our estimates hold uniformly in viscosity. In fact, we exploit the dissipative term exactly once, in Lemma \ref{lem:NS}, 
but it is then 
used to control {\it only} the dissipative term itself.\\

We should finally mention that the methods used in both this paper and \cite{MRRSdefoc} are 
deeply connected with the analysis developed 
in our earlier work,   in particular in \cite{MRRnls}.\\

We will give the proof of Theorem \ref{thmmain} explicitly in the case of (NS) only. The Euler case follows verbatim the same path, is strictly simpler, and the condition $\ell>\sqrt{3}$ will not appear there as it measures only the compatibility of (NS) with (Euler). We will introduce a dimension parameter $d$. This is not to concern ourselves here with 
the higher dimensional Navier-Stokes (even though a certain range of $\gamma$ is available) but rather 
to facilitate considerations of the two dimensional Euler problem. As will be clear from the proof, the parameter $d$ enters meaningfully {\it only} in the treatment of the dissipative term.

\subsection{Organization}
In section \ref{frontrenorm}, we introduce the front renormalization and recall the main results of \cite{MRRSprofile} concerning the existence of $\mathcal C^\infty$ self-similar profiles to the compressible Euler equations. In section \ref{sectionlinear}, we recall the main decay estimates for the associated linearized operator.  Their detailed proofs are contained in \cite{MRRSdefoc}. In section \ref{sectionbootstrap}, we describe our set of initial data and detail the bootstrap bounds needed for our analysis. In section 5, we derive some non-renormalized estimates which are used to control the exterior region $|x|\ge 1$. 
In section 6, we derive a general quasilinear 
 energy estimate at the highest level of regularity. 
 In section 7, we use its {\it unweighted} version to close the bounds for the highest derivative in the $d=3, \ell>\sqrt 3$ 
 case. In section 8, we repeat the argument but this time with a combination of cut-off functions, to close the bounds for the highest derivative in the remaining Euler cases $d=2$ and $d=3, \ell<\sqrt 3$.
 In section 9, we derive and close {\it weighted} energy bounds for all sufficiently high derivatives. Sections 5-9 will allow us to close the pointwise bounds on the 
solution. In section \ref{low} we upgrade the linear estimates of section \ref{sectionlinear} to nonlinear ones and 
propagate them to any compact set in the renormalized variable $Z$ relative to which the acoustic cone terminating in 
a singular point corresponds to the equation $Z=Z_2$.Theorem \ref{thmmain} then follows from a now standard Brouwer like topological argument.\\

\subsection*{Constants and notations} Below we list constants, relations and conventions used throughout the text.\\
\noindent -- Parameters $p$ and $\gamma$ from the equation of state $\pi=\frac{\gamma-1}\gamma \rho^{\gamma}$
\be
\label{definitioncontstas1}
p-1=2(\gamma-1).
\ee
\noindent -- Parameter $\ell$
\be
\ell=\frac{2}{\gamma-1}=\frac{4}{p-1}.
\ee
\noindent -- Front speed parameter $r$ which is assumed to be strictly less but arbitrarily close to one of the limiting values
\be
\eye(d,\ell)=\left|\begin{array}{l} r^*(d,\ell)=\frac{d+\ell}{\ell+\sqrt{d}}\ \ \mbox{for}\ \ \ell<d\\
r_+(d,\ell)=1+\frac{d-1}{(1+\sqrt\ell)^2}\ \ \mbox{for}\ \ \ell>d.
\end{array}\right.
\ee
with $d$ -- general dimension parameter. In particular, we will always use that
$$
r>1;
$$
\noindent -- Parameter $\mathcal e$ measuring compatibility between the Euler and Navier-Stokes
\be\label{eq:e}
\mathcal e=\frac{(1+\gamma)r-2\gamma}{2(\gamma-1)}=\frac12[\ell(r-1)+r-2].
\ee
The requirement $\mathcal e>0$ will be imposed in the Navier-Stokes case to ensure the dominance of the 
Eulerian regime. It forces the 
restriction 
\be
\ell>\ell_0(d)=\frac{2\sqrt{d}-d}{d-1-\sqrt{d}}.
\ee
\noindent -- Original variables $(t,x)$
\noindent -- Renormalized variables $(\tau,Z)$
$$
(T-t)=2e^{-r\tau},\ \ Z=e^{\tau} x.
$$
\noindent -- Original unknowns $(\rho(t,x),u(t,x))$ and the potential $\Psi=\nabla u$.\\

\noindent -- First renormalization 
$$
\hat\rho(t,x)=\left(2^{\frac 1{\gamma-1}}\rho(2t,x)\right)^{\frac 12},\ \ \hat u(t,x)= u(2t,x).
$$
\noindent -- Second renormalization
$$
\rho_T(\tau,Z)=e^{-\frac \ell 2(r-1)\tau}\hat\rho(t,x),\ \ u_T(\tau,Z)=e^{-(r-1)\tau}\hat u(t,x).
$$

\noindent -- Renormalized viscosity parameter $b^2$
\be
b^2=e^{-2\e\tau}
\ee
\noindent -- Profile in renormalized variables $(\rho_P(Z), \Psi_P(Z))$ and the corresponding pair 
$(\hat\rho_P(t,x), \hat\Psi_P(t,x))$.\\
 
\noindent -- Dampened profile in renormalized variables $(\rho_D(\tau,Z), \Psi_D(\tau,Z))$ and the corresponding pair 
$(\hat\rho_D(t,x), \hat\Psi_D(t,x))$.\\

\noindent -- Linearization variables 
$$
\tilde \rho(\tau,Z)=\rho_T(\tau,Z)-\rho_D(\tau,Z),\ \ \Psit(\tau,Z)=\Psi_T(\tau,Z)-\Psi_D(\tau,Z),
$$
and velocity $\ut=\nabla\Psit$.\\

\noindent -- Depending on context, $\nabla$ may denote either derivatives in $x$ or $Z$.
$\nabla^\alpha$ with 
$$
 \alpha=(\alpha_1,\dots,\alpha_d)\in \Bbb N^d, \ \ |\alpha|=\alpha_1+\dots+\alpha_d=k
 $$
will denote a generic $\pa^{\alpha_1}_1....\pa_d^{\alpha_d}$-
derivative of order $k$. Sometimes we will abuse the notation and write $\nabla^k$.\\

\noindent -- $\pa^k$ will denote the vector $(\pa_1^k,...,\pa_d^k)$ of k-th order derivatives.\\

\noindent -- By abuse of notation we will identify $Z$ with $|Z|$ and denote by $\pa_Z$ the radial derivative.

\subsection*{Acknowledgements}  P.R. is supported by the ERC-2014-CoG 646650 SingWave. P.R would like to thank the Universit\'e de la C\^ote d'Azur where part of this work was done for its kind hospitality. I.R. is partially supported by the NSF 
grant DMS \#1709270 and a Simons Investigator Award. J.S  is supported by the ERC grant  ERC-2016 
CoG 725589 EPGR.

%%%%%%%%%%%%%%%%%%%%%%%%%%%%%%%%%%%%%%%%%%%%%%%%%%
%%%%%%%%%%%%%%%%%%%%%%%%%%%%%%%%%%%%%%%%%%%%%%%%%%

\section{Front renormalization}
\label{frontrenorm}
%%%%%%%%%%%%%%%%%%%%%%%%%%%%%%%%%%%%%%%%%%%%%%%%%%
%%%%%%%%%%%%%%%%%%%%%%%%%%%%%%%%%%%%%%%%%%%%%%%%%%

We compute the front renormalization which allows one to treat \eqref{NScomp} as a perturbation of \eqref{eulercomp} in a suitable regime of parameters. We then recall the main facts concerning the existence of $\matchal C^\infty$ smooth decaying at infinity self similar solutions to \eqref{eulercomp} for quantized values of the blow up speed obtained in \cite{MRRSprofile}.

%%%%%%%%%%%%%%%%%%%%%%%%%%%%%%%%%%%%%%%%%%%%%%%%%%

\subsection{Equivalent flow for non vanishing data}

%%%%%%%%%%%%%%%%%%%%%%%%%%%%%%%%%%%%%%%%%%%%%%%%%%

Let us consider the flow \eqref{NScomp} for non vanishing density solutions:
$$
\left|\begin{array}{lll}\pa_t\rho+\nabla\cdot(\rho u)=0\\
\rho\pa_tu+\rho u\cdot\nabla u+\nabla \pi=\mu\Delta u+\mu'\nabla\div u\\
\pi=\frac{\gamma-1}{\gamma}\rho^\gamma\end{array}\right., \ \ x\in \Bbb R^d.
$$ 
We change variables:
\be
\label{changeofvariables}
\left|\begin{array}{l}
\rho(t,x)=\frac{1}{2^{\frac{1}{\gamma-1}}}\rhoh^2\left(\frac t2,x\right)\\
u(t,x)=\hat{u}\left(\frac t2,x\right)=\nabla \Psih \left(\frac t2,x\right)
\end{array}\right.
\ee
The first equation is logarithmic in density:
\bee
&& \frac{\pa_t\rho}{\rho}+\nabla \cdot u+\frac{\nabla \rho}{\rho}\cdot\nabla u=0 \Leftrightarrow \frac{\pa_t\rhoh}{\rhoh}+\nabla \cdot \uh+\frac{2\nabla \rhoh}{\rhoh}\cdot \uh=0\\
&\Leftrightarrow&\pa_t\rhoh+\rhoh \nabla \cdot \uh+2\nabla \rhoh\cdot \uh=0\Leftrightarrow\pa_t\rhoh+\rhoh\Delta \Psih+2\nabla\Psih\cdot\nabla \rhoh=0.
\eee
The second equation becomes:
\bee
&&\frac12\pa_t\uh-\frac{1}{\frac{\rhoh^2}{2^{\frac{1}{\gamma-1}}}}\left(\mu\Delta\uh+\mu'\nabla\div \uh\right)+\uh\cdot\nabla \uh+(\gamma-1)\rho^{\gamma-1}\frac{\nabla \rho}{\rho}=0\\
&\Leftrightarrow &\frac 12 \pa_t\uh-\frac{2^{\frac{1}{\gamma-1}}}{\rhoh^2}\left(\mu\Delta\uh+\mu'\nabla\div \uh\right)+\uh\cdot\nabla \uh+\frac{\gamma-1}{2}\rhoh^{2(\gamma-1)}\frac{2\nabla \rhoh}{\rhoh}=0\\
& \Leftrightarrow& \frac 12 \pa_t\uh-\frac{2^{\frac{1}{\gamma-1}}}{\rhoh^2}\left(\mu\Delta\uh+\mu'\nabla\div \uh\right)+\uh\cdot\nabla \uh+\frac{p-1}{2}\rhoh^{p-1}\frac{\nabla \rhoh}{\rhoh}=0
\eee
and hence the equivalent formulation:
\be
\label{eqivneinoevlijforinoaa}
\left|\begin{array}{l}
\pa_t\rhoh+\rhoh \nabla \cdot \uh+2\nabla \rhoh\cdot \uh=0\\
\pa_t \uh-\frac{\alpha}{\rhoh^2} \left(\mu\Delta\uh+\mu'\nabla\div \uh\right)+2\uh\cdot\nabla \uh+\nabla \ph=0\\
\hat\pi=\rhoh^{p-1}\\
\alpha=2^{\frac{\gamma}{\gamma-1}}
\end{array}\right.
\ee
and hence for spherically symmetric solutions:
\bee
&&\frac 12\pa_t \pa_r\Psih-\frac{{\mu\Delta \uh+\mu'\nabla\div\uh}}{\frac{1}{2^{\frac{1}{\gamma-1}}}\rhoh^2}+\frac 12\pa_r(|\uh|^2)+\pa_r(\rho^{\gamma-1})=0\\
&\Leftrightarrow& \frac 12\pa_r\left[\pa_t\Psih-\mathcal F(\uh,\rhoh)+|\pa_r\Psih|^2+\rhoh^{2(\gamma-1)}\right]=0\\
&\Leftrightarrow& \pa_t\Psih-\mathcal F(\uh,\rhoh)+|\pa_r\Psih|^2+\rhoh^{2(\gamma-1)}=\mathcal B(t)
\eee
where $\mathcal B(t)$ is the Bernoulli function and
\be
\label{defmthjacfl}
\mathcal F(u,\rho)=2^{\frac{\gamma}{\gamma-1}}(\mu+\mu') \int_0^r\frac{\left(\Delta U(r')-\frac 2{r^2}U(r')\right)}{\rhoh^2(r')}dr',
\ee 
since in this case, for $U=\pa_r\Psi$ and $u=U \vec{e}_r$,
$$
\Delta {u}=\nabla\div {u} = \left(\Delta U -\frac {2U}{r^2}\right) \vec{e}_r
$$
By changing $\Psi\mapsto \Psi +a(t)$ with $$\dot{a}=\mathcal B, \ \ a(t)=\int_0^t\mathcal B(\tau)d\tau$$ which does not change velocity, we have the equivalent flow 
\be
\label{NScompbis}
\left|\begin{array}{l}
\pa_t\rhoh+\rhoh\Delta \Psih+2\pa_r\Psih\pa_r\rhoh=0\\
\pa_t\Psih-\mathcal F(\uh,\rhoh)+|\pa_r\Psih|^2+\rhoh^{2(\gamma-1)}=0
\end{array}\right.
\ee

%%%%%%%%%%%%%%%%%%%%%%%%%%%%%%%%%%%%%%%%%%%%%%%%%%

\subsection{Front renormalization}

%%%%%%%%%%%%%%%%%%%%%%%%%%%%%%%%%%%%%%%%%%%%%%%%%%

Let us recall that compressible Euler has the two parameter symmetry transformation group $$\left|\begin{array}{l}\left(\frac{\l}{\nu}\right)^{\frac{2}{\gamma-1}}\rho(s,Z), \ \ \frac{\l}{\nu}u(s,Z)\\
Z=\frac{x}{\l}, \ \ \frac{ds}{dt}=\frac{1}{\nu}
\end{array}\right.$$
which becomes for \eqref{NScompbis}: 
$$\left|\begin{array}{l}\left(\frac{\l}{\nu}\right)^{\frac{1}{\gamma-1}}\rhoh(s,Z), \ \ \frac{\l^2}{\nu}\Psih(s,Z)\\
Z=\frac{x}{\l}, \ \ \frac{ds}{dt}=\frac{1}{\nu}.
\end{array}\right.
$$

\begin{lemma}[Renormalization]
\label{lemmarenormalization}
Let $r$ be the front speed, recall \eqref{eq:e}, and let 
\be
\label{scalinglaws}
\lambda(\tau)=e^{- \tau}, \ \ 
\nu(\tau)=e^{-r\tau}, \ \ b(\tau)=e^{-{\mathcal e}\tau}
\ee
then the renormalization
\be
\label{renormalization}
\left|\begin{array}{l}
\rhoh(t,x)=\left(\frac{\l}{\nu}\right)^{\frac{1}{\gamma-1}}\rho_T(\tau,Z)\\
\Psih(t,x)=\frac{\l^2}{\nu}(\Psi_T+a(\tau))(s,Z), \ \ u_T=\pa_Z\Psi_T\\
a(\tau)=e^{-(r-2)\tau}\\
Z=\frac{x}{\l}, \ \ \frac{d\tau}{dt}=\frac{1}{\nu}
\end{array}\right.
\ee
transforms \eqref{NScompbis} into:
\be
\label{renormalizedflow}
\left|\begin{array}{l}
\pa_\tau \rho_T=-\rho_T\Delta \Psi_T-\frac{\ell(r-1)}{2}\rho_T-(2\pa_Z\Psi_T+ Z)\pa_Z\rho_T\\
 \pa_\tau \Psi_T=b^2\mathcal F(u_T,\rho_T)-\left[(\pa_Z\Psi_T)^2+(r-2)\Psi_T+Z\pa_Z \Psi_T+\rho_T^{p-1}\right]
\end{array}\right.
\ee
with 
\be
\label{def-f}
\mathcal F(u_T,\rho_T)=2^{\frac{\gamma}{\gamma-1}}(\mu+\mu')\int_0^r\frac{\left(\Delta U_T(r')-\frac 2{r^2}U_T(r')\right)}{\rho_T^2(r')}dr'.
\ee 
\end{lemma}

\begin{proof}[Proof of Lemma \ref{lemmarenormalization}] We renormalize the first equation and obtain 
\bee
&&\pa_\tau\rho_T+\frac{r-1}{\gamma-1} \rho_T+\Lamdba \rho_T+2\pa_Z\Psi_T\pa_Z\rho_T=0\\
&\Leftrightarrow & \pa_\tau \rho_T=-\rho_T\Delta \Psi_T-\frac{\ell(r-1)}{2}\rho_T-(2\pa_Z\Psi_T+ Z)\pa_Z\rho_T.
\eee
For the second equation:
\bee
&&\pa_\tau(\Psi_T+a(\tau))+(r-2)(\Psi+a)+\Lambda \Psi_T+\left(\frac{\nu^{1+\gamma}}{\l^{2\gamma}}\right)^{\frac{1}{\gamma-1}}\mathcal F(u_T,\rho_T)+|\pa_Z\Psi_T|^2+\rho_T^{2(\gamma-1)}=0\\
&\Leftrightarrow & \pa_\tau \Psi_T=b^2\mathcal F(u_T,\rho_T)-\left[(\pa_Z\Psi_T)^2+(r-2)\Psi_T+Z\pa_Z \Psi_T+\rho_T^{p-1}\right]
\eee
with $a_\tau+(r-2)a=0$ and 
$$
b^2=\left(\frac{\nu^{1+\gamma}}{\l^{2\gamma}}\right)^{\frac{1}{\gamma-1}}=\left(e^{-\left[(1+\gamma)r-2\gamma\right]\tau}\right)^{\frac{1}{\gamma-1}}=e^{-2{\mathcal e}\tau},
$$
this is \eqref{renormalizedflow}. 
\end{proof}

We now observe  from \eqref{eq:e}
\be
\label{conditione}
{\mathcal e}>0\Leftrightarrow r>\frac{2+\ell}{1+\ell}.
\ee
and compute for $d=3$:\\
\underline{for $\ell<d$},
\bee
\nonumber &&r^*(d,\ell)=\frac{d+\ell}{\ell+\sqrt{d}}>\frac{2+\ell}{1+\ell}\Leftrightarrow (d+\ell)(1+\ell)>(2+\ell)(\ell+\sqrt{d})\\
\nonumber &\Leftrightarrow&  d+d\ell+\ell+\ell^2>2\ell+2\sqrt{d}+\ell^2+\ell\sqrt{d}\\
& \Leftrightarrow& \ell(d-1-\sqrt{d})>2\sqrt{d}-d  \Leftrightarrow \ell>\ell_0(d)=\frac{2\sqrt{d}-d}{d-1-\sqrt{d}},
\eee
which for $d=3 $ is $$\ell>\ell_0(3)=\frac{2\sqrt{3}-3}{2-\sqrt{3}}=\sqrt{3};$$ 
\underline{for $\ell>d$}, $$
r_+(d,\ell)=1+\frac{d-1}{(1+\sqrt{\ell})^2}>\frac{2+\ell}{1+\ell}\Leftrightarrow 1+\frac{2}{(1+\sqrt{\ell})^2}>\frac{2+\ell}{1+\ell}\Leftrightarrow (1-\sqrt{\ell})^2>0
$$ and thus always holds for $\ell>d=3$.
\begin{remark}
\label{rem:vis}
The requirement that $\mathcal e>0$ is equivalent to the decay (as $\tau\to\infty$) of the parameter $b^2$. This is precisely the value of 
renormalized viscosity in \eqref{renormalizedflow} and its decay signifies the dominance of the Euler dynamics on the 
approach to singularity. We therefore assume from now on and for the rest of this paper that  the parameter $\ell$ is in the range \eqref{ranegeulerns}. 
\end{remark}

The function $r^*(d,\ell)$ is a decreasing function of $\ell$. In particular, for $\ell>0$
\be\label{eq:r}
r^*(d,\ell)< \sqrt d
\ee
%%%%%%%%%%%%%%%%%%%%%%%%%%%%%%%%%%%%%%%%%%%%%%%%

\subsection{Blow up profile and Emden transform}

%%%%%%%%%%%%%%%%%%%%%%%%%%%%%%%%%%%%%%%%%%%%%%%%

A stationary solution to \eqref{renormalizedflow} for $b=0$ satisfies the self similar equation
\be
\label{profileequation}
\left|\begin{array}{l}
(\pa_Z\Psi_P)^2+\rho_P^{p-1}+(r-2)\Psi_P+\Lambda \Psi_P=0\\
\Delta \Psi_P+\frac{\ell(r-1)}{2}+(2\pa_Z\Psi_P+ Z)\frac{\pa_Z\rho_P}{\rho_P}=0
\end{array}\right.
\ee
which can be complemented by the boundary conditions  
\be
\label{boundarydata}
\Psi_P(0)=-\frac{1}{r-2}, \  \ \Psi_P'(0)=0, \ \ \rho_P(0)=1.
\ee
Following \cite{guderley}, \cite{sedov}, the Emden transform 
 \be
 \label{relationsprofileemden}
\left|\begin{array}{lll}
 \phi=\frac{1}{2}\sqrt{\ell}, \ \ p-1=\frac{4}{\ell}\\
 Q=\rho_P^{p-1}=\frac{1}{M^2}, \ \ \frac{1}{M}=\phi Z \sigma\\
 \frac{\Psi'_P}{Z}=-\frac{1}{2}w
 \end{array}\right., \ \ y=\log Z
\ee
maps \eqref{profileequation} into 
\be
\label{systemedefoc}
\left|\begin{array}{ll}
(w-1)w'+\ell \sigma\sigma'+(w^2-rw+\ell \sigma^2)=0\\
\frac{\sigma}{\ell}w'+(w-1)\sigma'+\sigma\left[w\left(\frac{d}{\ell}+1\right)-r\right]=0
\end{array}
\right.
\ee
or equivalently $$\left|\begin{array}{ll}a_1w' +b_1\sigma'+d_1=0\\
a_2 w'+b_2\sigma'+d_2=0
\end{array}\right.$$
with 
\be
\label{defvaluesboinedone}
\left|\begin{array}{ll}
a_1=w-1, \ \ b_1=\ell\sigma, \ \ d_1=w^2-rw+\ell\sigma^2\\
a_2=\frac{\sigma}{\ell}, \ \ b_2=w-1, \ \ d_2=\sigma\left[\left(1+\frac{d}{\ell}\right)w-r\right].
\end{array}\right.
\ee
 Let 
\be
\label{definitionwe}
w_e=\frac{\ell(r-1)}{d}
\ee
and the determinants
\be
\label{calculdeterm}
\left|\begin{array}{lll}
\Delta=a_1b_2-b_1a_2=(w-1)^2-\sigma^2\\
\Delta_1=-b_1d_2+b_2d_1=w(w-1)(w-r)-d(w-w_e)\sigma^2\\
\Delta_2=d_2a_1-d_1a_2=\frac{\sigma}{\ell}\left[(\ell+d-1)w^2-w(\ell+d+\ell r-r)+\ell r-\ell \sigma^2\right].
\end{array}\right.
\ee
then
$$
w'=-\frac{\Delta_1}{\Delta}, \ \ \sigma'=-\frac{\Delta_2}{\Delta}, \ \ 
\frac{dw}{d\sigma}=\frac{\Delta_1}{\Delta_2}.
$$
A solution $w=w(\sigma)$ of the above system can be found from the analysis of the  {\em phase portrait in the $(\sigma,w)$ plane}, see Figure \ref{fig:solutioncurve} and Figure \ref{fig:solutioncurvebis}. The shape of the phase portrait relies {\em in an essential way} on the polynomials $\Delta, \Delta_1, \Delta_2$ and the range of parameters $(r,d,\ell)$. In particular, it is easily seen that there is a unique solution which satisfies \eqref{boundarydata} and 
is $\mathcal C^\infty$ at $Z=0$. The key question is the behavior of this unique solution as $x\to+\infty$. In particular, this solution needs to pass through the point $P_2$, determined by the conditions
\be
\label{neovneneoneo}
\Delta(P_2)=\Delta_1(P_2)=\Delta_2(P_2).
\ee 
At $P_2$, generically (i.e., among all solutions passing through $P_2$,) solutions will  experience an unavoidable discontinuity of higher derivatives. Nonetheless, {\em for a discrete set values of the speed r}, our unique solution curve passes through $P_2$ in 
a ${\mathcal C}^\infty$ fashion.
The following structural proposition on the blow up profile is proved in the companion paper \cite{MRRSprofile}.

\begin{theorem}[Existence and asymptotics of a $\matchal C^\infty$ profile, \cite{MRRSprofile}]
\label{propexistenceprofile}
Let $d\in\{2,3\}$. There exists a (possibility empty) countable sequence $0<\ell_n$ which accumulation points can only be at $\{0,d,+\infty\}$ such that the following holds. Let 
$$\eye(d,\ell)=\left|\begin{array}{l} r^*(d,\ell)=\frac{d+\ell}{\ell+\sqrt{d}}\ \ \mbox{for}\ \ \ell<d\\
r_+(d,\ell)=1+\frac{d-1}{(1+\sqrt\ell)^2}\ \ \mbox{for}\ \ \ell>d
\end{array}\right.$$
be the limiting blow up speed.
Then there exists a sequence $(r_k)_{k\ge 1}$ with 
\be
\label{limimtirmt}
\lim_{k\to \infty} r_k=\eye(d,\ell), \ \ r_k<\eye(d,\ell)
\ee 
such that for all $k\ge 1$, the following holds:\\
\noindent{\em 1. Existence of a smooth profile at the origin}: the unique spherically symmetric solution to \eqref{profileequation} with Cauchy data at the origin \eqref{boundarydata} reaches in finite time $Z_2$ the point $P_2$.\\
\noindent{\em 2. Passing through $P_2$}: the solution passes through $P_2$ with $\mathcal C^\infty$ regularity.\\
\noindent{\em 3. Large $Z$ asymptotic}: the solution connects to the $P_6$ point with the asymptotics as $Z\to +\infty$:
\be
\label{limitprofilesbsi}
\left|\begin{array}{l}
w(Z)=\frac{c_w}{Z^r}\left(1+O\left(\frac{1}{Z^r}\right)\right)\\
\sigma(Z)=\frac{c_\sigma}{Z^r}\left(1+O\left(\frac{1}{Z^r}\right)\right)\\
\end{array}\right.
\ee
or equivalently
\be
\label{decayprofile}
\left|\begin{array}{ll}
Q(Z)=\rho_P^{p-1}(Z)=\frac{c_P^{p-1}}{Z^{2(r-1)}}\left(1+O\left(\frac{1}{Z^r}\right)\right), \\ Z\pa_Z\Psi_P(Z)=\frac{c_\Psi}{Z^{r-2}}\left(1+O\left(\frac{1}{Z^r}\right)\right)\end{array}\right.
\ee 
for some  non zero constants $c_w,c_\sigma,c_P,c_\Psi$, and similarly for all higher order derivatives.\\
\noindent{\em 4. Non vanishing}:  there holds $$\forall Z\ge 0, \ \ \rho_P>0.$$
\noindent{\em 5. Repulsivity inside the light cone}: let 
\be
\label{definitionF}
F=\sigma_P+\Lambda\sigma_P,
\ee 
then there exists $c=c(d,\ell,r)>0$ such that
\be
\label{coercivityquadrcouplinginside}
\forall 0\le Z\le Z_2, \ \ \left|\begin{array}{l}
(1-w-\Lambda w)^2-F^2>c\\
1-w- \Lambda  w-\frac{(1-w)F}{\sigma}\ge c
\end{array}\right.
\ee
\end{theorem}

The property \eqref{coercivityquadrcouplinginside} will be fundamental for the dissipativity ({\it in renormalized variables}) of the linearized flow {\em inside the light cone}\footnote{We should explain here that the cylinder $(\tau, Z=Z_2)$ corresponds to the light (null) cone of the acoustical metric associated to the solution $(\rho_P, \Psi_P)$ of the Euler equations. In original variables, this is the backward light cone
$(t, |x|=(T-t)^{\frac 1r})$ from the singular point $(T,0)$.}  $Z<Z_2$.
This is however insufficient. Dissipative term in the Navier-Stokes equations requires control of {\em global} Sobolev norms which, in turn,  demands \eqref{coercivityquadrcouplinginside} to hold globally in space.
\begin{lemma}[Repulsivity outside the light cone, \cite{MRRSprofile}]
\label{lemmnuericalvules}
Let $d=3$ and $$\ell_0(3)=\sqrt{3}<\ell,$$ then 
\be
\label{P}
(P)\ \ \ \  \exists c=c_{d,p,r}>0, \ \ \forall Z\ge Z_2, \ \ \left|\begin{array}{l} (1-w-\Lambda w)^2-F^2>c\\ 1-w-\Lambda w>c
\end{array}\right.
\ee
\end{lemma}

From now on and for the rest of this paper, we assume $$\hskip .3pc\quad\qquad\qquad\left|\begin{array}{l}d=3\\ \ell_0(3)<\ell
\end{array}\right.\qquad (Navier-Stokes)
$$ 
$$\left|\begin{array}{l}d=2,3\\ 0<\ell
\end{array}\right.\qquad (Euler)
$$ 
and pick once and for all a blow up speed $r=r_k$ close enough to $\eye(d,\ell)$ so that $(P)$ holds and $\e>0$.

\begin{figure}
\centering
\includegraphics[width=13cm]{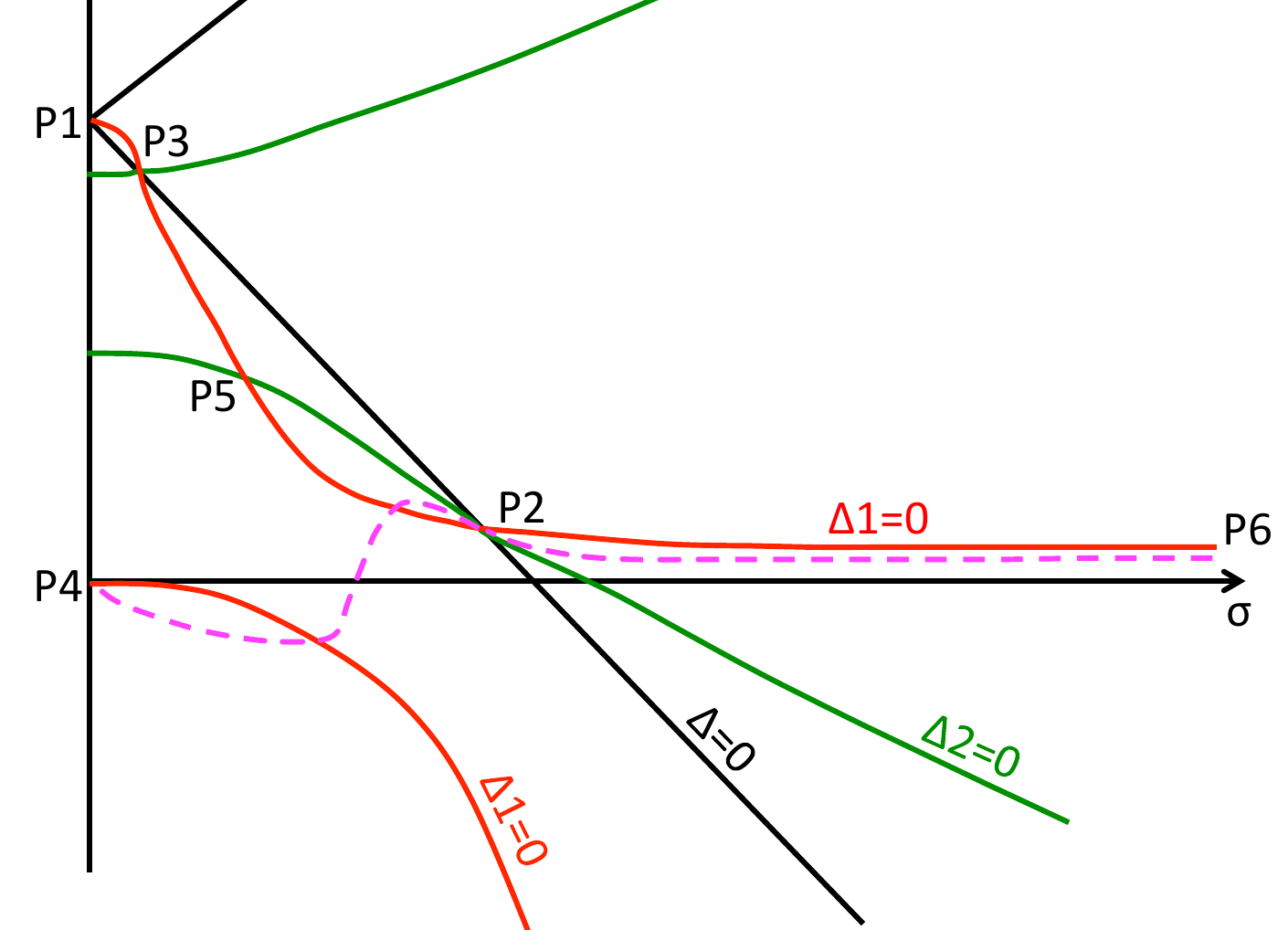}
\caption{Phase portrait in the range $1<r<r^*(d,\ell)$. Dashed curve is the trajectory of the solution constructed in Theorem \ref{propexistenceprofile}.}
\label{fig:solutioncurve}
\end{figure}

\begin{figure}
\centering
\includegraphics[width=13cm]{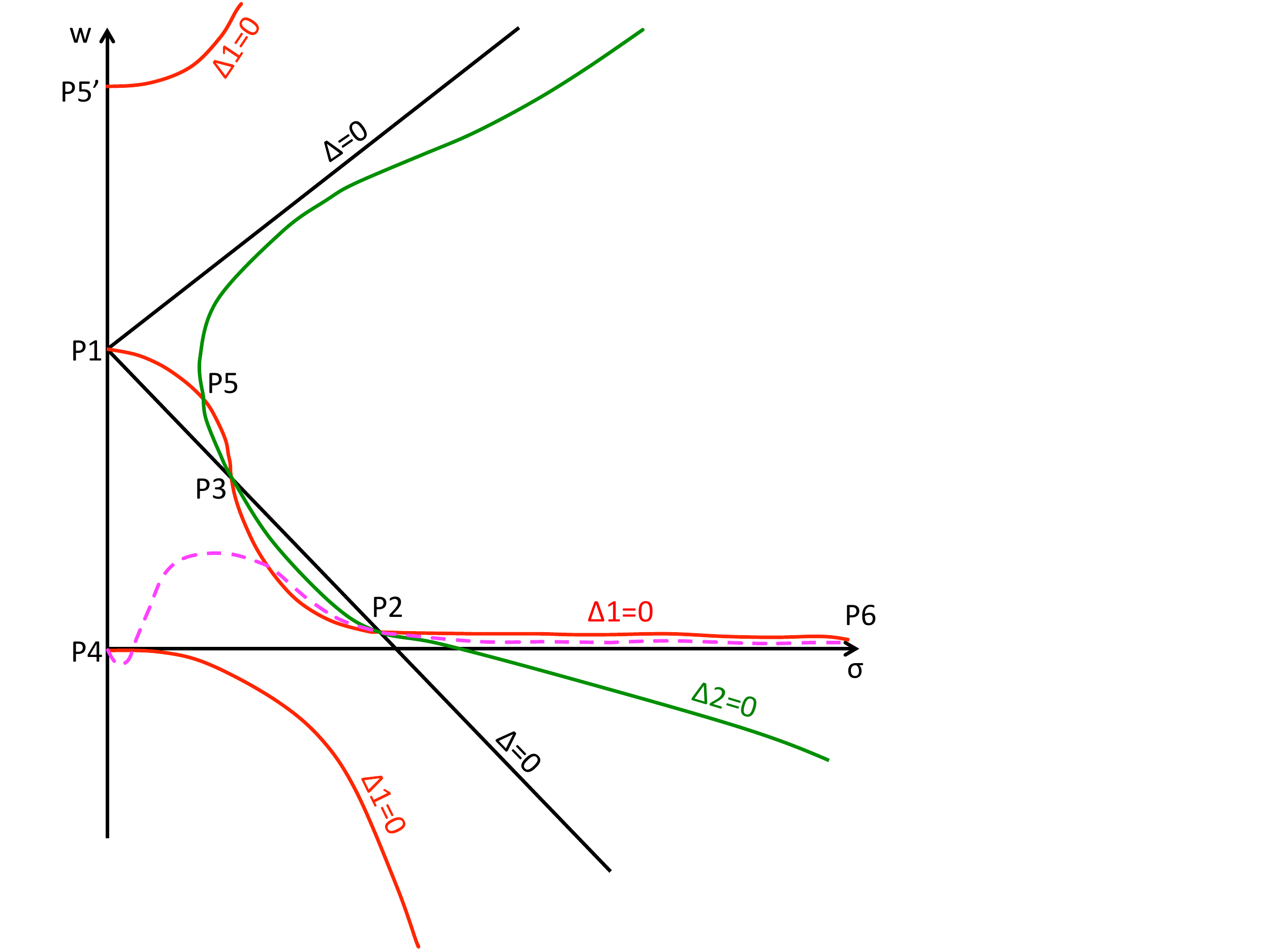}
\caption{Phase portrait in the range $r^*(d,\ell)<r<r_+(d,\ell)$, $\ell>d$. Dashed curve is the trajectory of the solution constructed in Theorem \ref{propexistenceprofile}.}
\label{fig:solutioncurvebis}
\end{figure}

%%%%%%%%%%%%%%%%%%%%%%%%%%%%%%%%%%%%%%%%%%%%%%

\subsection{Linearization of the renormalized flow}

%%%%%%%%%%%%%%%%%%%%%%%%%%%%%%%%%%%%%%%%%%%%%%

We aim at building a global in self-similar time $\tau\in[\tau_0,+\infty)$ solution to \eqref{renormalizedflow} with non vanishing density $\rho_T>0$. We define 
\be
\label{defhtwohunbis}
\left|\begin{array}{ll}
H_2=1+2\frac{\Psipx'}{Z}=1-w\\
H_1=-\left(\Delta \Psipx+\frac{\ell(r-1)}2\right)=H_2\frac{\Lambda \rhopx}{\rhopx}=\frac{\ell}{2}(1-w)\left[1+\frac{\Lambda \sigma}{\sigma}\right]
\end{array}\right.
\ee

We linearize \be\label{eq:lk}
\rho_T=\rho_P+\bar\rho, \ \ \Psi_T=\Psi_P+\bar\Psi.
\ee
 We compute, using the profile equation \eqref{profileequation}, for the first equation:
\bee
\pa_\tau \bar\rho&=&-(\rho_P+\bar\rho)\Delta(\Psi_P+\bar\Psi)-\frac{\ell(r-1)}{2}(\rho_P+\bar\rho)-(2\pa_Z\Psi_P+ Z+2\pa_Z\bar\Psi)(\pa_Z\rho_P+\pa_Z\bar\rho)\\
& = & -\rho_T\Delta \bar\Psi-2\nabla\rho_T\cdot\nabla \bar\Psi+H_1\bar\rho-H_2\Lambda \bar\rho
\eee
and for the second one:
\bee
\pa_\tau \bar\Psi&=&b^2\matchal F(u_T,\rho_T)-\Big\{|\nabla \Psi_P|^2+2\nabla \Psi_P\cdot\nabla \bar\Psi+|\nabla \bar\Psi|^2\\
&+&(r-2)\Psi_P+(r-2)\bar\Psi+(\Lambda\Psi_P+\Lambda\bar\Psi)+(\rho_P+\bar\rho)^{p-1}\Big\}\\
& = & b^2\matchal F(u_T,\rho_T)-\left\{2\nabla \Psi_P\cdot\nabla \bar\Psi+\Lambda \bar\Psi+(r-2)\bar\Psi+|\nabla\bar\Psi|^2+(\rho_P+\bar\rho)^{p-1}-\rho_P^{p-1}\right\}\\
& = & b^2\matchal F(u_T,\rho_T)-\left\{H_2\Lambda \bar\Psi+(r-2)\bar\Psi+|\nabla \bar\Psi|^2+(p-1)\rho_P^{p-2}\bar\rho +\NL(\rho)\right\}
\eee
with $$\NL(\rho)=(\rho_P+\bar\rho)^{p-1}-\rho_P^{p-1}-(p-1)\rho_P^{p-2}\bar\rho.$$
Hence the exact linearized flow
\be
\label{exactliearizedflow}
\left|\begin{array}{ll} \pa_\tau \bar\rho=H_1\bar\rho-H_2\Lambda \bar\rho-\rho_T\Delta \bar\Psi-2\nabla\rho_T\cdot\nabla \bar\Psi\\
\pa_\tau \bar\Psi=b^2\matchal F(u_T,\rho_T)-\left\{H_2\Lambda \bar\Psi+(r-2)\bar\Psi+|\nabla \bar\Psi|^2+(p-1)\rho_P^{p-2}\bar\rho +\NL(\rho)\right\}\end{array}\right.
\ee
Theorem \ref{thmmain} is therefore equivalent to constructing a finite co-dimensional manifold of smooth well localized initial data leading to global in renormalized $\tau$-time solutions to \eqref{exactliearizedflow}.

%%%%%%%%%%%%%%%%%%%%%%%%%%%%%%%%%%%%%%%
%%%%%%%%%%%%%%%%%%%%%%%%%%%%%%%%%%%%%%%

\section{Linear theory slightly beyond the light cone}
\label{sectionlinear}
%%%%%%%%%%%%%%%%%%%%%%%%%%%%%%%%%%%%%%%
%%%%%%%%%%%%%%%%%%%%%%%%%%%%%%%%%%%%%%%

Our aim in this section is to study the linearized problem \eqref{exactliearizedflow} for the exact Euler problem $b=0$. We in particular aim at setting up the suitable functional framework in order to apply classical propagator estimates which will yield exponential decay on compact sets in $Z$, modulo the control of a finite number of unstable directions. We mainly collect here the results which were proved in detail in \cite{MRRSdefoc} and apply verbatim.

%%%%%%%%%%%%%%%%%%%%%%%%%%%%%%%%%%%%%%%

\subsection{Linearized equations}

%%%%%%%%%%%%%%%%%%%%%%%%%%%%%%%%%%%%%%%

Recall the exact linearized flow \eqref{exactliearizedflow} which we rewrite:
$$
\left|\begin{array}{ll} \pa_\tau \bar\rho=H_1\bar\rho-H_2\Lambda \bar\rho-\rho_P\Delta \bar\Psi-2\nabla\rho_P\cdot\nabla \bar\Psi-\bar\rho\Delta \bar\Psi-2\nabla \bar\rho\cdot\nabla \bar\Psi\\
\pa_\tau \bar\Psi=b^2\mathcal F(u_T,\rho_T)-\left\{H_2\Lambda \bar\Psi+(r-2)\bar\Psi+(p-1)\rho_P^{p-2}\bar\rho +|\nabla \bar\Psi|^2+\NL(\rho)\right\}\end{array}\right.
$$
We introduce the new unknown 
\be
\label{defnewvariablephi}
\Phix=\rhopx \bar\Psi
\ee and obtain equivalently, using \eqref{defhtwohunbis}:
\be
\label{nekoneneon}
\left|\begin{array}{ll} \pa_\tau \bar\rhox=H_1\bar\rhox-H_2\Lambda \bar\rhox-\Delta \Phix+H_3\Phix+G_\rho\\
\pa_\tau \Phix=-(p-1)\qx\bar\rhox-H_2\Lambda \Phix+(H_1-(r-2))\Phix+G_\Phi
 \end{array}\right.
\ee
with $$\qx=\rhopx^{p-1}, \ \ H_3=\frac{\Delta\rhopx}{\rhopx}$$ and the nonlinear terms:
\be
\label{defgrho}
\left|\begin{array}{ll}G_\rho=-\bar\rho\Delta \bar\Psi-2\nabla\bar\rho\cdot\nabla \bar\Psix\\
G_\Phi=-\rho_P(|\nabla \bar\Psi|^2+\NL(\rho))+b^2\rho_P\mathcal F(u_T,\rho_T)
\end{array}\right.
\ee

 We transform \eqref{nekoneneon} into a wave equation for $\Phi$:
  \bee
&&\pa_\tau^2\Phix=  (p-1)Q\Delta \Phi-H_2^2\Lambda^2\Phi-2H_2\Lambda\pa_\tau \Phix+  A_1\Lambda \Phi+A_2\pa_\tau\Phix+A_3 \Phix\\
& + & \pa_\tau G_\Phi-\left(H_1+H_2\frac{\Lambda Q}{Q}\right)G_\Phi+H_2\Lambda G_\Phi-(p-1)QG_\rho
\eee
with
$$\left|\begin{array}{llll}
A_1=H_2H_1-H_2\Lambda H_2+H_2(H_1-(r-2))+H_2^2\frac{\Lambda Q}{Q}\\
A_2=2H_1-(r-2)+H_2\frac{\Lambda Q}{Q}\\
A_3=-(H_1-(r-2))H_1+H_2\Lambda H_1-H_2(H_1-(r-2))\frac{\Lambda Q}{Q} -  (p-1)QH_3
\end{array}\right.
$$
\begin{remark}[Null coordinates and red shift]
 We note that the principal symbol of the above wave equation is given by the second order operator
$$
\Box_Q:=\pa_\tau^2 - ((p-1)Q-H_2^2 Z^2)\pa_Z^2+2H_2Z \pa_Z\pa_\tau.
$$
This operator governs propagation of sound waves associated to the background solution $(\rho_P, \Psi_P)$ of the Euler equations.

In the variables of Emden transform $(\tau,y=\log Z)$, $\Box_Q$ can be written equivalently as 
$$
\Box_Q=\pa_\tau^2 -\left[\sigma^2-(1-w)^2\right]\pa_y^2+2(1-w)\pa_y\pa_\tau
$$
The two principal null direction associated with the above equation are 
$$
L=\pa_\tau+ \left[(1-w)-\sigma\right]\pa_y,\qquad \underline{L}=\pa_\tau+ \left[(1-w)+\sigma\right]\pa_y,
$$
so that 
$$
\Box_Q=L\underline{L}
$$
We observe that at $P_2$, we have $L=\pa_\tau$ and the surface $Z=Z_2$ is a null cone. Moreover, the associated acoustical metric\footnote{This is the metric on the $1+1$-dimensional quotient manifold obtained after removing the action of the rotation group.} is
$$
g_Q=\Delta d\tau^2 -2(1-w) d\tau dy +dy^2,\qquad \Delta=(1-w)^2-\sigma^2
$$
for which $\pa_\tau$ is a Killing field (generator of translation symmetry). Therefore, $Z=Z_2$ is a {\it Killing horizon} (generated by a null Killing field.) We can make it even more precise by transforming the metric $g_Q$ into a slightly different form by defining the coordinate $s$:
$$
s=\tau-f(y),\qquad f'=\frac {1-w}{\Delta},
$$
so that 
$$
g_Q=\Delta (ds)^2-\frac {\sigma^2}{\Delta} dy^2
$$
and then the coordinate $y^*$:
$$
y^*=\int \frac{\sigma}{\Delta} dy, 
$$
so that
$$
g_Q=\Delta\, d(s+y^*)\, d(s-y^*)
$$
and $y+x^*$ and $y-x^*$ are the null coordinates of $g_Q$.
The Killing horizon $Z=Z_2$ corresponds to $y^*=-\infty$ and $\Delta\sim e^{Cy^*}$ for some positive constant $C$.
In this form, near $Z_2$ the metric $g_Q$ resembles the $1+1$-quotient Schwarzschild metric near the black hole horizon.

The associated {\it surface gravity} $\kappa$ which can be computed according to
\begin{align*}
\kappa&=\frac {\pa_{y^*} \Delta}{2\Delta}|_{P_2} =\frac {\pa_{y} \Delta}{2\sigma}|_{P_2}=
\frac {-w'(1-w)-\sigma'\sigma)}{\sigma}|_{P_2}\\ &=(-w'-\sigma')|_{P_2}=1-w-\Lambda w -\frac{(1-w)F}{\sigma}|_{P_2}>0
\end{align*}
This is precisely the repulsive condition \eqref{coercivityquadrcouplinginside}  (at $P_2$).
The positivity of surface gravity implies the presence of the {\it red shift} effect along $Z=Z_2$ both as an optical phenomenon for the acoustical metric
$g_Q$ and also as an indicator of local monotonicity estimates for solutions of the wave equation 
$\Box_Q \varphi=0$, \cite{dr}. The complication in the analysis below is the presence of lower order terms in the wave equation as well as the need for global in space estimates.
\end{remark}

We focus now on deriving decay estimates for \eqref{nekoneneon}.

%%%%%%%%%%%%%%%%%%%%%%%%%%%%%%%%%%%%%%%

\subsection{The linearized operator with a shifted measure}

%%%%%%%%%%%%%%%%%%%%%%%%%%%%%%%%%%%%%%%

Pick a small enough parameter $$0<a\ll 1$$ and consider the new variable
\be
\label{defintionT}
\T=\pa_\tau\Phi+aH_2\Lambda \Phi,
\ee 
we compute the $(\T,\Phi)$ equation
\be
\label{newlinearflow}
\pa_\tau X=\mathcal M X +G, \ \ X=\left|\begin{array}{ll}\Phi\\ \Theta\end{array}\right., \ \ G=\left|\begin{array}{ll}0\\ G_\T\end{array}\right.
\ee
with 
\be
\label{defiiotinm}
\mathcal M=\left(\begin{array}{ll} -aH_2\Lambda & 1\\ (p-1)Q\Delta -(1-a)^2H_2^2\Lambda^2+\tilde{A_2}\Lambda +A_3& -(2-a)H_2\Lambda +A_2\end{array}\right)
\ee
where
\be
\label{defgt}
G_\T=\pa_\tau G_\Phi-\left(H_1+H_2\frac{\Lambda Q}{Q}\right)G_\Phi+H_2\Lambda G_\Phi-(p-1)QG_\rho
\ee 
and 
$$\tilde{A}_2=A_1+(2a-a^2)H_2\Lambda H_2-a A_2H_2.$$
The fine structure of the operator \eqref{defiiotinm} involves the understanding of the associated shifted light cone.

\begin{lemma}[Shifted measure, \cite{MRRSdefoc}]
\label{shiftoight}
Let 
\be
\label{defdacnoneo}
D_a=(1-a)^2(w-1)^2-\sigma^2
\ee
then for $0<a<a^*$ small enough, there exists a $\mathcal C^1$ map $a\mapsto Z_a$ with $$Z_{a=0}=Z_2, \ \ \frac{\pa Z_a}{\pa a}>0$$ such that 
\be
\label{eineneoneonoe}
\left|\begin{array}{l}
D_a(Z_a)=0\\
-D_a(Z)>0\ \ \mbox{on}\ \ 0\le Z<Z_a\\
\lim_{Z\to 0}Z^2(-D_a)>0.
\end{array}\right.
\ee
\end{lemma}

%%%%%%%%%%%%%%%%%%%%%%%%%%%%%%%%%%%%%%%

\subsection{Commuting with derivatives}

%%%%%%%%%%%%%%%%%%%%%%%%%%%%%%%%%%%%%%%

We define $$\T_k=\Delta^k\T, \ \ \Phi_k=\Delta^k\Phi$$ and commute the linearized flow with derivatives. 

\begin{lemma}[Commuting with derivatives, \cite{MRRSdefoc}]
\label{lemmaderivative}
Let $k\in \Bbb N$. There exists a smooth measure $g$ defined for $Z\in [0,Z_a]$ such that the following holds. Let 
$$\mathcal L_g\Phi_k=\frac{1}{gZ^{d-1}}\pa_Z\left(Z^{d-1}Z^2g(-D_a) \pa_Z\Phi_k\right),$$ 
then there holds 
\be
\label{commutationderivative}
\Delta^k(\mathcal MX)=\mathcal M_k\left|\begin{array}{ll}\Phi_k\\ \T_k\end{array}\right.+\widetilde{\mathcal M}_k X
\ee
with
$$
\mathcal M_k\left|\begin{array}{ll}\Phi_k\\ \T_k\end{array}\right.=\left|\begin{array}{ll}-aH_2\Lambda \Phi_k-2ak(H_2+\Lambda H_2)\Phi_k+ \T_k\\ \matchal L_g \Phi_k-(2-a)H_2\Lambda \T_k-2k(2-a)(H_2+\Lambda H_2) \T_k +A_2\T_k\end{array}\right.
$$
where $\widetilde{\mathcal M}_k$ satisfies the following pointwise bound
\be
\label{contlmak}
|\widetilde{\mathcal M}_kX|\lesssim_k \left|\begin{array}{ll} \displaystyle\sum_{j=0}^{2k-1}|\pa^{j}_Z\Phi|,\\
 \displaystyle\sum_{j=0}^{2k}|\pa^{j}_Z\Phi|+\sum_{j=0}^{2k-1}|\pa_Z^j\T|.
\end{array}\right.
\ee
Moreover, $g>0$ in $[0,Z_a)$ and admits the asymptotics:
\be
\label{asymptoticsg}
\left|\begin{array}{ll}g(Z)=1+O(Z^2)\ \ \mbox{as}\ \ Z\to 0\\
g(Z)=c_{a,d,r,\ell}(Z_a-Z)^{c_g}\left[1+O(Z-Z_a)\right]\ \mbox{as}\ \ Z\uparrow Z_a, \ \ .
\end{array}\right.
\ee
with 
\be
\label{estcg}
c_g>0
\ee for all $k\ge k_1$ large enough and $0<a<a^*$ small enough.
\end{lemma}

%%%%%%%%%%%%%%%%%%%%%%%%%%%%%%%%%%%%%%%%%%%

\subsection{Maximal accretivity and spectral gap}

%%%%%%%%%%%%%%%%%%%%%%%%%%%%%%%%%%%%%

The linear theory we use relies on the spectral structure of compact perturbations of maximal accretive operators.\\

\noindent\underline{Hilbert space}. We define the space of test functions
$$\mathcal D_0=\mathcal D_\Phi\times C_{\rm radial}^\infty([0,Z_a],\Bbb C),$$
and let ${\Bbb H}_{2k}$
 be the completion of $\mathcal D_0$ for the scalar product:
\be
\label{defscalarproduct}
 \la X,\tilde{X}\ra = \la\la\Phi, \Phi\ra\ra+(\T_k, \tilde{\T}_k)_g+\int \chi \T\overline{\tilde{\T}}Z^{d-1}dZ
\ee
where
\bea
\label{defscalarproductbis}
\la\la \Phi,\tilde{\Phi}\ra\ra&=&-(\mathcal L_g\Phi_k,\tilde{\Phi}_{k})_g+\int \chi \Phi\overline{\tilde{\Phi}}gZ^{d-1}dZ,
\eea
$$
(\T_k, \tilde{\T}_k)_g=\int \T_k\overline{\tilde{\T}_k}g Z^{d-1}dZ,
$$
$\chi$ be a smooth cut off function  supported on the set $|Z|<Z_2$ such that 
$$g\geq \frac 12\ \ \mbox{on}\ \ {\rm Supp}\chi.$$
\noindent\underline{Unbounded operator}. Following \eqref{defiiotinm} we define the operator  $$\ \ \mathcal M=\left(\begin{array}{ll} -aH_2\Lambda & 1\\ (p-1)Q\Delta -(1-a)^2H_2^2\Lambda^2+\tilde{A_2}\Lambda +A_3& -(2-a)H_2\Lambda {+A_2}\end{array}\right)$$
  with domain 
\be
\label{defintiondoamin}
D(\mathcal M)=\{X\in \Bbb H_{2k}, \ \ \mathcal M X\in \Bbb H_{2k}\}
\ee 
equipped with the domain norm. We then pick suitable directions $(X_i)_{1\leq i\leq N}\in \Bbb H_{2k}$ and consider the finite rank projection operator
$$\mathcal A=\sum_{i=1}^N\la \cdot,X_i\ra X_i.$$
The following fundamental accretivity property is proved in \cite{MRRSdefoc}.

\begin{proposition}[Maximal accretivity/dissipativity, \cite{MRRSdefoc}]
\label{propaccretif}
There exist $k_{\flat}\gg 1$ and $0<c^*,a^*\ll 1$ such that for all $k\ge k_{\flat}$,  $\forall 0<a<a^*$ small enough, there exist $N=N(k,a)$ directions $(X_i)_{1\leq i\leq N}\in \Bbb H_{2k}$ such that the modified unbounded operator 
$$\tilde{\mathcal M}:={\mathcal M -} \mathcal A$$
is {dissipative}
\be
\label{accretivityalmost}
\forall X\in \matchal D(\matchal M), \ \ \Re\la (-\mathcal {\tilde M} X,X\ra \geq c^*ak \la X,X\ra
\ee
and maximal:
\be
\label{estmaximal}
\forall R>0, \ \ \forall F\in \Bbb H_{2k}, \ \ \exists X\in \matchal D(\matchal M)\ \ \mbox{such that} \ \ (-\tilde{\matchal M}+R)X=F.
\ee
\end{proposition}

Exponential decay in time  locally in space will now follow from the following classical statement, see  \cite{engelnagel,MRRSdefoc} for a detailed proof.

\begin{lemma}[Exponential decay modulo finitely many instabilities]
 \label{elmsnjennnw}
  Let $\delta_g>0$ and let $T_0$ be the strongly continuous semigroup generated by a maximal dissipative operator 
  $\tilde {\mathcal M}+\delta_g$, and $T$ be the strongly continuous semi group generated by ${\mathcal M}=
  \tilde {\mathcal M}+ A$ where $A$ is a compact operator on $H$. 
   Then the following holds:\\
(i) the set $\Lambda_{\delta_g}(\mathcal M)=\sigma(A)\cap \{\l\in \Bbb C, \ \ \Re(\l)> -\frac{\delta_g}2\}$ is finite, 
each eigenvalue $\l\in \Lambda_{\delta_g}(\mathcal M)$ has finite algebraic multiplicity $k_\l$. In particular, the subspace 
$V_{\delta_g}(\mathcal M)$ is finite dimensional;\\
(ii) 
We have $\Lambda_{\delta_g}(\mathcal M)=\overline{ \Lambda_{\delta_g}(\mathcal M^*)}$
and $dim V_{\delta_g}(\mathcal M^*)=dim V_{\delta_g}(\mathcal M)$.
The direct sum decomposition 
\be\label{eq:direct}
H=V_{\delta_g}(\mathcal M)\bigoplus V^\perp_{\delta_g}(\mathcal M^*)
\ee
is preserved by $T(t)$ and there holds:
\be
\label{stabiliteexpo}
\forall X\in V^\perp_{\delta_g}(\mathcal M^*), \ \ \|T(t)X\|\leq M_{\delta_g} e^{-\frac {\delta_g}2 t}\|X\|.
\ee
(iii) The restriction of $A$ to $V_{\delta_g}(\mathcal M)$ is given by a direct sum of $(m_\l\times m_\l)_{\l\in \Lambda_{\delta_g}(\mathcal M)}$
matrices each of which is the Jordan block associated to the eigenvalue $\l$ and the number of Jordan blocks corresponding 
to $\l$ is equal to the geometric multiplicity of $\l$ -- $m^g_\l=dim ker (\mathcal M-\l I)$. In particular, $m^a_\l\le m^g_\l k_\l$.
Each block corresponds to an invariant subspace $J_\l$ and the semigroup $T$ restricted to $J_\l$ is given by the nilpotent 
matrix 
$$
T(t)|_{J_\l}=\begin{pmatrix} e^{\l t} & te^{\l t}&...&t^{m_\l-1} e^{\l t}\\
0& e^{\l t}&...&t^{m_\l-2} e^{\l t}\\
...\\
0&0&...& e^{\l t}
\end{pmatrix}
$$

  \end{lemma}
  Our final result in this section is  a Brouwer type argument for the evolution of unstable modes.
\begin{lemma}[Brouwer argument, \cite{MRRSdefoc}]
\label{browerset}
Let $A,\delta_g$ as in Lemma \ref{elmsnjennnw} with the decomposition 
$$
H=U\bigoplus V
$$
into stable and unstable subspaces Fix sufficiently large $t_0>0$ (dependent on $A$). 
Let $F(t)$ such that, $\forall t\ge t_0$, $F(t)\in V$ and 
$$
\|F(t)\|\le e^{-\frac {2\delta_g}3 t}
$$
 be given. Let $X(t)$ denote  the solution to the ode
$$
\left|\begin{array}{l}
\frac {dX} {dt} = A X + F(t)\\
 X(t_0)=x\in V.
 \end{array}\right. 
$$
Then, for any $x$ in the ball
$$
\|x\|\le e^{-\frac {3\delta_g} 5t_0},
$$
we have
\be\label{eq:growth}
\|X(t)\|\le e^{-\frac{\delta_g} 2 t},\qquad t_0\le t\le t_0+\Gamma
\ee
for some large constant $\Gamma$ (which only depends on $A$ and $t_0$.)
Moreover, there exists $x^*\in V$ in the same ball as a above such that $\forall t\ge t_0$,
$$
\|X(t)\|\le e^{-\frac {3\delta_g}5 t}
$$

\end{lemma}

%%%%%%%%%%%%%%%%%%%%%%%%%%%%%%%%%%%%%%%%
%%%%%%%%%%%%%%%%%%%%%%%%%%%%%%%%%%%%%%%%

\section{Setting up the bootstrap}
\label{sectionbootstrap}
%%%%%%%%%%%%%%%%%%%%%%%%%%%%%%%%%%%%%%%%
%%%%%%%%%%%%%%%%%%%%%%%%%%%%%%%%%%%%%%%%

In this section we detail the set of smooth well localized initial data which lead to the conclusions of 
Theorem \ref{thmmain}. 

%%%%%%%%%%%%%%%%%%%%%%%%%%%%%%%%%%%%%%%%

\subsection{Cauchy theory and renormalization} 

%%%%%%%%%%%%%%%%%%%%%%%%%%%%%%%%%%%%%%%%

We use local Cauchy theory for strong solutions for Navier-Stokes from \cite{choe}.

\begin{theorem}[Local Cauchy theory NS, \cite{choe}]
Assume 
\be
\label{assumptionsdata}
\left|\begin{array}{l}
\rho_0\in H^1\cap W^{1,6}\\
u_0\in \dot{H^1}\cap \dot{H}^2\\
\frac{-\Delta u_0+\nabla p_0}{\sqrt{\rho_0}}\in L^2
\end{array}\right.
\ee
then there exists a unique local strong solution $(\rho,u)\in L^\infty([0,T)\cap H^1\cap W^{1,6})\times L^\infty([0,T), \dot{H^1}\cap\dot{H}^2)$ to \eqref{NScomp}. Moreover, the maximal time of existence $T$ is characterized by the 
condition 
\bea\label{eq:blowupcriterionNScase}
\int_0^T \|\nabla u\|_{L^\infty(\Bbb R^3)}=\infty
\eea
\end{theorem}

In the Euler case we can use the results from \cite{mku} and \cite{chemin}
\begin{theorem}[Local Cauchy theory, Euler, \cite{mku,chemin}]
Assume 
\be
\label{assumptionsdatae}
\rho^{\frac{\gamma-1}2}_0, u_0\in H^s
\ee
for some $s>1+\frac d2$,
then there exists a unique local strong solution $(\rho^{\frac{\gamma-1}2},u)\in C^0([0,T)\cap H^s)$ to \eqref{eulercomp}. Moreover, the maximal time of existence $T$ is characterized by the 
condition 
\bea\label{eq:blowupcriterionEulercase}
\int_0^T \|\nabla u\|_{L^\infty(\Bbb R^d)}=\infty
\eea
\end{theorem}

On an interval $[0,T^*]$, $T^*\le T$, where $\rho(t,x)$ does not vanish,
we equivalently work with \eqref{NScompbis} and proceed to the decomposition of Lemma \ref{lemmarenormalization} 
$$\left|\begin{array}{l}
\rhoh(t,x)=\left(\frac{\l}{\nu}\right)^{\frac{1}{\gamma-1}}\rho_T(\tau,Z)\\
\uh(t,x)=\frac{\l}{\nu}u_T(s,Z), \ \ u_T=\nabla\Psi_T=U_T \vec{e}_r
\end{array}\right.
$$
with the renormalization:
\be
\label{renormalizationbis}
\left|\begin{array}{lll}
Z=y\sqrt{b}=Z^*x, \ \ Z^*=\frac{1}{\l}=e^{\tau}\\
\lambda(\tau)=e^{- \tau}, \ \ \nu(\tau)=e^{-r\tau}, \ \ b(\tau)=e^{-e\tau}\\
\tau=\frac{-\log(T-t)}{r}, \ \ \tau_0=\frac{-\log T}{r}.
\end{array}\right.
\ee
Our claim is that given $$\tau_0=\frac{-\log T}{r}$$
large enough, we can construct a finite co-dimensional manifold of smooth well localized initial data $(\rhoh_0, \uh_0)$ such that the corresponding solution to the renormalized flow \eqref{renormalizedflow} is global $\tau\in[\tau_0,+\infty)$, bounded in a suitable topology and non vanishing. Going back to the original variables yields a solution to \eqref{NScomp} which blows up at $T$ in the regime described by Theorem \ref{thmmain}.\\

%%%%%%%%%%%%%%%%%%%%%%%%%%%%%%%%%%%%%%%%

\subsection{Regularity and dampening of the profile outside the singularity}

%%%%%%%%%%%%%%%%%%%%%%%%%%%%%%%%%%%%%%%%

The profile solution $(\rho_P,\Psi_P)$ has an intrinsic slow decay as $Z\to +\infty$ forced by the self similar equation $$\rho_P(Z)=\frac{c_P}{\la Z\ra^{\frac{2(r-1)}{p-1}}}\left(1+O\left(\frac{1}{\la Z\ra^r}\right)\right)$$
which need be regularized in order to produce finite energy non vanishing initial data.\\

\noindent{\em 1. Regularity of the profile}. Recall the asymptotics \eqref{decayprofile} and the choice of parameters \eqref{scalinglaws} which show that in the original variables  $(t,x)$ both the density and the velocity profiles are regular away from the singular point $x=0$:
\bea
\label{outerprofile}
\nonumber
&&\rho_P(t,x)=\left(\frac{\l}{\nu}\right)^{\frac{1}{\gamma-1}}\rho_P\left(\frac{x}{\lambda}\right)=\frac{c_\rho e^{\frac{(r-1)}{\gamma-1}\tau}}{Z^{\frac{2(r-1)}{p-1}}}\left[1+O\left(\frac{1}{\la Z\ra^{r}}\right)\right]\\
&=& \frac{c_\rho}{|x|^{\frac{2(r-1)}{p-1}}}\left[1+O\left(\frac{1}{\la Z\ra^{r}}\right)\right]
\eea
and
\bea
\label{outervelocity}
\nonumber U_P(t,x)&=&\frac{\l}{\nu}\pa_Z\Psi_P\left(\frac{x}{\lambda}\right)=e^{(r-1)\tau}\frac{c_\Psi}{Z^{r-1}}\left[1+O\left(\frac{1}{\la Z\ra^{r}}\right)\right]\\
&=& \frac{c_{\Psi}}{x^{r-1}}\left[1+O\left(\frac{1}{\la Z\ra^{r}}\right)\right]
\eea

\noindent{\em 2. Dampening of the tail}.  The above regularity allows us to turn our profile into a finite energy (and better)
solution. We dampen the tail outside the singularity $x\ge 1$, i.e., $Z\ge Z^*$ as follows.  Let
\be
\label{defconnexion}
\mathcal K_{\rho}(x)=\left|\begin{array}{ll} 0\ \  \mbox{for}\ \ |x|\leq 5\\
n_P-\frac{2(r-1)}{p-1}\ \ \mbox{for}\ \ |x|\geq 10
\end{array}\right., 
\ee
for some large enough universal constant $$n_P=n_P(d)\gg 1.$$
We then define the dampened tail profile $\rho_D$: in the original variables
\be
\label{dampenedprofile}\hat\rho_D(t,x)=\hat\rho_P(t,x)e^{-\int_0^x\frac{\matchal K_\rho(x')}{x'}dx'}=\left|\begin{array}{ll} \rhoh_P(t,x)\ \ \mbox{for}\ \ |x|\leq 5\\  
\frac{c_{n,\delta}}{|x|^{n_P}}\left[1+O\left(e^{-r\tau})\right)\right]\ \ \mbox{for}\ \ |x|\ge 10\end{array}\right., 
\ee
and in the renormalized variables: 
\be
\label{definitionprofilewithtailchange}
\rho_D(\tau,Z)=\left(\frac{\nu}{\l}\right)^{\frac 2{p-1}} \rhoh_D(t,x), \ \ x=\frac{Z}{Z^*}.
\ee 
Let $$\zeta(x)=e^{-\int_0^x\frac{\matchal K_\rho(x')}{x'}dx'}, $$ we have the equivalent representation:
\be
\label{fromularhod}
\rho_D(\tau,Z)=(\l\sqrt{b})^{\frac 2{p-1}} \rhoh_D(\tau, x)=(\l\sqrt{b})^{\frac 2{p-1}}\rhoh_P(t,x)\zeta(x)=\zeta(\l Z)\rho_P(Z)
\ee

Note that by construction for  $j\in \Bbb N^*$:
\be
\label{fundmaetnaldecay}
-\frac{Z^j\pa^j_Z\rho_D}{\rho_D}=\left|\begin{array}{ll}\ \ (-1)^{j-1}\left(\frac{2(r-1)}{p-1}\right)^j+O\left(\frac{1}{\la Z\ra^r}\right)\ \ \mbox{for}\ \ Z\leq 5Z^*\\
(-1)^{j-1}n_P^j+O\left(\frac{1}{\la Z\ra^r}\right)\mbox{for}\ \ Z\geq 10Z^*
\end{array}\right.
\ee
and 
\be
\label{kevnkneonenoen}
 \left|\frac{\la Z\ra^j\pa_j\rho_D}{\rho_D}\right|_{L^\infty}\lesssim c_{j}.
 \ee
 We  proceed similarly for the velocity profile which can be even made compactly supported.  Let
 $$
\zeta_{u}(x)=\left|\begin{array}{ll} 1\ \  \mbox{for}\ \ |x|\leq 5\\
0\ \ \mbox{for}\ \ |x|\geq 10
\end{array}\right., 
$$
and define
\be
\label{dampenedprofilebis}
U_D(t,x)=\hat U_P(t,x)\zeta_u(x)=\left|\begin{array}{ll} \hat U_P(t,x)\ \ \mbox{for}\ \ |x|\leq 5\\  
0\ \ \mbox{for}\ \ |x|\ge 10\end{array}\right., 
\ee
 and thus in renormalized variables: 
\be
\label{definitionprofilewithtailchangespeed}
U_D(\tau,Z)=\frac{\nu}{\l} \hat U_D(t,x)=\zeta_u(\l Z)U_P(Z), \ \ x=\frac{Z}{Z^*}.
\ee 
We then let $$\Psi_D(\tau,Z)=-\frac{1}{r-2}+\int_0^ZU_D(\tau,z)dz$$ so that by construction $\Psi_D=\Psi_P$ for $Z\le 5Z^*$.

%%%%%%%%%%%%%%%%%%%%%%%%%%%%%%%%%%%%%%%%

\subsection{Initial data}

%%%%%%%%%%%%%%%%%%%%%%%%%%%%%%%%%%%%%%%%

We now describe explicitly open set of initial data which are perturbations of the profile $(\rho_D,\Psi_D)$ in a suitable topology. The conclusions of Theorem \ref{thmmain} will hold for a finite co-dimension set of such data.  Our first restriction is that the initial data $(\rho_0,u_0)$ in the original, non-renormalized variables satisfy the assumptions \eqref{assumptionsdata} and \eqref{assumptionsdatae} for the validity of the local Cauchy theory.\\

We now pick universal constants $0<a\ll 1$, $Z_0\gg 1$ which will be adjusted along the proof and depend only on $(d,\ell)$.
 We define two levels of regularity $$\frac d2\ll k_{\flat}\ll {k^\sharp}$$ where ${k^\sharp}$ denotes the maximum level of regularity required for the solution and $k_{\flat}$ is the level of regularity required for linear spectral theory on a compact set.\\

\noindent{\em 0}. Variables and notations for derivatives. We define the variables
 \be
 \label{definitionvariables}
 \left|\begin{array}{l}
 \rho_T=\rho_D+\rhot\\
  \Psi_T=\Psi_D+\Psit \\
  u_D=\nabla\Psi_D,\ \ \ut=\nabla\Psit\\
  U_D=\pa_Z\Psi_D,\ \ \tilde U=\pa_Z\Psit\\
  \Phi=\rho_P\Psi
  \end{array}\right.
 \ee
and specify the data in the $(\rhot,\Psit)$ variables. We will use the following notations for derivatives. Given $k\in \Bbb N$, we note $$\pa^k=(\pa_1^k,...,\pa_d^k), \ \ f^{(k)}:=\pa^kf$$ the vector of $k$-th derivatives in each direction. The notation $\pa_Z^k f$ is the $k$-th radial derivative. We  let $$\rhot_k=\Delta^k \rhot,\qquad \Psit_k=\Delta^k\Psi.$$
Given a multiindex $\alpha=(\alpha_1,\ldots,\alpha_d)\in \Bbb N^d$, we note $$\nabla^\alpha=\pa^{\alpha_1}_1\dots\pa^{\alpha_d}_d, \ \ |\alpha|=\alpha_1+\dots+\alpha_d.$$ Sometimes, we will use the notation 
$\nabla^k$ to denote a $\nabla^\alpha$ derivatives of order $k=|\alpha|$.
 
\noindent{\em 1. Initializing the Brouwer argument}.  We define the variables adapted to the spectral analysis according to \eqref{defnewvariablephi},  \eqref{defintionT}: 
  \be
  \label{notatinotphi}
  \left|\begin{array}{ll}
  \Phi=\rho_P\Psi \\
   T=\pa_\tau\Phi+aH_2\Lambda \Phi
   \end{array}\right., \ \ X=\left|\begin{array}{ll}\Phi\\ \T \end{array}\right.
  \ee
and recall the scalar product \eqref{defscalarproduct}. For $0<c_g,a\ll1$ small enough, we choose $k_{\flat}\gg 1$ such that Proposition \ref{propaccretif}  applies in the Hilbert space $\Bbb H_{2k_{\flat}}$
with the spectral gap 
\be
\label{accretivityalmostbis}
\forall X\in \matchal D(\matchal M), \ \ \Re\la (-\matchal M+\matchal A) X,X\ra \geq c_g\la X,X\ra.
\ee
Hence $$\mathcal M=(\mathcal M-\matchal A+c_g)-c_g+\mathcal A$$ and we may apply Lemma \ref{elmsnjennnw}: 
\be
\label{nvknneknengno}
\Lambda_0=\{\l \in \Bbb C, \ \ \Re(\l)\ge 0\} \cap \{\l\ \ \mbox{is an eigenvalue of}\ \ \mathcal M\}=(\l_i)_{1\le i\le N}
\ee 
is a finite set corresponding to unstable eigenvalues, $V$ is an associated (unstable) finite dimensional invariant set, $U$ is the complementary 
(stable) invariant set  
\be\label{eq:decomp}
  \Bbb H_{2k_{\flat}}=U\bigoplus V
  \ee
  and $P$ is the associated projection on $V$. We denote by ${\mathcal N}$ the nilpotent part of the matrix representing 
  ${\mathcal M}$ on V:
  \be\label{eq:nilp}
  {\mathcal M}|_V={\mathcal N} + {\text {diag}}
  \ee 
Then there exist $C, \delta_g>0$ such that \eqref{stabiliteexpo} holds: 
$$
\forall X\in U, \ \ \|e^{\tau\mathcal M}X\|_{\Bbb H_{2k_{\flat}}}\leq C e^{-\frac{\delta_g}{2} \tau}\|X\|_{\Bbb H_{2k_{\flat}}},\qquad \forall \tau\ge\tau_0.
$$
We now choose the data at $\tau_0$ such that 
$$
\|(I-P) X(\tau_0)\|_{\Bbb H_{2k_{\flat}}}\le e^{-\frac{\delta_g}{2} \tau_0},\qquad \|PX(\tau_0)\|_{\Bbb H_{2k_{\flat}}}\le e^{-\frac{3\delta_g}5 \tau_0}.
$$

   \noindent{\em 2. Bounds on local low Sobolev norms}. Let $0\le m\leq 2k_{\flat}$ and 
   \be
   \label{venovnoenneneo}
   \nu_0=-\frac{2(r-1)}{p-1}+\frac{\delta_g}{2},
   \ee 
   let the weight function 
  \be
 \label{defchimunu}
 \xi_{\nu_0,m}=\frac{1}{\la Z\ra^{d-2(r-1)+2(\nu_0-m)}}\zeta\left(\frac{Z}{Z^*}\right), \ \ \zeta(Z)=\left|\begin{array}{ll}1\ \ \mbox{for}\ \ Z\leq 2\\ 0\ \ \mbox{for}\ \ Z\ge 3.
 \end{array}\right.
 \ee
Then:
 \be
 \label{improvedsobolevlowinit}
 \sum_{m=0}^{2k_{\flat}}\int \xi_{\nu_0,m}\left((p-1)\rho_P^{p-1}(\nabla^m\rho(\tau_0))^2+|\nabla \nabla^m\Phi(\tau_0)|^2\right)\leq e^{-\delta_g\tau_0}.
 \ee
    
\noindent{\em 3. Pointwise assumptions}. We assume the following interior pointwise bounds 
\be
\label{smallnessoutsideinitbis}
\forall 0\leq k\leq 2k^{\sharp}, \ \ \left\|\frac{\la Z\ra^k\pa_Z^k\rhot(\tau_0)}{\rho_D}\right\|_{L^\infty(Z\le Z^*_0)}+\left\|{\la Z\ra^{r-1}}\la Z\ra^k\nabla^k\ut(\tau_0)\right\|_{L^\infty(Z\le Z^*_0)}\le \lambda_0^{c_0}
\ee
for some small enough universal constant $c_0$, and the exterior bounds:
\be
\label{smallnessoutsideinit}
\forall 0\leq k\leq 2{k^\sharp}, \ \ \left\|\frac{Z^{k+1}\pa_Z^k\rhot(\tau_0)}{\rho_D}\right\|_{L^\infty(Z\ge Z^*_0)}+\frac{\|Z^{k{+1}}\nabla^k\ut(\tau_0)\|_{L^\infty(Z\ge Z^*_0)}}{\lambda_0^{r-1}}\le \lambda_0^{C_0}
\ee
for some large enough universal $C_0(d,r,\ell)$. Note in particular that \eqref{smallnessoutsideinitbis}, \eqref{smallnessoutsideinit} ensure that for all $0<\lambda_0$ small enough:
\be
\label{smallenenoendata}
\left\|\frac{\rhot(\tau_0)}{\rho_D}\right\|_{L^\infty}\leq \mathcal d_0\ll 1
\ee and hence the data does not vanish.\\

\noindent {\em 4. Global bounds for high energy norms}.  We pick a large enough constant ${k^\sharp}(d,r,\ell)$ and consider the global energy norm 
\be
\label{globalsobolevnorm}
\|\rhot,\Psit\|_{{k^\sharp}}^2:=\sum_{j=0}^{{k^\sharp}}\sum_{|\alpha|=j}\int \frac{(p-1)\rho_D^{p-2}\rho_T(\nabla^\alpha\rhot)^2+\rho_T^2|\nabla \nabla^\alpha\Psi|^2}{\la Z\ra^{2({k^\sharp}-j)}},
\ee
then we require:
\be
\label{sobolevinit}
 \|\rhot(\tau_0),\Psit(\tau_0)\|_{{k^\sharp}}\le \mathcal d_0
\ee
We now define the weight functions 
\be\label{eight}
\chi_k= \la Z\ra^{2k-2\sigma-d+\frac{2(r-1)(p+1)}{p-1}}\left\la \frac Z{Z^*}\right\ra^{2n_P+2\sigma-\frac{2(r-1)(p+1)}{p-1}}
\ee
and the associated weighted energy norms
\be\label{aeight}
\|\rhot,\Psit\|_{m,\sigma}^2=\sum_{j=0}^m\sum_{|\alpha|=j}\int \chi_j\left[(p-1)\rho_D^{p-2}\rho_T|\nabla^\alpha\rhot|^2+\rho_T^2|\nabla \nabla^\alpha\Psit|^2\right]
\ee
We fix $0<\sigma({k^\sharp})\ll \de_g$ and require that, {for $\sigma=\sigma({k^\sharp})$,}
\be\label{sobolevinitnioeneon}
\|\rhot(\tau_0),\Psit(\tau_0)\|_{{k^\sharp},\sigma}\le \mathcal d_0 e^{-\sigma \tau_0}
\ee

\begin{remark}
$\mathcal d_0$ will denote any small constant dependent on the smallness of initial data and $\tau_0^{-1}$.
\end{remark}
\begin{remark}
We note that a straightforward integration by parts and induction argument implies that the norms $\|\rhot,\Psit\|_{m,\sigma}$
and $\|\rhot,\Psit\|_{{k^\sharp}}$ are equivalent to the ones with $\nabla^\alpha\rhot$ and $\nabla^\alpha\Psit$ replaced by 
$$
\pa^j\rhot=\{\pa_1^j,...,\pa_d^{j}\rhot\},\qquad \pa^j\Psit=\{\pa_1^{j},...,\pa_d^{j}\Psit\}
$$
 as well as 
$\Delta^j\rhot,\Delta^j\Psit$
with $j$ varying from $0$ to $\frac m2$ and $\frac {{k^\sharp}}2$ respectively (if $m$ and ${k^\sharp}$ are even.) In what follows, we will
use this equivalence continually and without mentioning. In fact, in what follows we will specifically work with the norms
$$
\|\rhot,\Psit\|_{{k^\sharp}}^2:=\sum_{j=0}^{\frac {{k^\sharp}}2}\int \frac{(p-1)\rho_D^{p-2}\rho_T(\Delta^j\rhot)^2+\rho_T^2|\nabla \Delta^j\Psi|^2}{\la Z\ra^{2({k^\sharp}-2j)}}
$$
and 
$$
\|\rhot,\Psit\|_{m,\sigma}^2=\sum_{j=0}^m\int \chi_j\left[(p-1)\rho_D^{p-2}\rho_T|\pa^j\rhot|^2+\rho_T^2|\nabla \pa^j\Psit|^2\right]
$$
\end{remark}

%%%%%%%%%%%%%%%%%%%%%%%%%%%%%%%%%%%%%

\subsection{Bootstrap bounds}

%%%%%%%%%%%%%%%%%%%%%%%%%%%%%%%%%%%%%

 Since the initial data satisfy \eqref{assumptionsdata} we have a local in time solution which can be decomposed and renormalized according to \eqref{renormalizationbis} and \eqref{definitionvariables}.
We now consider the time interval $[\tau_0,\tau^*)$ such that the following bounds hold on $[\tau_0,\tau^*)$:\\

\noindent{\em 1. Control of the unstable modes}:
Assume (see \eqref{eq:nilp}) that
\be\label{eq:unstboot}
\|e^{t{\mathcal N}} PX(\tau)\|_{\Bbb H_{2k_{\flat}}}\le e^{-\frac {19\delta_g}{30}\tau} 
\ee
\noindent {\em 2. Local decay of low Sobolev norms}: for any $0\le k\le 2k_{\flat}$, any large $\hat Z\le Z^*$ 
and universal constant $C=C(k_{\flat})$:
\be\label{eq:bootdecay}
 \|(\rhot,\Psit)\|_{H^k(Z\le \hat{Z})}\le \hat Z^C  e^{-\frac{3\delta_g}8\tau}
\ee

\noindent {\em 3. Global weighted energy bound}. We fix $0<\sigma({k^\sharp})\ll \de_g$. {For $\sigma=\sigma({k^\sharp})$,} we assume the bound:
\be
\label{boundbootbound}
 \|\rhot,\Psit\|_{{k^\sharp},\sigma}^2\leq e^{-2\sigma \tau}.
\ee

\noindent {\em 4. Pointwise bounds}: 
\be
\label{smallglobalboot}
\left|\begin{array}{l}
 0\le k\le {{k^\sharp}-2}, \ \ \left\|\frac{\la Z\ra^{k}\rhot_k}{\rho_D}\right\|_{L^\infty}+\left\|\la Z\ra ^{k}{\la Z\ra^{r-1}}\ut_k\right\|_{L^\infty(Z \le Z^*)}\leq \mathcal d\\[3mm]
  0\le k\le {{k^\sharp}-1},\ \ \left\|\la Z\ra ^{k}\la Z\ra^{r-1}\left\la\frac Z{Z^*}\right\ra^{-(r-1)}\ut_k\right\|_{L^\infty(1\le Z)}\le \mathcal d
\end{array}\right.
\ee
for some small enough universal constant $0<\mathcal d\ll 1$.

 The heart of the proof of Theorem \ref{thmmain} is the following:

\begin{proposition}[Bootstrap]
\label{propboot}
 Assume that \eqref{eq:unstboot}, \eqref{eq:bootdecay}, \eqref{boundbootbound}, \eqref{smallglobalboot}
  hold on $[\tau_0,\tau^*]$ with ${\mathcal d}^{-1},\tau_0$ large enough. Then the following holds:\\
\noindent{\em 1. Exit criterion}. The bounds  \eqref{eq:bootdecay}, \eqref{boundbootbound}, \eqref{smallglobalboot} can be strictly improved on $[\tau_0,\tau^*)$. Equivalently, $\tau^*<+\infty$ implies 
\be
\label{eibovbeiboebeo}
\|e^{t{\mathcal N}} PX(\tau^*)\|_{\Bbb H_{2k_{\flat}}} e^{\frac{19\delta_g}{30}\tau^*}=1.
 \ee
 \noindent{\em 2. Linear evolution}. The right hand side $G$ of the equation for $X(\tau)$ 
 $$
 \pa_\tau X = {\mathcal M} X + G 
 $$
 satisfies 
 \be\label{eq:Gest}
 \|G(\tau)\|_{\Bbb H_{2k_{\flat}}} \le e^{-\frac{2\delta_g}3\tau},\qquad \forall\tau\in [\tau_0,\tau^*]
 \ee
\end{proposition}
\begin{remark}
\label{weak}
We note that the assumption \eqref{eq:unstboot}
implies that 
\be\label{eq:unstX}
\|PX(\tau)\|_{\Bbb H_{2k_{\flat}}}\le e^{-\frac {\delta_g}{2}\tau}, \qquad \forall\tau\in [\tau_0,\tau^*)
\ee
We will prove the bootstrap proposition \ref{propboot} under the weaker assumption \eqref{eq:unstX}.
Specifically, we will define $[\tau_0,\tau^*]$ to be the maximal time interval on which \eqref{eq:unstX} holds 
and will show that both the bounds  \eqref{eq:bootdecay}, \eqref{boundbootbound}, \eqref{smallglobalboot} can be improved and that $G$ satisfies \eqref{eq:Gest}.
\end{remark}

An elementary application of the Brouwer topological theorem will ensure that there must exist a data such that $\tau^*=+\infty$, and these are the blow up waves of Theorem \ref{thmmain}.\\

We now focus on the proof of Proposition \ref{propboot} and work on a time interval $[\tau_0,\tau^*]$, $\tau_0<\tau^*\le+\infty$ on which  \eqref{eq:bootdecay}, \eqref{boundbootbound}, \eqref{smallglobalboot} hold.

 %%%%%%%%%%%%%%%%%%%%%%%%%%%%%%%%%%%%%%%%%%
%%%%%%%%%%%%%%%%%%%%%%%%%%%%%%%%%%%%%%%%

 %%%%%%%%%%%%%%%%%%%%%%%%%%%%%%%%%%%%%%%%%%
%%%%%%%%%%%%%%%%%%%%%%%%%%%%%%%%%%%%%%%%

\section{Global non-renormalized estimate}

  %%%%%%%%%%%%%%%%%%%%%%%%%%%%%%%%%%%%%%%%
  
  Recall the original (NS) equations \eqref{eqivneinoevlijforinoaa} (written for the square root of the density):
\be\label{eq:non}
\left|\begin{array}{l}
\pa_t\rhoh+\rhoh \nabla \cdot \uh+2\nabla \rhoh\cdot \uh=0\\
\rhoh^2\pa_t \uh-\alpha \left(\mu\Delta\uh+\mu'\nabla\div \uh\right)+2\rhoh^2\uh\cdot\nabla \uh+(p-1)\rhoh^p\nabla \rhoh=0\\
\ph=\rhoh^{p-1}\\
\end{array}\right.
\ee
The standard energy estimate for the above equation takes the form 
$$
 \frac d{dt} \int \left(\frac{1}{p+1} \rhoh^{p+1}+\frac{1}{2}\rhoh^2 |\uh|^2\right)+\alpha \left(\int \mu|\nabla\uh|^2+\mu'|\div\uh|^2\right)=0.
$$
In view of the assumptions on initial data, consistent with rapid vanishing of the dampened profile density 
$\hat\rho_D\sim x^{-n_P}$, this estimate and its higher derivative versions provide very weak control of solutions for 
large $x$. To gain such control we use an auxiliary estimate. First, we once again observe that for spherically symmetric
solutions $\Delta \hat u=\nabla\div \hat u$.

\begin{lemma}[Velocity dissipation]
\label{lemmainitialization}
There following inequality holds for any $t\in (0,T)$
\bea
\label{intiializetion}
\nonumber&&\int|\nabla \uh(t,\cdot)|^2+(\mu+\mu')\alpha \int_{0}^t\int\frac{|\Delta \uh|^2}{\rhoh^2}\\
&\lesssim& \int_{0}^t\int \Big[|\uh||\nabla \uh||\Delta\uh|+|\nabla \rhoh^{p-1}| |\Delta\uh|\Big]+\int|\nabla u(0,\cdot)|^2.
\eea
\end{lemma}

The main feature of the above estimate is the second term on the left hand side generated by the dissipative term in the Navier-Stokes equations. With the density in the denominator, it provides very strong control on the velocity at infinity.

\begin{proof}
Recall \eqref{eqivneinoevlijforinoaa}:
Dividing by $\rhoh^2$ and multiplying the second equation in \eqref{eq:non} by $\Delta\uh$ we compute:
\bee
&&\frac{1}{2}\frac{d}{dt}\int |\nabla \uh|^2+\alpha (\mu+\mu') \int \frac{|\Delta \uh|^2}{\rhoh^2}=2\int\uh\cdot\nabla \uh\cdot\Delta\uh+\int(p-1)\rhoh^{p-2}\nabla \rhoh\cdot \Delta\uh
\eee
which concludes the proof of \eqref{intiializetion}.\\
\end{proof}

We now reinterpret this estimate in the renormalized variables and show the boundedness of the right hand side.
Recall that 
 $$
\left|\begin{array}{l}
\rhoh(t,x)=\left(\frac{\l}{\nu}\right)^{\frac2{p-1}}\rho_T(\tau,Z)\\
\uh(t,x)=\frac{\l}{\nu}u_T(\tau,Z), \ \  u_T=\nabla\Psi_T
\end{array}\right.
$$
and
$$ \left|\begin{array}{l} \nu_t=\frac{\nu_\tau}{\nu}=-r, \ \ \nu=(T-t)=e^{-r\tau}\\
\l(\tau)=e^{-\tau}=(T-t)^{\frac{1}{r}},\qquad Z^*=e^\tau, \qquad b^2=(Z^*)^{-\ell(r-1)-r+2}.
\end{array}\right.
$$
Then 
\bee
(\mu+\mu')\int_{0}^T\int \frac{|\Delta \uh|^2}{\rhoh^2}&=&(\mu+\mu')\int_{\tau_0}^\infty \int (Z^*)^{-d-r+4+2(r-1)-\ell(r-1)}
\frac{|\Delta u_T|^2}{\rho_T^2}\\&=&(\mu+\mu')\int_{\tau_0}^\infty b^2 (Z^*)^{-d+2r}\int 
\frac{\left|\Delta u_T\right|^2}{\rho_T^2}
\eee
and
\bee
&&
\int_{0}^T\int \left(|\uh||\nabla \uh||\Delta\uh|+|\nabla \rhoh^{p-1}||\Delta \uh|\right)\\
&=&
\int_{\tau_0}^\infty \int (Z^*)^{-d-r+3+3(r-1)}  \left(|u_T||\nabla u_T||\Delta u_T|+|\nabla \rho_T^{p-1}||\Delta u_T|\right).
\eee
We observe that by the definition of the weight function $\chi_0$ in \eqref{eight}
$$
\la Z\ra^{-d {+} 2(r-1)-2\sigma}\left \la \frac Z{Z^*}\right \ra^{2\sigma -2(r-1)}\lesssim \chi_0 \rho_D^2
$$
We now use the pointwise bootstrap estimates \eqref{smallglobalboot}, which hold for both $\rhot, \ut$ and $\rho_T, u_T$ to estimate{, recalling also that we are in the case $d=3$,}
\bee
&&\int_{0}^T\int\left(|\uh||\nabla \uh||\Delta\uh|+|\nabla \rhoh^{p-1}||\Delta \uh|\right)\\
&\lesssim& \int_{\tau_0}^\infty (Z^*)^{-d-r+3+3(r-1)}\\
&&\times\int \la Z\ra^{-3-{(r-1)}} \left \la \frac Z{Z^*}\right \ra^{{r-1}} 
\left (|u_T| |\la Z\ra \nabla u_T|+ \rho_T^{(p-3)/2} |\la Z\ra \nabla\rho_T| \rho_T^{(p-1)/2}\right)\\
&\lesssim& \int_{\tau_0}^\infty (Z^*)^{-d-r+3+3(r-1)+2\sigma }\\
&&\times\int \chi_0\rho_D^2{\la Z\ra^{-3(r-1)}}\left \la \frac Z{Z^*}\right \ra^{{3}(r-1)} 
\left (|u_T| |\la Z\ra \nabla u_T|+ \rho_T^{(p-3)/2} |{\la Z\ra}\nabla\rho_T| \rho_T^{(p-1)/2}\right)\\
&\lesssim& \int_{\tau_0}^\infty (Z^*)^{-d-r+3+3(r-1)+2\sigma } \|\rho_T,\Psi_T\|_{1,\sigma}^2
\eee
In the last inequality we used {the fact that $r>1$ and} the definition of the norm $\|\rho_T,\Psi_T\|_{1,\sigma}$ from \eqref{aeight}.
We now note that \eqref{boundbootbound}{, together with the decay of the profile $(\rho_D,u_D)$ and the fact that $\sigma>0$,} implies that $\|\rho_T,\Psi_T\|_{1,\sigma}\leq \mathcal D$
with some constant $\mathcal D$ which depends on the size of the profile $(\rho_D,u_D)$. 
As a consequence,
$$
\int_{0}^T\int\left(|\uh||\nabla \uh||\Delta\uh|+|\nabla \rhoh^{p-1}||\Delta \uh|\right)\lesssim \mathcal D\int_{\tau_0}^\infty (Z^*)^{-d-r+3+3(r-1)+2\sigma }
$$
As long as the constant $\sigma$ is chosen to be sufficiently small, the convergence of the above integral is guaranteed
under the condition 
$$
r<\frac d2
$$
Since $r$ is either close to $r^*(d,\ell)$ or $r_+(d,\ell)$  and $r^*<r_+$, we need
$$
r_+=1+\frac {d-1}{(1+\sqrt \ell)^2}<\frac{d}{2}.
$$
For $d=3$ this holds for $\ell>1$,
which is satisfied  in view of the condition $\ell>\sqrt 3$. {Hence, we have obtained
\bee
\int_{0}^T\int\left(|\uh||\nabla \uh||\Delta\uh|+|\nabla \rhoh^{p-1}||\Delta \uh|\right) &\lesssim& \mathcal D.
\eee}
Furthermore, since the initial data, since $\uh$ at $t={0}$ is assumed to be in $\dot H^1$, {we infer, in view of Lemma \ref{lemmainitialization},
\bee
(\mu+\mu') \int_{0}^T\int\frac{|\Delta \uh|^2}{\rhoh^2} &\lesssim& \mathcal D
\eee
and hence
\bee
(\mu+\mu')\int_{\tau_0}^\infty b^2 (Z^*)^{-d+2r}\int 
\frac{\left|\Delta u_T\right|^2}{\rho_T^2} &\lesssim& \mathcal D
\eee}

We now observe that by bootstrap assumptions \eqref{smallglobalboot}.
\bee
\int_{\tau_0}^\infty b^2 (Z^*)^{-d+2r}\int_{Z\le Z^*} &&
\frac{|u_T|^2+Z^4\left|\Delta u_T\right|^2}{\la Z\ra^4\rho_T^2}\\&&\lesssim\int_{\tau_0}^\infty (Z^*)^{-\ell(r-1)-r+2-d+2r}\int_{Z\le Z^*} 
 \langle Z\rangle^{-4-2(r-1)+\ell(r-1) +d-1}\\ &&\lesssim \int_{\tau_0}^\infty \max\{(Z^*)^{-r},
 (Z^*)^{-\ell(r-1)-r+2-d+2r}\}\lesssim 1
\eee
as long as $r>0$ and 
$$
\ell(r-1)+r-2+d-2r>0
$$
For $r=r^*(d,\ell)$
\bee
\ell(r-1)+r-2+d-2r&=&\ell \frac{d-\sqrt d}{\ell+\sqrt d} -2+d-\frac{\ell+d}{\ell+\sqrt d}\\ &=&\frac {\ell(d-\sqrt d)-\ell-d-2\ell
-2\sqrt d +\ell d+d\sqrt d}{\ell+\sqrt d}\\&=&\frac {\ell(d-\sqrt d)-d
+\ell (d-3)+(d-2)\sqrt d}{\ell+\sqrt d}.
\eee
For $d=3$, this is equivalent to $\ell>1$, which holds. On the other hand, the value of $r_+(d,\ell)$
for $d=3$ and  $\ell>3$
$$
r_+(d,\ell)=1+\frac {d-1}{(1+\sqrt\ell)^2}<\frac d2,
$$
which, in view of the condition $\ell(r-1)+r-2>0$, gives us the desired conclusion. So, for $Z\le Z^*$ we control
both $u_T$ and $\Delta u_T$.

In the region $Z>Z^*$, we note again that
$$
\Delta u_T=\left(\Delta U_T-\frac 2{Z^2} U_T\right) \vec{e_r}.
$$
In addition,
$$
\left(\Delta U_T-\frac 2{Z^2} U_T\right) = \pa_Z\left(\frac 1{Z^2} \pa_Z( Z^2 U_T)\right) 
$$
Using the bootstrap assumption 
$$
\frac 12\le\frac {\rho_T}{\rho_D}\le 2
$$
with 
$$
c\left\langle\frac Z{Z^*}\right\rangle^{-n_p+\frac {2(r-1)}{p-1}}\langle Z\rangle^{-\frac {2(r-1)}{p-1}}\le\rho_D\le C\left\langle\frac Z{Z^*}\right\rangle^{-n_p+\frac {2(r-1)}{p-1}}\langle Z\rangle^{-\frac {2(r-1)}{p-1}}
$$
we can apply a Hardy type inequality (twice) in the region $Z\ge Z^*$, to arrive at the following global dissipative estimate in renormalized variables:
\begin{lemma}\label{lem:NS}
\be\label{eq:globald}
(\mu+\mu')\int_{\tau_0}^\infty b^2 (Z^*)^{-d+2r}\int 
\frac{|u_T|^2+|Z^2\Delta u_T|^2}{\la Z\ra^{4}\rho_T^2}\le \mathcal D,
\ee
where $\mathcal D$ is a constant dependent only on the (full, i.e., including the profile) initial data.

\end{lemma}
\begin{remark}
The inequality \eqref{eq:globald} is used in the treatment of the Navier-Stokes case {\it only}. As a result, the same
applies to the dimensional 
calculations appearing in its proof. 
\end{remark}
We also use the opportunity to translate our bootstrap assumptions back to the original variables.
Below we will include estimates which apply to the full solution $\uh, \rhoh$ rather than the full solution minus the 
profile and only in the exterior region $|x|\ge 10$.

\noindent {\em 1. Exterior weighted Sobolev bounds}. \eqref{boundbootbound} translates into the following bounds for the velocity
$\uh$: $\forall 0\le k\le {k^\sharp}$
\be\label{eq:xuh}
\int_{|x|\ge 10} \la x\ra^{-d+2k}|\nabla^k\uh|^2\lesssim 1 
\ee
and the density $\rhoh$: $\forall 0\le k\le {k^\sharp}$
\be\label{eq:xrhoh}
\int_{10\le |x|\le 12} |\nabla^k\rhoh|^2\lesssim 1.
\ee

\noindent {\em 2. Exterior pointwise bounds}. \eqref{smallglobalboot} translates into the following bounds
$\uh$: $\forall 0\le k\le \frac {{k^\sharp}}2$
\be\label{eq:xuhrho}
\left\|\frac {\la x\ra^k \nabla^k\rhoh}{\rhoh_D}\right\|_{L^\infty(|x|\ge 10)} + \|{\la x\ra^{k} \nabla^k\uh}\|_{L^\infty(|x|\ge 10)}\lesssim 1 
\ee
We now derive {\it improved}, relative to the bootstrap assumptions, exterior weighted   Sobolev and pointwise bounds for 
the density $\rhot$. We let 
$$
\rho_I(x)=\frac{c_{n,\delta}}{|x|^{n_P}}
$$
denote the $t$-independent leading order term in $\rho_D$, so that according to \eqref{dampenedprofile}
$$
\left|\frac{\rho_I-\rho_D}{\rho_D}\right|\lesssim e^{-r\tau}, \ \ {|x|\geq 10},
$$
with the similar inequalities also holding for derivatives. In particular,  \eqref{eq:xuhrho}
holds with $\rho_I$ in place of $\rho_D$.

Let $\zeta(x)$ be a smooth function vanishing for $|x|\le 10$ such that 
\be\label{eq:ze}
\zeta(x) \lesssim \la x\ra |\nabla\zeta(x)|\lesssim \zeta(x) + {\bf 1}_{10\le |x|\le 12}
\ee
and $\nabla^\alpha$ denote a generic $x$-derivative of order $|\alpha|\le {k^\sharp}-1$.
Applying $\nabla^\alpha$ to the first equation of \eqref{eq:non}
\be\label{eq:drt}
\pa_t\nabla^\alpha\rhoh=-\sum_{\beta+\gamma=\alpha} \nabla^\beta \rhoh \nabla\nabla^\gamma \uh-
2\nabla \nabla^\beta \rhoh \nabla^\gamma \uh,
\ee
multiplying by 
$\zeta^2(x) \la x\ra^{2|\alpha|}\frac{\nabla^\alpha \rhoh}{\rhoh_I^2}$ 
and integrating we easily derive 
\bea\label{eq:rho}
&&\frac d{dt} \int \zeta^2\la x\ra ^{2|\alpha|} \left|\frac{\nabla^\alpha\rhoh}{\rhoh_I}\right|^2 \le \left(\int \zeta^2 \la x\ra^{2|\alpha|-1}\left|\frac{\nabla^\alpha\rhoh}{\rhoh_I}\right|^2\right)^{\frac 12}\sum_{|\beta|+|\gamma|= |\alpha|+1, |\beta|\le \frac {{k^\sharp}}2}\left\| \la x\ra^{|\beta|} \frac{\nabla^\beta\rhoh}{\rhoh_I}\right\|_{L^\infty(|x|\ge 10)}\notag\\
&&\times
\left(\int \zeta^2 \la x\ra^{2|\gamma|-1} |\nabla^\gamma\uh|^2\right)^{{\frac{1}{2}}} + \left(\int \zeta^2 \la x\ra^{2|\alpha|-1}(x) \left|\frac{\nabla^\alpha\rhoh}{\rhoh_I}\right|^2\right)^{\frac 12}\notag\\  &&\times \sum_{|\beta|+|\gamma|= |\alpha|+1, {1\leq}|\gamma|\le \frac {{k^\sharp}}2}\left(\int \zeta^2\la x\ra^{2|\beta|-1} \left|\frac{\nabla^\beta\rhoh}{\rhoh_I}\right|^2\right)^{\frac 12}\left\|\la x\ra^{{|\gamma|}}\nabla^\gamma\uh\right\|_{L^\infty(|x|\ge 10)}
\notag\\
&&+ {\left(\int \zeta^2 \la x\ra^{2|\alpha|-1}(x) \left|\frac{\nabla^\alpha\rhoh}{\rhoh_I}\right|^2\right)\left\|\uh\right\|_{L^\infty(|x|\ge 10)}}\notag\\
&&+ \left(\int \zeta^2 \la x\ra^{2|\alpha|-1}(x) \left|\frac{\nabla^\alpha\rhoh}{\rhoh_I}\right|^2\right)^{\frac 12}\left(\int (\nabla\zeta)^2\la x\ra^{2|\alpha|+1} \left|\frac{\nabla^\alpha\rhoh}{\rhoh_I}\right|^2\right)^{\frac 12}\left\|\uh\right\|_{L^\infty(|x|\ge 10)}
\eea
{where the last two terms on the right hand side come from the integration by parts of $\nabla^\alpha\rhoh\nabla\nabla^\alpha\rhoh$,} and where while integrating by parts we used the bound
$$
{\la x\ra\frac{|\nabla\rhoh_I|}{\rhoh_I}\lesssim 1}.
$$
We now examine our pointwise and integrated bootstrap assumptions \eqref{eq:xuh}, \eqref{eq:xrhoh}, \eqref{eq:xuhrho} to see that we can choose
$\zeta$ to be a smooth function supported in $|x|\ge 10$ and for large $x$ behaving like
$$
\zeta^2(x)\sim \la x\ra^{-d+2(r-1)},
$$
but with this choice, after time integration, the initial data would be an infinite integral. Therefore, we  
first integrate the above differential inequality with 
$$
\zeta^2(x)\sim \la x\ra^{-d-2\sigma}
$$
{for large $x$ and } for some $\sigma>0$ to obtain that 
$$
\int_{|x|\ge 12}\la x\ra ^{-d-2\sigma+2|\alpha|} \left|\frac{\nabla^\alpha\rhoh}{\rhoh_I}\right|^2(t)
\le \mathcal D
$$
with a constant $\mathcal D$ depending on the full profile.
We now rewrite \eqref{eq:drt} by subtracting $\nabla^\alpha\rhoh_I${, and by noticing that $\pr_t\rhoh_I=0$,}
$$
\pa_t(\nabla^\alpha(\rhoh-\rhoh_I))=-\sum_{\beta+\gamma=\alpha} \nabla^\beta \rhoh \nabla\nabla^\gamma \uh-
2\nabla \nabla^\beta \rhoh \nabla^\gamma \uh,
$$
multiply by 
$\zeta^2(x) \la x\ra^{2|\alpha|}\frac{\nabla^\alpha (\rhoh-\rhoh_I)}{\rhoh_I^2}$ and derive the energy identity, similar to the above:
\bea\label{eq:rhof}
&&\frac d{dt} \int \zeta^2\la x\ra ^{2|\alpha|} \left|\frac{\nabla^\alpha(\rhoh-\rhoh_I)}{\rhoh_I}\right|^2 
\le \left(\int \zeta^2 \la x\ra^{2|\alpha|-1}\left|\frac{\nabla^\alpha(\rhoh-\rhoh_I)}{\rhoh_I}\right|^2\right)^{\frac 12}\notag\\ 
&\times&\sum_{|\beta|+|\gamma|= |\alpha|+1, |\beta|\le \frac {{k^\sharp}}2}\left\| \la x\ra^{|\beta|} \frac{\nabla^\beta\rhoh}{\rhoh_I}\right\|_{L^\infty(|x|\ge 10)}
\left(\int \zeta^2 \la x\ra^{2|\gamma|-1} |\nabla^\gamma\uh|^2\right)^{{\frac{1}{2}}}\notag\\ 
&+& \left(\int \zeta^2 \la x\ra^{2|\alpha|-1}(x) \left|\frac{\nabla^\alpha(\rhoh-\rhoh_I)}{\rhoh_I}\right|^2\right)^{\frac 12}\notag\\  
&&\times \sum_{|\beta|+|\gamma|= |\alpha|+1, {1\leq}|\gamma|\le \frac {{k^\sharp}}2}\left(\int \zeta^2\la x\ra^{2|\beta|-1} \left|\frac{\nabla^\beta\rhoh}{\rhoh_I}\right|^2\right)^{\frac 12}\left\|\la x\ra^{{|\gamma|}}\nabla^\gamma\uh\right\|_{L^\infty(|x|\ge 10)}
\notag\\
&&+{ \left(\int \zeta^2 \la x\ra^{2|\alpha|-1}(x) \left|\frac{\nabla^\alpha(\rhoh-\rhoh_I)}{\rhoh_I}\right|^2\right)\left\|\uh\right\|_{L^\infty(|x|\ge 10)}}\notag\\
&&+ \left(\int \zeta^2 \la x\ra^{2|\alpha|-1}(x) \left|\frac{\nabla^\alpha(\rhoh-\rhoh_I)}{\rhoh_I}\right|^2\right)^{\frac 12}\left(\int (\nabla\zeta)^2\la x\ra^{2|\alpha|+1} \left|\frac{\nabla^\alpha(\rhoh-\rhoh_I)}{\rhoh_I}\right|^2\right)^{\frac 12}\left\|\uh\right\|_{L^\infty(|x|\ge 10)}\notag\\  &+&\left(\int \zeta^2 \la x\ra^{2|\alpha|-1}(x) \left|\frac{\nabla^\alpha(\rhoh-\rhoh_I)}{\rhoh_I}\right|^2\right)^{\frac 12}\left(\int \zeta^2\la x\ra^{2|\alpha|+1} \left|\frac{\nabla\nabla^\alpha\rhoh_I}{\rhoh_I}\right|^2\right)^{\frac 12}\left\|\uh\right\|_{L^\infty(|x|\ge 10)}
\eea
We integrate this differential inequality with 
$$
\zeta^2(x)\sim \la x\ra^{-d-2\sigma+\mu},
$$
where 
$$
\mu=\min\{1,2(r-1)\}>0.
$$
All the norms involving $\rhoh$ and $\rhoh-\rhoh_I$ (note that we can either control the latter by absorbing them to 
the left hand side or split them into $\rhoh$ and $\rhoh_I$ and use the previous step to control $\rhoh$ and
the integrability of the function $\zeta^2\la x\ra^{-1}$ to control $\rhoh_I$) on the right hand side will be finite by the previous step, the norms involving $\uh$ will 
be finite by the bootstrap assumptions and the choice of $\mu$, the initial data will be small in view of the assumptions on
$\rhoh-\rhoh_I$ and so will be the time interval $[T_0,T]$. We obtain
\be\label{eq:strrho}
\int_{|x|\ge 12}\la x\ra ^{-d-2\sigma+\mu+2k} \left|\frac{\nabla^k(\rhoh-\rhoh_I)}{\rhoh_I}\right|^2\le \mathcal d_0
\ee
for any $0\le k\le {k^\sharp}-1$.
This estimate immediately implies the pointwise bound
\be\label{eq:strinf}
\left\|\frac{\la x\ra^{k+\frac \mu 2-\sigma}\nabla^k(\rhoh-\rhoh_D)}{\rhoh_D}\right\|_{L^\infty(|x|\ge 12)}\le \mathcal d_0
\ee
for any $0\le k\le {k^\sharp}-2$.
We can translate the above bounds to renormalized variables to obtain
\be\label{eq:strrhor}
\int_{Z\ge 12Z^*}\la Z\ra ^{-d+2k} \left\la \frac Z{Z^*}\right\ra^{\mu-2\sigma}
\left|\frac{\nabla^k\rhot}{\rho_D}\right|^2\le \mathcal d_0
\ee
for any $0\le k\le {k^\sharp}-1$ and 
\be\label{eq:strinfr}
\left\|\left\la\frac Z{Z^*}\right\ra^{\frac \mu 2-\sigma}\frac{\la Z\ra^{k}\nabla^k\rhot}{\rho_D}\right\|_{L^\infty(Z\ge 12Z^*)}\le \mathcal d_0
\ee
for any $0\le k\le {k^\sharp}-2$.

%%%%%%%%%%%%%%%%%%%%%%%%%%%%%%%%%%%%%%%%%%%
\section{Quasilinear energy identity}

%%%%%%%%%%%%%%%%%%%%%%%%%%%%%%%%%%%%%%%%%%%%%%%%
%%%%%%%%%%%%%%%%%%%%%%%%%%%%%%%%%%%%%%%%%%%%%%%%%%
%%%%%%%%%%%%%%%%%%%%%%%%%%%%%%%%%%%%%%%%%%%%%%%%%%%%

\subsection{Linearized flow and control of the potentials}

%%%%%%%%%%%%%%%%%%%%%%%%%%%%%%%%%%%%%%%%

We derive the equations taking into account the localization of the profile.\\

\noindent{\bf step 1} Equation for $\tilde{\rho},\tilde{\Psi}$. Recall \eqref{renormalizedflow}:
$$\left|\begin{array}{ll}\pa_\tau \rho_T=-\rho_T\Delta \Psi_T-\frac{\ell(r-1)}{2}\rho_T-\left(2\pa_Z\Psi_T+ Z\right)\pa_Z\rho_T\\
\pa_\tau \Psi_T=b^2\mathcal F-\left[|\nabla \Psi_T|^2+(r-2)\Psi_T+\Lambda \Psi_T+\rho_T^{p-1}\right]
\end{array}\right.
$$
We  define
\be
\label{profileequationtilde}
\left|\begin{array}{ll}
\pa_\tau \Psi_D+\left[|\nabla \Psi_D|^2+\rho_D^{p-1}+(r-2)\Psi_D+\Lambda \Psi_D\right]=\tilde{\mathcal E}_{P,\Psi}\\
\pa_\tau \rho_D+\rho_D\left[\Delta \Psi_D+\frac{\ell(r-1)}{2}+\left(2\pa_Z\Psi_D+ Z\right)\frac{\pa_Z\rho_D}{\rho_D}\right]=\tilde{\mathcal E}_{P,\rho}
\end{array}\right.
\ee
with $\Et_{P,\rho}, \Et_{P,\Psi}$ supported in $Z\geq 3Z^*$. We introduce the modified potentials  
\be
\label{vneioneinoenvoen}
\Ht_2=1+2\frac{\Psi'_D}{Z}, \ \ \Ht_1=-\left(\Delta \Psi_D+\frac{\ell(r-1)}{2}\right).
\ee 
Their leading order asymptotic behavior for large $Z$ is the same as $H_1, H_2$. It is not affected by dampening 
of the profile.
We now compute the linearized flow in the variables \eqref{definitionvariables}:
\be
\label{exactliearizedflowtilde}
\left|\begin{array}{ll} \pa_\tau \rhot=-\rho_T\Delta \Psit-2\nabla\rho_T\cdot\nabla \Psit+\Ht_1\rhot-\Ht_2\Lambda \rhot-\tilde{\mathcal E}_{P,\rho}\\
\pa_\tau \Psit=b^2\mathcal F-\left[\Ht_2\Lambda \Psit+(r-2)\Psit+|\nabla \Psit|^2+(p-1)\rho_D^{p-2}\rhot +\NL(\rhot)\right]-\Et_{P,\Psi}
\end{array}\right.
\ee
with the nonlinear term 
\be
\label{defnlthoth}
\NL(\tilde{\rho})=(\rho_D+\rhot)^{p-1}-\rho_D^{p-1}-(p-1)\rho_D^{p-2}\rhot.
\ee
Our main task is now to produce an energy identity for \eqref{exactliearizedflowtilde} which respects the quasilinear nature of \eqref{exactliearizedflowtilde} and does not loose derivatives. Observe that the asymptotic bounds for  $Z$ large \eqref{esterrorpotentials}, \eqref{esterrorpotentialsbis} of the potentials are still valid after localization. They will be systematically used in the sequel.\\

\noindent{\bf step 2} Estimate of the potential. We recall the Emden transform formulas \eqref{defhtwohunbis}:
$$
\left|\begin{array}{ll}
H_2=(1-w)\\
H_1=\frac{\ell}{2}(1-w)\left[1+\frac{\Lambda \sigma}{\sigma}\right]\\
H_3=\frac{\Delta \rho_P}{\rho_P}
\end{array}\right.
$$
which, using \eqref{limitprofilesbsi}, \eqref{decayprofile}, yield the bounds:
 $$
\left|\begin{array}{llll}
H_2=1+O\left(\frac{1}{\la Z\ra^r}\right), \ \ H_1=-\frac{2(r-1)}{p-1}+O\left(\frac{1}{\la Z\ra^r}\right)\\
 |\la Z\ra^j\pa_Z^j H_1|+||\la Z\ra^j\pa_Z^jH_2|\lesssim \frac 1{\la Z\ra^{r}}, \ \ j\ge 1\\
 |\la Z\ra^j\pa_Z^j H_3|\lesssim \frac{1}{\la Z\ra^2}\\
  \frac{1}{\la Z\ra^{2(r-1)}}\left[1+O\left(\frac{1}{\la Z\ra^{r}}\right)\right]\lesssim_j |\la Z\ra^j\pa_Z^jQ|\lesssim_j \frac{1}{\la Z\ra^{2(r-1)}}
 \end{array}\right.
 $$
and the commutator bounds
\be
\label{esterrorpotentialsbis}
\left|\begin{array}{lllll}
|[\pa_i^m,H_1]\rho|\lesssim  \sum_{j=0}^{m-1}\frac{|\pa_Z^j\rho|}{\la Z\ra^{r+m-j}}\\
|\nabla\left([\pa_i^m,H_1]\rho\right)|\lesssim \sum_{j=0}^{m}\frac{|\pa_Z^j\rho|}{\la Z\ra^{m-j+r+1}}\\
|\pa_i^m(Q\rho)-Q \rho_m|\lesssim Q\sum_{j=0}^{m-1}\frac{|\pa_Z^j\rho|}{\la Z\ra^{m-j}}\\
|[\pa_i^m,H_2]\Lambda\rho|\lesssim \sum_{j=1}^{m}\frac{|\pa_Z^j\rho|}{\la Z\ra^{r+m-j}}\\
|\nabla\left([\pa_i^m,H_2]\Lambda\Phi\right)|\lesssim \sum_{j=1}^{m+1}\frac{|\pa^j_Z\Phi|}{\la Z\ra^{r+1+m-j}}.
\end{array}\right.
\ee
The same bounds hold for the modified potentials $\Ht_1,\Ht_2$ from \eqref{vneioneinoenvoen}.

%%%%%%%%%%%%%

\subsection{Equations}

%%%%%%%%%%%%%

We have 
$$
\left|\begin{array}{l} \pa_\tau \rhot=-\rho_T\Delta \Psit-2\nabla\rho_T\cdot\nabla \Psit+\Ht_1\rhot-\Ht_2\Lambda \rhot-\tilde{\mathcal E}_{P,\rho}\\
\pa_\tau \Psit=b^2\matchal F -\left[\Ht_2\Lambda \Psit+(r-2)\Psit+|\nabla \Psit|^2+(p-1)\rho_D^{p-2}\rhot +\NL(\rhot)\right]-\Et_{P,\Psi}.
\end{array}\right.
$$
We let  $$ \rhot_{({k^\sharp})}=\Delta^{K}\rhot, \ \ \Psit_{({k^\sharp})}=\Delta^{K}\Psi, \ \ \ut_{({k^\sharp})}=\nabla \Psit_{({k^\sharp})}$$
We use $$[\Delta^{K},\Lambda]={k^\sharp}\Delta^{K}$$ and (recall \eqref{estimatecommutatorvlkeveln}):
$$[\Delta^k,V]\Phi-2k\nabla V\cdot\nabla \Delta^{k-1}\Phi=\sum_{|\alpha|+|\beta|=2k,|\beta|\le 2k-2}c_{k,\alpha,\beta}\nabla^\alpha V\nabla^\beta\Phi
$$ 
which gives:
$$
\Delta^{K}(\Ht_2\Lambda \rhot)={k^\sharp}(\Ht_2+\Lambda \Ht_2)\rho_{({k^\sharp})}+\Ht_2\Lambda \rho_{({k^\sharp})}+\mathcal A_{{k^\sharp}}(\rhot) 
$$
with from \eqref{esterrorpotentials}:
\be
\label{estimatepenttniialvone}
\left|\begin{array}{l}
|\A_{{k^\sharp}}(\rhot)|\lesssim c_k\sum_{j=1}^{{k^\sharp}-1}\frac{|\nabla^j\rhot|}{\la Z\ra^{{k^\sharp}+r-j}}\\
|\nabla \A_{{k^\sharp}}(\rhot)|\lesssim c_k\sum_{j=1}^{{k^\sharp}}\frac{|\nabla^j\rhot|}{\la Z\ra^{{k^\sharp}+r+1-j}}
\end{array}\right.
\ee
where $\nabla^j=\pa^{\alpha_1}_1\dots\pa^{\alpha_d}_d$, $j=\alpha_1+\dots+\alpha_d$ denotes a generic derivative of order $j$. Using \eqref{estimatecommutatorvlkeveln} again:
\bea
\label{estqthohrkbisbis}
\nonumber \pa_\tau \rhot_{({k^\sharp})}&=&\left[\tilde H_1-{k^\sharp}(\tilde H_2+\Lambda \Ht_2)\right]\rhot_{({k^\sharp})}-\Ht_2\Lambda \rhot_{({k^\sharp})}-(\Delta^{K}\rho_T)\Delta \Psit-{k^\sharp}\nabla\rho_T\cdot\nabla \Psit_{({k^\sharp})}-\rho_T\Delta\Psit_{({k^\sharp})}\\
 &-& 2\nabla(\Delta^{K}\rho_T)\cdot\nabla \Psit-2\nabla \rho_T\cdot\nabla \Psit_{({k^\sharp})}+  F_1
\eea
with
\bea
\label{formluafonebis}
F_1&=&-\Delta^{K}\tilde{\mathcal E}_{P,\rho}+[\Delta^{K},\Ht_1]\rhot-\A_{{k^\sharp}}(\rhot)\\
\nonumber &-& \sum_{\left|\begin{array}{ll} j_1+j_2={k^\sharp}\\ j_1\geq 2, j_2\geq 1\end{array}\right.}c_{j_1,j_2}\nabla^{j_1}\rho_T\pa^{j_2}\Delta\Psit-\sum_{\left|\begin{array}{ll}j_1+j_2={k^\sharp}\\
j_1,j_2\geq 1\end{array}\right.}c_{j_1,j_2}\nabla^{j_1}\nabla\rho_T\cdot\nabla^{j_2}\nabla\Psit.
\eea
For the second equation, we have similarly:
\bea
\label{nkenononenonbis}
 &&\pa_\tau\Psit_{({k^\sharp})}\\
\nonumber &=&- {k^\sharp}(\Ht_2+\Lambda \Ht_2)\Psit_{({k^\sharp})}-\Ht_2\Lambda \Psit_{({k^\sharp})}-(r-2)\Psit_{({k^\sharp})}-2\nabla \Psit\cdot\nabla \Psit_{({k^\sharp})}\\
\nonumber&-& \left[(p-1)\rho_P^{p-2}\rhot_{({k^\sharp})}+{k^\sharp}(p-1)(p-2)\rho_D^{p-3}\nabla\rho_D\cdot\nabla\Delta^{K-1}\rhot\right]+F_2
\eea
with
\bea
\label{estqthohrkbisbisbisbisbibfebjbifebfji}
\nonumber F_2&=& b^2\Delta^{K}\matchal F -\Delta^{K}\tilde{\mathcal E}_{P,\Psi}-\A_{{k^\sharp}}(\Psit)-(p-1)\left([\Delta^{K},\rho_D^{p-2}]\rhot-{k^\sharp}(p-2)\rho_D^{p-3}\nabla\rho_D\cdot\nabla\Delta^{K-1}\rhot\right)\\
&-& \sum_{j_1+j_2={k^\sharp},j_1,j_2\geq 1}\nabla^{j_1}\nabla\Psit\cdot\nabla^{j_2}\nabla\Psit-\Delta^{K}\NL(\rhot).
\eea

\noindent{\bf step 1} Algebraic energy identity.

Let $\chi$ be a smooth function $\chi=\chi(\tau,Z)$ 
and compute the quasilinear energy identity:
\bee
&&\frac 12\frac{d}{d\tau}\left\{(p-1)\int \chi\rho_D^{p-2}\rho_T\rhot_{({k^\sharp})}^2+\int\chi\rho_T^2|\nabla \Psit_{({k^\sharp})}|^2\right\}\\
&=& \frac 12\left\{(p-1)\int \pa_\tau\chi\rho_D^{p-2}\rho_T\rhot_{({k^\sharp})}^2+\int\pa_\tau\chi\rho_T^2|\nabla \Psit_{({k^\sharp})}|^2\right\}
\\&+&
\frac{p-1}{2}\int \chi(p-2)\pa_\tau\rho_D\rho_D^{p-3}\rho_T\rhot_{({k^\sharp})}^2+ \int\chi\pa_\tau\rho_T\left[\frac{p-1}{2}\rho_D^{p-2}\rhot_{({k^\sharp})}^2+\rho_T|\nabla \Psit_{({k^\sharp})}|^2\right]\\
&+& \int \pa_\tau\rhot_{({k^\sharp})}\left[(p-1)\chi\rho_D^{p-2}\rho_T\rhot_{({k^\sharp})}\right]\\
&-&  \int\pa_\tau\Psit_{({k^\sharp})}\left[2\chi\rho_T\nabla\rho_T\cdot\nabla \Psit_{({k^\sharp})}+\chi\rho_T^2\Delta \Psit_{({k^\sharp})}+\rho_T^2\nabla \chi\cdot\nabla \Psit_{({k^\sharp})}\right].
\eee
We inject the equation:
\bee
&&\int\pa_\tau \rhot_{({k^\sharp})}\left[(p-1)\chi\rho_D^{p-2}\rho_T\rhot_{({k^\sharp})}\right]=  \int F_1\left[(p-1)\chi\rho_D^{p-2}\rho_T\rhot_{({k^\sharp})}\right]\\
& + & \int \left[(\Ht_1-k^{\sharp}(\Ht_2+\Lambda \Ht_2))\rhot_{({k^\sharp})}-\Ht_2\Lambda \rhot_{({k^\sharp})}-(\Delta^K\rho_T)\Delta \Psit-2\nabla(\Delta^K\rho_T)\cdot\nabla \Psit\right]\\
&\times& \left[(p-1)\chi\rho_D^{p-2}\rho_T\rhot_{({k^\sharp})}\right]\\
& - & \int k^{\sharp}\nabla\rho_T\cdot\nabla\Psit_{({k^\sharp})}\left[(p-1)\chi\rho_D^{p-2}\rho_T\rhot_{({k^\sharp})}\right]\\
&-&  \int (\rho_T\Delta\Psit_{({k^\sharp})}+2\nabla \rho_T\cdot\nabla \Psit_{({k^\sharp})})\left[(p-1)\chi\rho_D^{p-2}\rho_T\rhot_{({k^\sharp})}\right]
\eee
and
\bee
&-&  \int\pa_\tau\Psit_{({k^\sharp})}\left[2\chi\rho_T\nabla\rho_T\cdot\nabla \Psit_{({k^\sharp})}+\chi\rho_T^2\Delta \Psit_{({k^\sharp})}+\rho_T^2\nabla \chi\cdot\nabla \Psit_{({k^\sharp})}\right]=  -\int F_2\nabla\cdot(\chi\rho_T^2\nabla \Psit_{({k^\sharp})})\\
&-& \int\Big\{-k^{\sharp}(\Ht_2+\Lambda \Ht_2)\Psit_{({k^\sharp})}-\Ht_2\Lambda \Psit_{({k^\sharp})}-(r-2)\Psit_{({k^\sharp})}-2\nabla \Psit\cdot\nabla \Psit_{({k^\sharp})}\\
&-& \left[(p-1)\rho_P^{p-2}\rhot_{({k^\sharp})}+k^{\sharp}(p-1)(p-2)\rho_D^{p-3}\nabla\rho_D\cdot\nabla\Delta^{K-1}\rhot\right]\Big\}\\
&\times&\left[2\chi\rho_T\nabla\rho_T\cdot\nabla \Psit_{({k^\sharp})}+\chi\rho_T^2\Delta \Psit_{({k^\sharp})}+\rho_T^2\nabla \chi\cdot\nabla \Psit_{({k^\sharp})}\right]\\
& = & \int \chi\rho^2_T\nabla \Psit_{({k^\sharp})}\cdot\nabla F_2\\
&-& \int\left[-k^{\sharp}(\Ht_2+\Lambda\Ht_2)\Psit_{({k^\sharp})}-\Ht_2\Lambda \Psit_{({k^\sharp})}-(r-2)\Psit_{({k^\sharp})}-2\nabla \Psit\cdot\nabla \Psit_{({k^\sharp})}\right]\nabla\cdot(\chi\rho_T^2\nabla \Psit_{({k^\sharp})})\\
& +& \int (p-1)\rho_P^{p-2}\rhot_{({k^\sharp})}\left[2\chi\rho_T\nabla\rho_T\cdot\nabla \Psit_{({k^\sharp})}+\chi\rho_T^2\Delta \Psit_{({k^\sharp})}+\rho_T^2\nabla \chi\cdot\nabla \Psit_{({k^\sharp})}\right]\\
& + & \int k^{\sharp}(p-1)(p-2)\rho_D^{p-3}\nabla\rho_D\cdot\nabla\Delta^{K-1}\rhot\nabla\cdot(\chi\rho_T^2\nabla \Psit_{({k^\sharp})})
\eee
Adding both identities yields the quasilinear energy identity:
\bea\label{enerq}
&&\frac 12\frac{d}{d\tau}\left\{(p-1)\int \chi\rho_D^{p-2}\rho_T\rhot_{({k^\sharp})}^2+\int\chi\rho_T^2|\nabla \Psit_{({k^\sharp})}|^2\right\}\notag\\
\nonumber & = & \frac 12\int\left(\frac{\pa_\tau\chi}{\chi}+\frac{\pa_\tau\rho_T}{\rho_T}+(p-2)\frac{\pa_\tau\rho_D}{\rho_D}\right)(p-1)\chi\rho_D^{p-2}\rho_T\rhot_{({k^\sharp})}^2+\frac 12\int\left(\frac{\pa_\tau\chi}{\chi}+2\frac{\pa_\tau\rho_T}{\rho_T}\right)\rhot_T^2|\nabla \Psi_{({k^\sharp})}|^2\\
\nonumber & + & \int F_1\chi(p-1)\rho_D^{p-2}\rho_T\rhot_{({k^\sharp})}+\int \chi\rho^2_T\nabla F_2\cdot\nabla \Psit_{({k^\sharp})}\notag\\
\nonumber  & + & \int \left[(\Ht_1-k^{\sharp}(\Ht_2+\Lambda\Ht_2))\rhot_{({k^\sharp})}-\Ht_2\Lambda \rhot_{({k^\sharp})}-(\Delta^K\rho_T)\Delta \Psit-2\nabla(\Delta^K\rho_T)\cdot\nabla \Psit\right]\\
&\times & (p-1)\chi\rho_D^{p-2}\rho_T\rhot_{({k^\sharp})}\notag\\
\nonumber&-& \int\left[-k^{\sharp}(\Ht_2+\Lambda \Ht_2)\Psit_{({k^\sharp})}-\Ht_2\Lambda \Psit_{({k^\sharp})}-(r-2)\Psit_{({k^\sharp})}-2\nabla \Psit\cdot\nabla \Psit_{({k^\sharp})}\right]\nabla\cdot(\chi\rho_T^2\nabla \Psi_{({k^\sharp})})\notag\\
\nonumber& - & \int k^{\sharp}\nabla\rho_T\cdot\nabla \Psit_{({k^\sharp})}\left[(p-1)\chi\rho_D^{p-2}\rho_T\rhot_{({k^\sharp})}\right]\\
\nonumber &+&  \int k^{\sharp}(p-1)(p-2)\rho_D^{p-3}\nabla\rho_D\cdot\nabla\Delta^{K-1}\rhot\nabla\cdot(\chi\rho_T^2\nabla \Psit_{({k^\sharp})})\\
& + &\int(p-1)\rho_D^{p-2}\rhot_{({k^\sharp})}\left[\rho^2_T\nabla \chi\cdot\nabla \Psit_{({k^\sharp})}\right].
\eea

\noindent{\bf step 2} Reexpressing the quadratic terms. We integrate by parts:
\bee
-\int \Ht_2\Lambda \rhot_{({k^\sharp})}(p-1)\chi\rho_D^{p-2}\rho_T\rhot_{({k^\sharp})}=\frac{p-1}{2}\int  \chi\Ht_2\rho_T\rho_D^{p-2}\rhot_{({k^\sharp})}^2\left(d+\frac{\Lambda \Ht_2}{\Ht_2}+\frac{(p-2)\Lambda\rho_D}{\rho_D}+\frac{\Lambda\chi}{\chi}\right).
\eee
Then
\bee
&&k^{\sharp}\int(\Ht_2+\Lambda \Ht_2)\Psit_{({k^\sharp})}\nabla\cdot(\chi\rho_T^2\nabla \Psi_{({k^\sharp})})\\
&=&  -k^{\sharp}\int\chi\rho_T^2(\Ht_2+\Lambda \Ht_2)|\nabla \Psit_{({k^\sharp})}|^2-k^{\sharp}\int \chi\rho_T^2\Psit_{({k^\sharp})}\nabla \Psit_{({k^\sharp})}\cdot \nabla(\Ht_2+\Lambda \Ht_2)
\eee
and using spherical symmetry:
\bee
&&\int \Ht_2\Lambda \Psit_{({k^\sharp})}\nabla\cdot(\chi\rho_T^2\nabla \Psi_{({k^\sharp})})=-\int\chi\rho_T^2\nabla \Psi_{({k^\sharp})}\cdot\nabla(\Ht_2\Lambda \Psit_{({k^\sharp})})\\
&=& -\int\chi\Lambda \Ht_2\rho_T^2|\nabla \Psit_{({k^\sharp})}|^2-  \int \Ht_2\chi\rho_T^2\pa_Z\Psit_{({k^\sharp})}\pa_Z(\Lamdba \Psit_{({k^\sharp})})\\
& = & -\int\chi\rho^2_T\Lambda \Ht_2\rho_T^2|\nabla \Psit_{({k^\sharp})}|^2+\frac 12\int\chi\rho_T^2\Ht_2|\nabla \Psit_{({k^\sharp})}|^2\left[d-2+\frac{\Lamdba \Ht_2}{\Ht_2}+\frac{\Lambda \chi}{\chi}+\frac{2\Lambda\rho_T}{\rho_T}\right]\\
& = & \int\chi\Ht_2\rho_T^2|\nabla \Psit_{({k^\sharp})}|^2\left[\frac{d-2}{2}-\frac 12\frac{\Lamdba \Ht_2}{\Ht_2}+\frac 12\frac{\Lambda \chi}{\chi}+\frac{\Lambda\rho_T}{\rho_T}\right]
\eee
and
$$(r-2)\int \Psit_{({k^\sharp})}\nabla\cdot(\chi\rho_T^2\nabla \Psi_{({k^\sharp})})=-(r-2)\int\chi\rho_T^2|\nabla \Psit_{({k^\sharp})}|^2$$
and integrating by parts and using radiality:
\bee
&& \int k^{\sharp}(p-1)(p-2)\rho_D^{p-3}\nabla\rho_D\cdot\nabla\Delta^{K-1}\rhot\nabla\cdot(\chi\rho_T^2\nabla \Psit_{({k^\sharp})})\\
 & = & -k^{\sharp}(p-1)(p-2)\int \chi\rho_T^2\nabla \Psi_{({k^\sharp})}\cdot\nabla\left[\rho_D^{p-3}\nabla\rho_D\cdot\nabla\Delta^{K-1}\rhot\right]\\
 & = & -k^{\sharp}(p-1)(p-2)\int \chi\rho_T^2\rho_D^{p-3}\pa_Z\rho_D\rhot_{({k^\sharp})}\pa_Z \Psit_{({k^\sharp})}\\
 & - & k^{\sharp}(p-1)(p-2)\int \chi\rho_T^2\pa_Z \Psit_{({k^\sharp})}\left[\pa_Z\left(\rho_D^{p-3}\pa_Z\rho_D\pa_Z\Delta^{K-1}\rhot\right)-\rho_D^{p-3}\pa_Z\rho_D\rhot_{({k^\sharp})}\right]
 \eee
 We therefore arrive to the quasilinear energy identity:
 \bea
\label{algebracienergyidnentiypouet}
  &&\frac 12\frac{d}{d\tau}\left\{(p-1)\int \chi\rho_D^{p-2}\rho_T\rhot_{({k^\sharp})}^2+\int\chi\rho_T^2|\nabla \Psit_{({k^\sharp})}|^2\right\}\\
\nonumber& = &  \frac 12\int\left(\frac{\pa_\tau\chi}{\chi}+\frac{\pa_\tau\rho_T}{\rho_T}+(p-2)\frac{\pa_\tau\rho_D}{\rho_D}\right)(p-1)\chi\rho_D^{p-2}\rho_T\rhot_{({k^\sharp})}^2+\frac 12\int\left(\frac{\pa_\tau\chi}{\chi}+2\frac{\pa_\tau\rho_T}{\rho_T}\right)\rhot_T^2|\nabla \Psi_{({k^\sharp})}|^2\\
\nonumber& - & \int (p-1)\chi\rho_D^{p-2}\rho_T\rhot_{({k^\sharp})}^2\left[-\Ht_1+k^{\sharp}(\Ht_2+\Lambda\Ht_2)-\frac d2\Ht_2-\frac 12\Lambda \Ht_2-\frac{p-2}{2}\Ht_2\frac{\Lambda\rho_D}{\rho_D}-\frac{\Ht_2}2\frac{\Lambda\chi}{\chi}\right]\\
\nonumber& - & \int \chi\rho_T^2|\nabla \Psit_{({k^\sharp})}|^2\left[k^{\sharp}(\Ht_2+\Lambda \Ht_2)+r-2-\frac{d-2}{2}\Ht_2+\frac 12\Lamdba \Ht_2-\frac{\Ht_2}2\frac{\Lambda \chi}{\chi}-\Ht_2\frac{\Lambda\rho_T}{\rho_T}\right]\\
\nonumber& &+\int \rhot_{({k^\sharp})}\pa_Z\Psit_{({k^\sharp})}\left[- k^{\sharp}(p-1)\chi\rho_D^{p-2}\rho_T\pa_Z\rho_T\right.\\
\nonumber &- & \left.k^{\sharp}(p-1)(p-2)\chi\rho_T^2\rho_D^{p-3}\pa_Z\rho_D+(p-1)\rho_D^{p-2}\rho^2_T\pa_Z\chi\right]\\
\nonumber & + & \int F_1\chi(p-1)\rho_D^{p-2}\rho_T\rhot_{({k^\sharp})}+\int \chi\rho^2_T\nabla F_2\cdot\nabla \Psit_{({k^\sharp})}\\
\nonumber& + & \nonumber  \int \left[-(\Delta^K\rho_T)\Delta \Psit-2\nabla(\Delta^K\rho_T)\cdot\nabla \Psit\right](p-1)\chi\rho_D^{p-2}\rho_T\rhot_{({k^\sharp})}\\
\nonumber&- & k^{\sharp}(p-1)(p-2)\int \chi\rho_T^2\pa_Z \Psi_{({k^\sharp})}\left[\pa_Z\left(\rho_D^{p-3}\pa_Z\rho_D\pa_Z\Delta^{K-1}\rhot\right)-\rho_D^{p-3}\pa_Z\rho_D\rhot_{({k^\sharp})}\right]\\
\nonumber& + & 2\int\nabla \Psit\cdot\nabla \Psit_{({k^\sharp})}\nabla\cdot(\chi\rho_T^2\nabla \Psit_{({k^\sharp})})-k^{\sharp}\int \chi\rho_T^2\Psit_{({k^\sharp})}\nabla \Psit_{({k^\sharp})}\cdot \nabla(\Ht_2+\Lambda \Ht_2).
\eea

%%%%%%%%%%%%%%%%%%%%%%%%%%%%%%%%%%%%%%%%%%%%%

\subsection{Quadratic  forms}

%%%%%%%%%%%%%%%%%%%%%%%%%%%%%%%%%%%%%%%%%%%%%

We study the $(\chi,\Lamdba \chi)$ quadratic forms appearing in \eqref{algebracienergyidnentiypouet}.\\

\noindent{\bf step 1} Leading order $\chi$ quadratic form. We recall from \eqref{coercivityquadrcouplinginside}, \eqref{P}:
\be
\label{vnioevineneovklvknlve}
H_2+\Lambda H_2=(1-w-\Lambda w)\ge c_{d,p,r}>0.
\ee 
We assume that $k^\sharp\gg 1$, so that the terms with $k^\sharp$ dominate:
$$
-\Ht_1+k^{\sharp}(\Ht_2+\Lambda\Ht_2)-\frac d2\Ht_2-\frac 12\Lambda \Ht_2-\frac{p-2}{2}\Ht_2\frac{\Lambda\rho_D}{\rho_D}=k^{\sharp}\left(1+O\left(\frac 1{k^\sharp}\right)\right)(\Ht_2+\Lambda\Ht_2)
$$
$$
k^{\sharp}(\Ht_2+\Lambda \Ht_2)+r-2-\frac{d-2}{2}\Ht_2+\frac 12\Lamdba \Ht_2-\Ht_2\frac{\Lambda\rho_T}{\rho_T}
=k^{\sharp}\left(1+O\left(\frac 1{k^\sharp}\right)\right)(\Ht_2+\Lambda\Ht_2)
$$
and  claim the pointwise coercivity of the quadratic form: $\exists c_{d,p,r}>0$ such that uniformly $\forall Z\ge 0$,
\bea
\label{enineinevnveo}
\nonumber 
&&k^\sharp (\Ht_2+\Lambda \Ht_2)\left[ (p-1)\rho_D^{p-2}\rho_T\rhot_{{({k^\sharp})}}^2+\rho_T^2|\nabla \Psit_{({k^\sharp})}|^2\right]+ k^\sharp(p-1)\rho_D\pa_Z (\rho_D^{p-1})\rhot_{({k^\sharp})}\pa_Z \Psit_{({k^\sharp})}\\
& \ge & c_{d,p,r} k^\sharp\left[ (p-1) \rho_D^{p-2}\rho_T\rhot_{{({k^\sharp})}}^2+\rho_T^2|\nabla \Psit_{({k^\sharp})}|^2\right]
\eea
The cross term is lower order for $Z$ large: $$|(p-1)\rho_D\pa_Z (\rho_D^{p-1})\rhot_{({k^\sharp})}\pa_Z \Psit_{({k^\sharp})}|\lesssim \frac{\rho_T^{p-1}}{\la Z\ra}\rhot_{({k^\sharp})}\rho_T\pa_Z\Psi_{({k^\sharp})}\le \mathcal d\left[ (p-1)\rho_D^{p-2}\rho_T\rhot_{{({k^\sharp})}}^2+\rho_T^2|\nabla \Psit_{({k^\sharp})}|^2\right]
$$ for $Z>Z(\mathcal d)$ large enough. For $Z\le Z(\mathcal d)$, using the smallness \eqref{smallglobalboot}, \eqref{enineinevnveo}
 is implied by: 
\bea
\label{enineinevnveobis}
\nonumber 
&&(H_2+\Lambda H_2)\left[ (p-1)Q\rhot_{{({k^\sharp})}}^2+\rho_P^2|\nabla \Psit_{({k^\sharp})}|^2\right]+ (p-1)\rho_P\pa_Z Q\rhot_{({k^\sharp})}\pa_Z \Psit_{({k^\sharp})}\\
& \ge & c_{d,p,r}\left[ (p-1)Q\rhot_{{({k^\sharp})}}^2+\rho_P^2|\nabla \Psit_{({k^\sharp})}|^2\right]
\eea
We compute the discriminant:
\bee
&&{\rm Discr}=(p-1)^2\rho_P^2(\pa_ZQ)^2-4(p-1)\rho_P^2Q(H_2+\Lambda H_2)^2\\
& = & (p-1)\rho_P^2 Q\left[(p-1)\frac{(\pa_ZQ)^2}{Q}-4(1-w-\Lambda w)^2\right]
\eee
We compute from \eqref{relationsprofileemden} recalling \eqref{definitionF}:
\bee
(p-1)\frac{(\pa_ZQ)^2}{Q} &=&   (p-1)\left(2\pa_Z\sqrt{Q}\right)^2=(p-1)\left(\frac{1-e}{2}\sqrt{\ell}\pa_Z(\sigma_P Z)\right)^2=(1-e)^2(\pa_Z(Z\sigma_P))^2\\
& =& \frac{4}{r^2}(\pa_Z(Z\sigma_P))^2=4F^2
\eee
and hence from \eqref{coercivityquadrcouplinginside}, \eqref{P} the lower bound:
$$-D=4(p-1)\rho_P^2 Q\left[(1-w-\Lambda w)^2-F^2\right]\ge c_{d,p,r}(p-1)\rho_P^2Q, \ \ c_{d,p,r}>0$$
which together with \eqref{vnioevineneovklvknlve} concludes the proof of \eqref{enineinevnveo}.
\\

\noindent{\bf step 2} Leading order $\Lambda \chi$ quadratic form. The quadratic form containing $\Lambda\chi$:
\bee
&&\int -\Lambda \chi\left\{\frac{H_2}{2}\left[(p-1) Q\rhot_{({k^\sharp})}^2+\rho_P^2|\nabla \Psit_{({k^\sharp})}|^2\right]-\frac{1}{Z}(p-1)\rho_P Q\rhot_{({k^\sharp})}\pa_Z\Psit_{({k^\sharp})}\right\}
\eee
Its discriminant is 
\bee\label{outside}
{\rm Discr}&=&\left(\frac{(p-1)Q\rho_P}{Z}\right)^2-(p-1)Q\rho_P^2H_2^2=(p-1)Q\rho_P^2\left[\sigma^2-(1-w)^2\right]<0
\eee
for $Z>Z_2$.\\

We note that \eqref{enineinevnveo} holds for {\it all} $Z$ only under the condition \eqref{P} which hold in $d=3$ and 
$\ell>\sqrt 3$. On the other hand, for $d=2$ or $d=3$ and $\ell\leq\sqrt 3$,  \eqref{enineinevnveo} still holds for $Z\leq Z_2$
and $Z$ sufficiently large $Z\geq Z(\d)$. In those cases, choosing
\be
\chi=\left|\begin{array}{ll}1\ \ \ \ Z\leq Z_2\\
e^{-j^\sharp (Z-Z_2)}\ \ Z>Z_2
\end{array}\right.
\ee
with $j^\sharp\gg k^\sharp$ ensures that the {\it full} $(\chi,\Lambda\chi)$ quadratic form is positive definite:
\bea
\label{coereuler}
\nonumber 
k^\sharp \chi (\Ht_2&+&\Lambda \Ht_2)\left[ (p-1)\rho_D^{p-2}\rho_T\rhot_{{({k^\sharp})}}^2+\rho_T^2|\nabla \Psit_{({k^\sharp})}|^2\right]+ (p-1)\chi\rho_D\pa_Z (\rho_D^{p-1})\rhot_{({k^\sharp})}\pa_Z \Psit_{({k^\sharp})}
\\&-&\Lambda \chi\left\{\frac{H_2}{2}\left[(p-1) Q\rhot_{({k^\sharp})}^2+\rho_P^2|\nabla \Psit_{({k^\sharp})}|^2\right]-\frac{1}{Z}(p-1)\rho_P Q\rhot_{({k^\sharp})}\pa_Z\Psit_{({k^\sharp})}\right\}\notag\\
& \ge & c_{d,p,r} k^\sharp \chi\left[ (p-1) \rho_D^{p-2}\rho_T\rhot_{{({k^\sharp})}}^2+\rho_T^2|\nabla \Psit_{({k^\sharp})}|^2\right]
\eea

\section{The highest unweighted energy norm}
\label{sectionsobolev}

%%%%%%%%%%%%%%%%%%%%%%%%%%%%%%%%%%%%%%%%
 %%%%%%%%%%%%%%%%%%%%%%%%%%%%%%%%%%%%%%%%%%

In this section we establish control of the highest energy norm. This is an essential step to control the $b$ dependence of the flow. It will be achieved through an {\it unweighted} energy estimate for the highest order derivatives. Below we will systematically exploit the gains achieved through faster decay in $Z$ of various tail terms, see e.g. \eqref{esterrorpotentials}. Typical 
improvements will be usually of order $r$ or $(r-1)$\footnote{Recall that in the range of considered 
$r$, close to the limiting values $\eye(d,\ell)$, we have $r_+(d,\ell)>r^*(d,\ell)=\frac{d+\ell}{\sqrt d+\ell}>1$.}.
 Sometimes, we will replace them by a generic constant $\delta>0$.

%%%%%%%%%%%%%%%%%%%%%%%%%%%%%%%%%%%%%%%%%%%%%%%%%%%%%%

\subsection{Controlling the highest energy norm}

%%%%%%%%%%%%%%%%%%%%%%%%%%%%%%%%%%%%%%%%%%%%%%%%%%%%%%

We now prove the highest order energy estimate {\em without weight}. Coercivity of a quadratic form arising in the estimate will follow thanks to the global lower bound \eqref{P} and \eqref{enineinevnveo}.
We let $${k^\sharp}=2K, \ \ K\in \Bbb N$$ and denote in this section $$ \rhot_{({k^\sharp})}=\Delta^{K}\rhot, \ \ \Psit_{({k^\sharp})}=\Delta^{K}\Psi, \ \ \ut_{({k^\sharp})}=\nabla \Psit_{({k^\sharp})}.$$

\begin{lemma}[Control of the highest unweighted energy norm]
\label{proproopr}
For some universal constant $c_{{k^\sharp}}$ ($0<c_{{k^\sharp}}\ll \delta_g$),
\be
\label{cneioneoneon}
(p-1)\int \rho_D^{p-2}\rho_T\rhot_{({k^\sharp})}^2+\int\rho_T^2|\nabla \Psit_{({k^\sharp})}|^2\le e^{-c_{{k^\sharp}}\tau}
\ee
\end{lemma}

\begin{proof}[Proof of Lemma \ref{proproopr}] \,{}

\noindent{\bf step 1} Control of lower order terms. We interpolate the rough bound inherited from 
\eqref{boundbootbound}:
$$(p-1)\int \rho_D^{p-2}\rho_T\rhot_{({k^\sharp})}^2+\int\rho_T^2|\nabla \Psit_{({k^\sharp})}|^2\le 1
$$ with the low Sobolev bound \eqref{eq:bootdecay} for $Z\leq (Z^*)^c$, with $0<c=c(k^\sharp, \delta_g)\ll 1$, and use \eqref{boundbootbound} for $Z>(Z^*)^c$
 to estimate: 
\be
\label{veniovnineonelweroidre} 
\sum_{j=0}^{{k^\sharp}-1}\sum_{i=1}^d(p-1)\int \rho_D^{p-2}\rho_T\frac{(\pa^j_i\rhot)^2}{\la Z\ra^{2({k^\sharp}-j)}}+\int\rho_T^2\frac{|\nabla \pa^j_i\Psit|^2}{\la Z\ra^{2({k^\sharp}-j)}}\le e^{-c_{{k^\sharp}}\tau}
\ee where $c_{{k^\sharp}}=c({k^\sharp},\delta_g)>0$. The estimate \eqref{veniovnineonelweroidre} will be 
used repeatedly in the sequel.\\

\noindent{\bf step 2} Energy identity. We use the identity derived in \eqref{enerq} with $\chi\equiv 1$:
\bea
\label{algebracienergyidnentiybis}
&&\frac 12\frac{d}{d\tau}\left\{(p-1)\int\rho_D^{p-2}\rho_T\rhot_{({k^\sharp})}^2+\int\rho_T^2|\nabla \Psit_{({k^\sharp})}|^2\right\}\\
\nonumber & = & \frac 12\int\left(\frac{\pa_\tau\rho_T}{\rho_T}+(p-2)\frac{\pa_\tau\rho_D}{\rho_D}\right)(p-1)\rho_D^{p-2}\rho_T\rhot_{({k^\sharp})}^2+\frac 12\int\left(2\frac{\pa_\tau\rho_T}{\rho_T}\right)\rhot_T^2|\nabla \Psi_{({k^\sharp})}|^2\\
\nonumber & + & \int F_1(p-1)\rho_D^{p-2}\rho_T\rhot_{({k^\sharp})}+\int \rho^2_T\nabla F_2\cdot\nabla \Psit_{({k^\sharp})}\\
\nonumber  & + & \int \left[(\Ht_1-{k^\sharp}(\Ht_2+\Lambda \Ht_2))\rhot_{({k^\sharp})}-\Ht_2\Lambda \rhot_{({k^\sharp})}-(\Delta^{K}\rho_T)\Delta \Psit-2\nabla(\Delta^{K}\rho_T)\cdot\nabla \Psit\right](p-1)\rho_D^{p-2}\rho_T\rhot_{({k^\sharp})}\\
\nonumber&-& \int\left[-{k^\sharp}(\Ht_2+\Lamdba \Ht_2)\Psit_{({k^\sharp})}-\Ht_2\Lambda \Psit_{({k^\sharp})}-(r-2)\Psit_{({k^\sharp})}-2\nabla \Psit\cdot\nabla \Psit_{({k^\sharp})}\right]\nabla\cdot(\rho_T^2\nabla \Psi_{({k^\sharp})})\\
\nonumber& - & {k^\sharp}\int \nabla\rho_T\cdot\nabla\Psit_{({k^\sharp})}\left[(p-1)\rho_D^{p-2}\rho_T\rhot_{({k^\sharp})}\right]+  \int {k^\sharp}(p-1)(p-2)\rho_D^{p-3}\nabla\rho_D\cdot\nabla \Delta^{K-1}\rhot\nabla\cdot(\rho_T^2\nabla \Psi_{({k^\sharp})}).
\eea
We now estimate all terms in \eqref{algebracienergyidnentiybis}. We track {\em exactly} the quadratic terms which arise at the highest level of derivatives
and which will be shown to be coercive provided ${k^\sharp}>{k^\sharp}^*(d,r,p)\gg 1$ has been chosen large enough. 

We denote 
$$I_{{k^\sharp}}=(p-1)\int\rho_D^{p-2}\rho_T\rhot_{{k^\sharp}}^2+\int\rho_T^2|\nabla \Psit_{{k^\sharp}}|^2.$$

\noindent{\bf step 3} Leading order terms.\\

\noindent\underline{\em Cross term}. We use 
\be
\label{vnenvenveneneonv}
\frac{|\rhot|}{\rho_T}+\frac{|\Lambda\rhot|}{\rho_T}\lesssim \mathcal d
\ee
 to  compute the first coupling term:
\bee
{k^\sharp}(p-1)&&\int\nabla\rho_T\cdot\nabla\Psit_{({k^\sharp})}\rho_D^{p-2}\rho_T\rhot_{({k^\sharp})}= -{k^\sharp}\int\rho_D\nabla \rho^{p-1}_D\cdot\nabla \Psit_{({k^\sharp})}\rhot_{({k^\sharp})}\\&&+ O\left(\mathcal d\int \frac{|\nabla \Psit_{({k^\sharp})}|\rho_D^{p-1}\rho_T|\rhot_{({k^\sharp})}|}{\la Z\ra^{\frac 12}}\right)\\
&& =  -{k^\sharp}\int\rho_D\nabla \rho^{p-1}_D\cdot\nabla \Psit_{({k^\sharp})}\rhot_{({k^\sharp})}+O\left(\mathcal d I_{{k^\sharp}}\right) \eee
The second coupling term is computed after an integration by parts using \eqref{vnenvenveneneonv}, the control of lower order terms \eqref{veniovnineonelweroidre} and the radial assumption:
\bee
& & {k^\sharp}(p-1)(p-2)\int \nabla \cdot(\rho_T^2\nabla \Psi_{({k^\sharp})})\rho_D^{p-3}\nabla\rho_D\cdot\nabla\Delta^{K-1}\rhot\\
& = & -{k^\sharp}(p-1)(p-2)\int\rho_T^2\nabla \Psi_{({k^\sharp})}\cdot\nabla\left(\rho_D^{p-3}\nabla\rho_D\cdot\nabla\Delta^{K-1}\rhot\right)\\
& = &  -{k^\sharp}(p-1)(p-2)\int\rho_T^2\pa_Z\Psi_{({k^\sharp})}\pa_Z\left(\rho_D^{p-3}\pa_Z\rho_D\pa_Z\Delta^{K-1}\rhot\right)\\
& = &  -{k^\sharp}(p-1)(p-2)\int\rho_D^{p-3}\pa_Z\rho_D\rho_T^2\pa_Z\Psi_{({k^\sharp})}\pa_Z^2\Delta^{K-1}\rhot+O\left(\int c_{({k^\sharp})}\rho_T|\nabla \Psi_{({k^\sharp})}|\rho_T^{p-1}\frac{|\pa^{{k^\sharp}-1}\rhot|}{\la Z\ra}\right)\\
& = & -\int {k^\sharp}(p-2)\rho_D\pa_Z(\rho_D^{p-1})\pa_Z\Psit_{({k^\sharp})}\rhot_{({k^\sharp})}+ O\left(\mathcal d I_{{k^\sharp}} +\int c_{({k^\sharp})}\rho_T|\nabla \Psi_{({k^\sharp})}|\rho_T^{p-1}\frac{|\pa^{{k^\sharp}-1}\rhot|}{\la Z\ra}\right)\\
& = & -{k^\sharp}(p-2)\int\rho_D\nabla \rho^{p-1}_D\cdot\nabla \Psit_{({k^\sharp})}\rhot_{({k^\sharp})}+O(e^{-c_{{k^\sharp}}\tau}+\mathcal d I_{{k^\sharp}}).
\eee

\noindent\underline{\em $\rho_{({k^\sharp})}$ terms}. We compute:
$$\int(\Ht_1-{k^\sharp}(\Ht_2+\Lambda \Ht_2)\rhot_{({k^\sharp})})(p-1)\rho_D^{p-2}\rho_T\rhot_{({k^\sharp})}= \int(\Ht_1-{k^\sharp}(\Ht_2+\Lambda \Ht_2))(p-1)\rho_D^{p-2}\rho_T\rhot^2_{({k^\sharp})}.$$
We now use the global lower bound $$H_2+\Lambda H_2=(1-w-w')\ge c_{p,d,r}, \ \ c_{p,d,r}>0$$ to conclude that 
the same bound holds for $\Ht_2$, see \eqref{vneioneinoenvoen}, and 
to estimate using \eqref{esterrorpotentials}, \eqref{veniovnineonelweroidre}:
\bee
&&\int(\Ht_1-{k^\sharp}(\Ht_2+\Lambda \Ht_2))(p-1)\rho_D^{p-2}\rho_T\rhot^2_{({k^\sharp})}\\
& \leq& -{k^\sharp}\int\left[1+O_{{k^\sharp}\to +\infty}\left(\frac{1}{{k^\sharp}}\right)\right](\Ht_2+\Lambda \Ht_2)(p-1)\rho_D^{p-2}\rho_T\rhot^2_{({k^\sharp})}\\
\eee
Next,
\bee
&&\left|\int \left[(\Delta^{K}\rho_D)\Delta \Psit-2\nabla(\Delta^{K}\rho_D)\cdot\nabla \Psit\right](p-1)\rho_D^{p-2}\rho_T\rhot_{({k^\sharp})}\right|\\
& \leq & \mathcal d\int \rho_D^{p-2}\rho_T\rhot_{({k^\sharp})}^2+\frac{C}{\mathcal d}\int\rho_T^{p-2}\rho_T^2\left[\frac{|\pa^2\Psit|^2}{\la Z\ra^{2{k^\sharp}}}+\frac{|\pa\Psit|^2}{\la Z\ra^{2({k^\sharp}+1)}}\right]\\
& \leq & \mathcal d I_{{k^\sharp}}+e^{-c_{{k^\sharp}}\tau}
\eee
and for the nonlinear term after an integration by parts:
$$\left|\int \left[\rhot_{({k^\sharp})}\Delta \Psit-2\nabla\rhot_{({k^\sharp})}\cdot\nabla \Psit\right](p-1)\rho_D^{p-2}\rho_T\rhot_{({k^\sharp})}\right|\lesssim \mathcal d\int \rho_D^{p-2}\rho_T\rhot_{({k^\sharp})}^2.
$$
Integrating by parts and using \eqref{esterrorpotentials}, \eqref{globalbeahvoiru}:
\bee
&&-\int \Ht_2\Lambda \rhot_{({k^\sharp})}\left[(p-1)\rho_D^{p-2}\rho_T\rhot_{({k^\sharp})}\right]+\frac{p-1}{2}\int (p-2)\pa_\tau\rho_D\rho_D^{p-3}\rho_T\rhot_{({k^\sharp})}^2+\frac{p-1}{2}\int \pa_\tau\rho_T\rho_D^{p-2}\rhot_{({k^\sharp})}^2\\
&=& \frac {p-1}2\int \rhot_{({k^\sharp})}^2\left[\nabla \cdot(Z\Ht_2\rho_D^{p-2}\rho_T)+\pa_\tau(\rho_D^{p-2})\rho_T+ \pa_\tau(\rho_T)\rho^{p-2}_D\right]=  O\left(\int \rho_D^{p-2}\rho_T\rhot_{({k^\sharp})}^2\right)
\eee

\noindent\underline{\em $\Psi_{({k^\sharp})}$ terms}. We estimate:
\bee
&&(r-2)\int\rho_T\Psit_{({k^\sharp})}\left[2\nabla\rho_T\cdot\nabla \Psit_{({k^\sharp})}+\rho_T\Delta \Psit_{({k^\sharp})}\right] \\
&=&  -(r-2)\int \Psit_{({k^\sharp})}^2\nabla\cdot(\rho_T\nabla \rho_T)-(r-2)\int \nabla \Psit_{({k^\sharp})}\cdot\nabla(\rho_T^2\Psit_{({k^\sharp})})\\
& = & -(r-2)\int \rho_T^2|\nabla \Psit_{({k^\sharp})}|^2
\eee
and similarly, using \eqref{esterrorpotentials}, \eqref{veniovnineonelweroidre}:
\bee
&&{k^\sharp}\int\rho_T(\Ht_2+\Lambda \Ht_2)\Psit_{({k^\sharp})}\left[2\nabla\rho_T\cdot\nabla \Psit_{({k^\sharp})}+\rho_T\Delta \Psit_{({k^\sharp})}\right]={k^\sharp}\int (\Ht_2+\Lambda \Ht_2)\Psi_{({k^\sharp})}\nabla\cdot(\rhot_T^2\nabla \Psi_{({k^\sharp})})\\
& = & -{k^\sharp}\left[\int |\nabla \Psit_{({k^\sharp})}|^2(\Ht_2+\Lambda \Ht_2)\rho_T^2|\nabla \Psit_{({k^\sharp})}|^2+\int \rho_T^2\Psit_{({k^\sharp})}^2\left(\frac{\nabla \cdot(\rho_T^2\nabla(\Ht_2+\Lambda \Ht_2))}{2\rho^2_T}\right)\right]\\
&=& -{k^\sharp}\int \left[1+O\left(\frac 1{{k^\sharp}}\right)\right] (\Ht_2+\Lambda \Ht_2)\rho_T^2|\nabla \Psit_{({k^\sharp})}|^2+e^{-c_{{k^\sharp}}\tau}
\eee
Then from \eqref{smallglobalboot}:
$$
\left|\int 2\rho_T\nabla \Psit\cdot\nabla \Psit_{({k^\sharp})}(2\nabla \rho_T\cdot\nabla \Psit_{({k^\sharp})})\right|\lesssim \int \rho_T^2|\nabla \Psit_{({k^\sharp})}|^2$$
 and using \eqref{pohozaevbispouet}:
\bee
\left|\int 2\rho_T\nabla \Psit\cdot\nabla \Psit_{({k^\sharp})}(\rho_T\Delta \Psit_{({k^\sharp})})\right|\lesssim \int |\nabla \Psit_{({k^\sharp})}|^2|\pa(\rho_T^2\nabla \Psit)|\lesssim \int \rho_T^2|\nabla\Psit_{({k^\sharp})}|^2.
\eee
Arguing verbatim like in the proof of \eqref{shaprpohoazev}:
$$\left|\int \rho_TH_2\Lambda \Psit_{({k^\sharp})}\left(2\nabla \rho_T\cdot\nabla \Psit_{({k^\sharp})}+\rho_T\Delta \Psit_{({k^\sharp})}\right)\right|\lesssim \int\rho_T^2|\nabla \Psi_{({k^\sharp})}|^2.
$$

\noindent\underline{Remaining terms}. We claim the following exact identities:
\be
\label{globalbeahvoiru}
\frac{\pa_\tau\rho_D+\Lambda\rho_D}{\rho_D}=-\frac{2(r-1)}{p-1}+O\left(\frac{1}{\la Z\ra^\delta}\right)
\ee
and
\be
\label{globalbeahvoirubis}
\frac{\pa_\tau\rho_T+\Lambda \rho_T}{\rho_T}=-\frac{2(r-1)}{p-1}+O\left(\frac{1}{\la Z\ra^\delta}\right)
\ee
which imply the rough bound
$$\left| \frac 12\int\left(\frac{\pa_\tau\rho_T}{\rho_T}+(p-2)\frac{\pa_\tau\rho_D}{\rho_D}\right)(p-1)\rho_D^{p-2}\rho_T\rhot_{({k^\sharp})}^2+\frac 12\int\left(2\frac{\pa_\tau\rho_T}{\rho_T}\right)\rhot_T^2|\nabla \Psi_{({k^\sharp})}|^2\right|\lesssim I_{{k^\sharp}}.$$
\noindent{\em Proof of \eqref{globalbeahvoiru}, \eqref{globalbeahvoirubis}}. From \eqref{fromularhod} and since 
$\lambda=e^{-\tau}$: 
\bee
\pa_\tau\rho_D+\Lambda \rho_D=-\Lambda \zeta(\lambda Z)\rho_P(Z)+\Lambda \zeta(\lambda Z)\rho_P(Z)+\zeta(\l Z)\Lambda \rho_P= \zeta(\l Z)\Lambda \rho_P
\eee
$$\frac{\pa_\tau\rho_D+\Lambda\rho_D}{\rho_D}=\frac{\Lambda \rho_P}{\rho_P}=-\frac{2(r-1)}{p-1}+O\left(\frac{1}{\la Z\ra^{\delta}}\right)
$$
and \eqref{globalbeahvoiru} is proved. We then recall \eqref{renormalizedflow}:
$$\pa_\tau \rho_T=-\rho_T\Delta \Psi_T-\frac{\ell(r-1)}{2}\rho_T-\left(2\pa_Z\Psi_T+ Z\right)\pa_Z\rho_T$$ which yields:
\bee
\left|\frac{\pa_\tau \rho_T+\Lambda \rho_T}{\rho_T}+\frac{\ell(r-1)}{2}\right|=\left|-\Delta \Psi_T-2\frac{\pa_Z\Psi_T\pa_Z\rho_T}{\rho_T}\right|
\eee
and \eqref{globalbeahvoirubis} follows from \eqref{smallglobalboot}.\\

\noindent{\bf step 4} $F_1$ terms. We claim the bound:
\be
\label{estfoneessentialbis}
\int \rho_D^{p-1}F_1^2\lesssim {\mathcal d} I_{{k^\sharp}}+  e^{-c_{{k^\sharp}}\tau}.
\ee

\noindent\underline{\em Source term induced by localization}.  Recall \eqref{profileequationtilde}
 \bee
\tilde{\mathcal E}_{P,\rho}&=&\pa_\tau \rho_D+\rho_D\left[\Delta \Psi_D+\frac{\ell(r-1)}{2}+\left(2\pa_Z\Psi_D+ Z\right)\frac{\pa_Z\rho_D}{\rho_D}\right]\\
& = & \pa_\tau \rho_D+\Lambda \rho_D+\frac{\ell(r-1)}{2}\rho_D+\rho_D\Delta\Psi_D+2\pa_Z\Psi_D\pa_Z\rho_D
\eee
which together with the cancellation \eqref{globalbeahvoiru} which holds with similar proof for higher derivative, and the space localization of $\tilde{\mathcal E}_{P,\rho}$ 
ensures:
\be
\label{neineinneonev}
|\nabla^{{k^\sharp}}\tilde{\mathcal E}_{P,\rho}|\lesssim c_{{k^\sharp}}\frac{\rho_D}{\la Z\ra^{{k^\sharp}+\delta}}{\bf 1}_{Z\ge Z^*}
\ee
for some $\delta>0$. This implies that for ${k^\sharp}$ large enough:
$$
\int \rho_D^{p-2}\rho_T|\Delta^{K}\tilde{\mathcal E}_{P,\rho}|^2\le e^{-c_{{k^\sharp}}\tau}.
$$

\noindent\underline{\em $[\Delta^{K},H_1]$ term}. We use \eqref{veniovnineonelweroidre}, \eqref{esterrorpotentialsbis} to estimate
 $$(p-1)\int \rho_D ^{p-1}([\Delta^{K},H_1]\rhot)^2\lesssim \sum_{j=0}^{k^\sharp -1}\int  \rho_D^{p-1}\frac{|\nabla^j\rhot|^2}{\la Z\ra^{2(r+k-j)}}\leq e^{-c_{{k^\sharp}}\tau}.
$$

\noindent\underline{\em $\mathcal A_{{k^\sharp}}(\rhot)$ term}. From \eqref{estimatepenttniialvone}, \eqref{veniovnineonelweroidre}:
$$(p-1)\int \rho_D^{p-1}(\mathcal A_{{k^\sharp}}(\rhot))^2\lesssim \sum_{j=1}^{{k^\sharp}-1}\int \rho_D^{p-1} \frac{|\nabla^j\rhot|^2}{\la Z\ra^{2(r+{k^\sharp}-j)}}\leq e^{-c_{{k^\sharp}}\tau}
$$
and \eqref{estfoneessentialbis} is proved for this term.\\

\noindent\underline{\em Nonlinear term}. Changing indices, we need to estimate
$$
N_{j_1,j_2}=\nabla^{j_1}\rho_T\nabla^{j_2}\nabla \Psit, \ \ j_1+j_2={k^\sharp}+1, \ \ \left|\begin{array}{l} j_1\le {k^\sharp}\\ j_2\le {k^\sharp}-1
\end{array}\right.
$$
For $j_1\le {k^\sharp}-1$, we may use the pointwise bound \eqref{smallglobalboot} to estimate:
$$|\pa^{j_1}\rho_T\pa^{j_2}\nabla \Psit|\lesssim \rho_D\frac{|\nabla^{j_2}\nabla \Psit|}{\la Z\ra^{j_1}}= \rho_D\frac{|\nabla^{j_2}\nabla \Psit|}{\la Z\ra^{{k^\sharp}+1-j_2}}.$$ Then, after recalling \eqref{veniovnineonelweroidre},
$$\int(p-1) N_{j_1,j_2}^2\rho_D^{p-2}\rho_T\lesssim \int \frac{\rho_T^2|\nabla^{j_2}\nabla \Psit|^2}{\la Z\ra^{2({k^\sharp}+1-j_2)+2(r-1)}}\leq e^{-c_{{k^\sharp}}\tau}
$$
since $j_2\le {k^\sharp}-1$.
For $j_1={k^\sharp}$, $j_2=1$ and hence using \eqref{smallglobalboot}:
\bee
\int(p-1) N_{j_1,j_2}^2\rho_D^{p-2}\rho_T\lesssim \int \frac{\rho_T^2|\nabla^{j_2}\nabla \Psit|^2}{\la Z\ra^{2({k^\sharp}+1-j_2)+2(r-1)}}+ \int \rho_D^{p-1}|\nabla^{{k^\sharp}}\rhot|^2|\nabla^2\Psit|^2\le e^{-c_{{k^\sharp}}\tau}+\mathcal d I_{{k^\sharp}}
\eee
with $\mathcal d$ smallness coming form the bound on $\nabla^2\Psit$. This concludes the proof of \eqref{estfoneessentialbis}.

\noindent{\bf step 5} Dissipation term. We treat the dissipative term in $F_2$:
\bee
{\rm Diss}&=& \int\rho_T^2\nabla(b^2\Delta^{K}\mathcal F)\cdot\nabla \Psit_{({k^\sharp})}=(\mu+\mu')b^2\int \rho_T^2\Delta^{K}\left(\frac{\Delta \nabla\Psi_T}{\rho_T^2}\right)\cdot\nabla \Psit_{({k^\sharp})},
\eee
where we used that in spherical symmetry $\Delta \nabla\Psi_T=\nabla\div \nabla\Psi_T$.
The term with most derivatives falling on $\Psi_T$ is 
\bee(\mu+\mu')b^2\int \rho_T^2\frac{\Delta^{K+1}\nabla\Psi_T}{\rho_T^2}\cdot\nabla \Psit_{({k^\sharp})}&=&(\mu+\mu')b^2\int\left[\Delta^{K+1}(u_D)+\Delta \ut_{({k^\sharp})}\right]\cdot \ut_{({k^\sharp})}\\&\le& -(\mu+\mu')\frac{b^2}{2}\int |\nabla \ut_{({k^\sharp})}|^2+e^{-c_{{k^\sharp}}\tau}.
\eee
By Leibniz, we then need to estimate a generic term  with $k_1+k_2={k^\sharp}$, $k_2\ge 1$ 
$$I_{k_1,k_2}=(\mu+\mu')b^2\int \rho_T^2\nabla^{k_1+2} u_T\nabla^{k_2}\left(\frac 1{\rho_T^2}\right)\cdot\ut_{({k^\sharp})}.
$$
\noindent\underline{Pointwise bound}. We  claim: 
 \be
\label{nboienneoen1}
\left|\nabla^{j_2}\left(\frac{1}{\rho_T}\right)\right|\lesssim \left|\begin{array}{l} \frac{1}{\rho_T\la Z\ra^{j_2}}\ \ \mbox{for}\ \ j_2\le {k^\sharp}-2\\
\frac{1}{\rho_T\la Z\ra^{j_2}}+\frac{|\nabla^{{k^\sharp}-1} \rho_T|}{\rho_T^2}\ \ \mbox{for}\ \ j_2={k^\sharp}-1\\
\frac{1}{\rho_T\la Z\ra^{j_2}}+\frac{|\nabla^{{k^\sharp}-1} \rho_T|}{\la Z\ra \rho_T^2}+\frac{|\nabla^{{k^\sharp}}\rho_T|}{\rho_T^2}\ \ \mbox{for}\ \ j_2={k^\sharp}
 \end{array}
 \right.
 \ee
We estimate from the Faa di Bruno formula, using the pointwise bound \eqref{smallglobalboot} for $j_2\le {k^\sharp}-2$:
\bea
\left|\nabla^{j_2}\left(\frac{1}{\rho_T}\right)\right|&\lesssim&  \frac{1}{\rho_T^{j_2+1}}\sum_{m_1+2m_2+\dots+j_2m_{j_2}=j_2}\Pi_{i=0}^{j_2}|(\nabla^i\rho_T)^{m_i}|\label{eq:faa}\\
& \lesssim &\frac{1}{\rho_D^{j_2+1}}\Pi_{i=0}^{j_2}\left(\frac{\rho_D}{\la Z\ra^i}\right)^{m_i}\lesssim \frac{\rho_D^{j_2}}{\rho_D^{j_2+1}\la Z\ra^{j_2}}\lesssim \frac{1}{\rho_D\la Z\ra^{j_2}}\notag
\eea
where $m_0+m_1+\dots+m_{j_2}=1$.

For $j_2={k^\sharp}-1$, $m_{j_2}\neq 0$ implies $m_{j_2}=1$, $m_1=\dots=m_{j_2-1}=0$, $m_0=j_2-1$ and therefore,
$$
\left|\nabla^{{k^\sharp}-1}\left(\frac{1}{\rho_T}\right)\right|\lesssim \frac{1}{\rho_D\la Z\ra^{{k^\sharp}-1}}+\frac{|\nabla^{{k^\sharp}-1}\rho_T|}{\rho_T^2}.
$$
Similarly,
if $j_2={k^\sharp}$, $m_{j_2}\neq 0$ implies $m_{j_2}=1$, $m_1=\dots=m_{j_2-1}=0$, $m_0=j_2-1$. Also,  if $m_{j_2}=0$ and $m_{j_2-1}\ne 0$ then $m_{j_2-1}=1$,  $m_1=1$ and $m_2=\dots=m_{j_2-2}=0$. Hence
$$
\left|\nabla^{{k^\sharp}}\left(\frac{1}{\rho_T}\right)\right|\lesssim \frac{1}{\rho_D\la Z\ra^{{k^\sharp}}}+\frac{|\nabla^{{k^\sharp}-1} \rho_T|}{\la Z\ra \rho_T^2}+\frac{|\nabla^{{k^\sharp}}\rho_T|}{\rho_T^2}$$
 and \eqref{nboienneoen1} is proved. \\
 
We now estimate $I_{k_1,k_2}$.\\
\noindent\underline{case $k_1={k^\sharp}-1$}. By Leibniz and \eqref{nboienneoen1} 
for $j\le {k^\sharp}-2$:
\be
\label{estimtineti}
\left|\nabla^{j}\left(\frac{1}{\rho^2_T}\right)\right|\lesssim \frac{1}{\rho^2_D\la Z\ra^{j}}.
\ee
This yields:
\bea
|I_{{k^\sharp}-1,1}|&\lesssim& (\mu+\mu')b^2\int \rho_T^2\frac{|\nabla^{{k^\sharp}+1}u_T|}{\la Z\ra \rho_T^2}|\ut_{({k^\sharp})}|\le (\mu+\mu')\frac{b^2}{10}\int |\nabla \ut_{({k^\sharp})}|^2+Cb^2\int \frac{|\ut_{({k^\sharp})}|^2}{\la Z\ra^2}\notag\\
&\le& (\mu+\mu')\frac{b^2}{10}\int |\nabla \ut_{({k^\sharp})}|^2+e^{-c_{{k^\sharp}}\tau}\label{eq:stup}
\eea
where in the last step we used that 
$$
\int \frac{|\ut_{({k^\sharp})}|^2}{\la Z\ra^2}\le \int \la Z\ra^{-2{k^\sharp}+d+2\sigma-\ell(r-1)} \chi_{{k^\sharp}} \rho_T^2\frac{|\ut_{({k^\sharp})}|^2}{\la Z\ra^2}\le
\|\rhot,\Psit\|^2_{{k^\sharp}, \sigma+ {k^\sharp}+1-\frac d2+\frac \ell 2(r-1)-\sigma}\le e^{-c_{{k^\sharp}}\tau},
$$
 since ${k^\sharp}$ is a large parameter $\gg \frac d2$ and $\sigma$ is fixed and small.

\noindent\underline{case $k_1\le {k^\sharp}-2$}.  Since $k_2\ge 1$, we integrate by parts  and use \eqref{boundbootbound}, \eqref{smallglobalboot} to estimate in the case when $k_2\le {k^\sharp}-1$:
\bee
 (\mu+\mu')b^2&&\left|\int \rho_T^2\nabla^{k_1+2} u_T\nabla^{k_2}\left(\frac 1{\rho_T^2}\right)\cdot\ut_{({k^\sharp})}\right|\lesssim b^2\int \left|\nabla^{k_2-1}\left(\frac{1}{\rho_T^2}\right)\right|\\ &&\left[|\nabla\rho_T^2||\nabla^{k_1+2}u_T||\ut_{({k^\sharp})}|+\rho_T^2|\nabla^{k_1+3}u_T||\ut_{({k^\sharp})}|+\rho_T^2|\nabla^{k_1+2}u_T||\nabla\ut_{({k^\sharp})}|\right]\\
&\lesssim& (\mu+\mu') b^2\int \frac{1}{\la Z\ra^{k_2-1}}\left[\frac{|\nabla^{k_1+2}u_T||\ut_{({k^\sharp})}|}{\la Z\ra}+|\nabla^{k_1+3}u_T||\ut_{({k^\sharp})}|+|\nabla^{k_1+2}u_T||\nabla\ut_{({k^\sharp})}|\right]\\
& \leq & (\mu+\mu')\frac{b^2}{10}\int |\nabla \ut_{({k^\sharp})}|^2+e^{-c_{{k^\sharp}}\tau}
\eee
It leaves us with the case ${k^\sharp}=k_2$ and $k_1=0$. We will take the highest order term in \eqref{nboienneoen1}
\bee
(\mu+\mu') b^2&&\left|\int \rho_T^2\nabla^{2} u_T\nabla^{{k^\sharp}}\left(\frac 1{\rho_T^2}\right)\cdot\ut_{({k^\sharp})}\right|\lesssim (\mu+\mu')b^2
\int \left| \nabla^{2} u_T\frac{\nabla^{{k^\sharp}}\rho_T}{\rho_T}\cdot\ut_{({k^\sharp})}\right|\\ &&+ (\mu+\mu')b^2\int\left| \nabla^{2} u_T\frac{\nabla^{{k^\sharp}-1}\rho_T}{\la Z\ra \rho_T}\cdot\ut_{({k^\sharp})}\right|+(\mu+\mu')b^2\int \left| \nabla^{2} u_T\frac{1}{\la Z\ra^{{k^\sharp}}}\cdot\ut_{({k^\sharp})}\right|
\eee
The last term is easily controlled:
$$
(\mu+\mu') b^2\int \left| \nabla^{2} u_T\frac{1}{\la Z\ra^{{k^\sharp}}}\cdot\ut_{({k^\sharp})}\right|\le (\mu+\mu')b^2\int \frac{|\ut_{({k^\sharp})}|^2}{\la Z\ra^{{2{k^\sharp}}+4-2d}}+(\mu+\mu')b^2\le e^{-c_{{k^\sharp}}\tau},
$$
where in the last step we used that ${k^\sharp}$ is large and the line of argument similar to \eqref{eq:stup}.
The most difficult term is
\bee
&& (\mu+\mu') b^2
\int \left| \nabla^{2} u_T\frac{\nabla^{{k^\sharp}}\rho_T}{\rho_T}\cdot\ut_{({k^\sharp})}\right|\\
&\le&  (\mu+\mu') b^2
\int \frac{|\nabla^{{k^\sharp}-1}\rho_T|}{\rho_T}\left[ |\nabla^3 u_T| |\ut_{({k^\sharp})}| + |\nabla^2 u_T| |\nabla\ut_{({k^\sharp})}|
+|\nabla^2 u_T| |\ut_{({k^\sharp})}|\frac{|\nabla\rho_T|}{\rho_T}\right]
\eee
We can estimate 
$$
(\mu+\mu') b^2\int \frac{|\nabla^{{k^\sharp}-1}\rho_T|}{\rho_T} |\nabla^2 u_T| |\nabla\ut_{({k^\sharp})}|\le 
(\mu+\mu')\frac{b^2}{10}\int  |\nabla\ut_{({k^\sharp})}|^2+ C b^2 \int\frac{|\nabla^{{k^\sharp}-1}\rho_T|^2}{\la Z\ra^4 \rho_T^2}
$$
To control the last term we first see that 
$$
b^2 \int\frac{|\nabla^{{k^\sharp}-1}\rho_D|^2}{\la Z\ra^4 \rho_T^2}\lesssim b^2 \int\frac{1}{\la Z\ra^{4 +2({k^\sharp}-1)}}\le b^2
$$
and for the remaining $\rhot$ contribution 
could again use the bootstrap assumptions on the $\|\rhot,\Psit\|$ norm 
$$
b^2 \int\frac{|\nabla^{{k^\sharp}-1}\rhot|^2}{\la Z\ra^4 \rho_T^2}\lesssim b^2 \int \frac 1{\rho_D^{p+1}\chi_{{k^\sharp}}} \rho_D^{p-1}\chi_{{k^\sharp}}
\frac{|\nabla^{{k^\sharp}-1}\rhot|^2}{\la Z\ra^4}\le e^{-c_{{k^\sharp}}\tau}
$$
using that in the expression 
$$
\frac 1{\rho_D^{p+1}\chi_{{k^\sharp}}} 
$$
the dominant factor is $\la Z\ra^{-2{k^\sharp}}$ since ${k^\sharp}$ is chosen to be large. Since $\rho_D^{-1}$ contains a factor of 
$\la Z\ra^{n_P}$, this would however require imposing the condition that ${k^\sharp}\gg n_P$ which is acceptable but not necessary.
We can take a slightly different route and use the estimate \eqref{eq:strrhor} instead:
$$
\int_{Z\ge 12Z^*}\la Z\ra ^{-d+2m} \left\la \frac Z{Z^*}\right\ra^{\mu-2\sigma}
\left|\frac{\nabla^m\rhot}{\rho_D}\right|^2\le \mathcal d
$$
which holds with $\mu=\min\{1,2(r-1)\}$ for any $m\le {k^\sharp}-1$ and $\sigma>0$. Then
$$
b^2 \int_{Z\ge 12 Z^*} \frac{|\nabla^{{k^\sharp}-1}\rhot|^2}{\la Z\ra^4 \rho_T^2}\lesssim b^2
$$
just under the condition that ${k^\sharp}\gg \frac d2$. On the other hand,
\bee
b^2 \int_{Z\le 12 Z^*} \frac{|\nabla^{{k^\sharp}-1}\rhot|^2}{\la Z\ra^4 \rho_T^2}&&\lesssim b^2\int
\la Z\ra^{-2{k^\sharp}+d-\ell(r-1)-2(r-1)+2(r-1)\frac{p+1}{p-1}}\chi_{{k^\sharp}} \rho_D^{p-1} \frac{|\nabla^{{k^\sharp}-1}\rhot|^2}{\la Z\ra^2}
\\ && \lesssim b^2\int
\la Z\ra^{-2{k^\sharp}+d}\chi_{{k^\sharp}} \rho_D^{p-1} \frac{|\nabla^{{k^\sharp}-1}\rhot|^2}{\la Z\ra^2}\le e^{-c_{{k^\sharp}}\tau}.
\eee
The remaining lower order terms can be treated similarly.

\noindent{\bf step 6} $F_2$ terms. We claim: 
\be
\label{estfoneessentialftwobisbis}
\int\rho_T^2|\nabla (F_2-b^2\Delta^{K}\mathcal F+\Delta^{K}\NL(\rhot))|^2 \leq C  I_{{k^\sharp}}+ e^{-c_{{k^\sharp}}\tau}
\ee
for some universal constant $C$ independent of ${k^\sharp}$.
The nonlinear term $\Delta^{K}\NL(\rhot)$ will be treated in the next step.

\noindent\underline{\em Source term induced by localization}. Recall \eqref{profileequationtilde}:
$$\tilde{\mathcal E}_{P,\Psi}=\pa_\tau \Psi_D+\left[|\nabla \Psi_D|^2+\rho_D^{p-1}+(r-2)\Psi_D+\Lambda \Psi_D\right]$$
which yields
$$\pa_Z\tilde{\mathcal E}_{P,\Psi}=\pa_\tau u_D+\left[2u_D\pa_Zu_D+(p-1)\rho_D^{p-1}\pa_Z\rho_D+(r-1)u_D+\Lambda u_D\right].$$ In view of the exact profile equation for $u_P$ and the fact that $u_P$ coincides with $u_D$ for $Z\le Z^*$,
$\pa_Z\tilde{\mathcal E}_{P,\Psi}$ is supported in $Z\ge Z^*$. Furthermore,
from \eqref{definitionprofilewithtailchange}:
$$u_D(\tau,Z)=\zeta(\l Z)u_P(Z)$$ and hence 
\bee
&&\pa_\tau u_D+\Lambda u_D+(r-1)u_D=-\Lambda \zeta(x) u_P(Z)+\Lambda \zeta(x) u_P(Z)+\zeta(x)\Lambda u_P(Z)+(r-1)\zeta(x)u_P(Z)\\
&=& \zeta(x)\left[(r-1)u_P+\Lambda u_P\right](Z)=O\left(\frac{{\bf 1}_{Z\le 10Z^*}}{\la Z\ra^{r-1+\delta}}\right).
\eee
Using that $|u_D|+\rho_D^{\frac{p-1}2}\lesssim \la Z\ra^{-(r-1)}$, with the inequality becoming $\sim$ in the region 
$Z^*\le Z\le 10 Z^*$ and that $u_D$ vanishes for $Z\ge 10 Z^*$, 
we infer 
$$|\pa_Z\tilde{\mathcal E}_{P,\Psi}|\lesssim \frac{{\bf 1}_{Z\ge Z^*}}{\la Z\ra^{r-1+\delta}}$$
with a similar statement holding for higher derivatives
$$|\nabla \nabla^{{k^\sharp}}\tilde{\mathcal E}_{P,\Psi}|\lesssim \frac{{\bf 1}_{Z\ge Z^*}}{\la Z\ra^{{k^\sharp}+r-1+\delta}}$$
 Then, 
 \bee
\int \rho_T^2|\nabla \nabla^{{k^\sharp}}\tilde{\mathcal E}_{P,\Psi}|^2\lesssim \int_{Z\ge Z^*}Z^{d-1} \frac{\rho_T^2}{\la Z\ra^{2{k^\sharp}+2(r-1)+2\delta}}dZ\leq e^{-c_{{k^\sharp}}\tau}
\eee
if ${k^\sharp}\gg \frac d2$ is large enough.

\noindent\underline{\em $ \matchal A_{{k^\sharp}}(\Psi)t$ term}. From \eqref{estimatepenttniialvone}
$$
|\nabla\mathcal A_{{k^\sharp}}(\Psit)|\lesssim \sum_{j=1}^{{k^\sharp}}\frac{|\nabla^j\Psit|}{\la Z\ra^{r+{k^\sharp}-j+1}}$$
and hence from \eqref{veniovnineonelweroidre} :
$$
\int\rho_T^2|\nabla\mathcal A_{{k^\sharp}}(\Psit)|^2\lesssim \sum_{j=0}^{{k^\sharp}-1}\int\rhot_T^2\frac{|\nabla \pa^j\Psit|^2}{\la Z\ra^{2(r+{k^\sharp}-j)+2}}\le  e^{-c_{{k^\sharp}}\tau}.
$$

\noindent\underline{$[\Delta^{K},\rho_D^{p-2}]$ term}. We first claim the bound: let $\alpha\in \Bbb R$ and $\beta\in \Bbb N^d$ with $|\beta|=m$. Then for any $m$
\be
\label{boundnonlienalpha}
\nabla^{\beta}(\rho_D^\alpha)=O_{\alpha,m}\left(\frac{\rho_D^\alpha }{\la Z\ra^m} \right)
\ee
This is proved below. We conclude from \eqref{estimatecommutatorvlkeveln}:
$$\left|[\Delta^{K},\rho_D^{p-2}]\rhot-{k^\sharp}(p-2)\rho_D^{p-3}\nabla\rho_D\cdot\nabla\Delta^{K-1}\rhot\right|\lesssim \sum_{j=0}^{{k^\sharp}-2}\frac{|\nabla^j\rhot|}{\la Z\ra^{{k^\sharp}-j}}\rho_D^{p-2}$$
 and similarly, taking a derivative and using \eqref{veniovnineonelweroidre},
\bee
&&\int \rho_T^2\left|\nabla \left[[\Delta^{K},\rho_D^{p-2}]\rhot-{k^\sharp}(p-2)\rho_D^{p-3}\nabla\rho_D\cdot\nabla\Delta^{K-1}\rhot\right]\right|^2\\
&\lesssim& \sum_{j=0}^{{k^\sharp}-1}\int\rho_D^{2(p-2)+2}\frac{|\nabla^j\rhot|^2}{\la Z\ra^{2({k^\sharp}-j)+2}}=\sum_{j=0}^{{k^\sharp}-1}\int\rho_D^{2(p-1)}\frac{|\nabla^j\rhot|^2}{\la Z\ra^{2({k^\sharp}-j)+2}}\leq e^{-c_{{k^\sharp}}\tau}.
\eee

\noindent{\em Proof of \eqref{boundnonlienalpha}}. Let $g=\rho_D^\alpha$, then $$\frac{\nabla g}{g}=\alpha\frac{\nabla\rho_D}{\rho_D}$$ and \eqref{kevnkneonenoen} yields: $$|\nabla g|\lesssim \frac{|g|}{\la Z\ra }\lesssim \frac{\rho_D^\alpha}{\la Z\ra}.$$ We now prove by induction on $m\ge1$:
\be
\label{boundnonlienarbis}
|\nabla^{m}g|\lesssim \frac{\rho_D^\alpha}{\la Z\ra^{m}}.
\ee
We assume  $m$ and prove $m+1$. Indeed,   
$$|\nabla^{m+1}g|=\left|\alpha \nabla^m\left[g\frac{\pa\rho_D}{\rho_D}\right]\right|\lesssim \sum_{j_1+j_2+j_3=m}|\nabla^{j_1}g|\left|\nabla^{j_2}\left(\frac{1}{\rho_D}\right)\right||\nabla^{j_3+1}\rho_D|.
$$
From \eqref{eq:faa} with $\rho_D$ in place of $\rho_T$:
$$\left|\nabla^{j_2}\left(\frac{1}{\rho_D}\right)\right|\lesssim \frac{1}{\rho_D\la Z\ra^{j_2}}$$ and hence using the induction claim:
 $$
 |\pa^{m+1}g|\lesssim \sum_{j_1+j_2+j_3=m}\frac{\rho^\alpha_D}{\la Z\ra^{j_1}}\frac{1}{\rho_D\la Z\ra ^{j_2}}\frac{\rho_D}{\la Z\ra^{j_3+1}}\lesssim \frac{\rho_D^\alpha}{\la Z\ra^{m+1}}
 $$
 and \eqref{boundnonlienarbis} is proved. 
 This concludes the proof of \eqref{boundnonlienalpha}.\\

\noindent\underline{Nonlinear $\Psi$ term}. Let $$\pa N_{j_1,j_2}=\nabla^{j_1}\nabla\Psi\cdot\nabla^{j_2}\nabla\Psi, \ \ j_1+j_2={k^\sharp}+1, \ \ j_1\leq j_2, \ \ j_1,j_2\ge 1.$$ We have $j_1\le \frac{{k^\sharp}}{2}$ and hence the $L^\infty $ smallness \eqref{smallglobalboot} yields:
$$
\int \rho_T^2|\nabla^{j_1}\nabla\Psi\nabla^{j_2}\nabla\Psi|^2\leq \mathcal d \int \rho_T^2\frac{|\nabla^{j_2}\nabla\Psi|^2}{\la Z\ra^{2({k^\sharp}-j_2)}}\le e^{-c_{{k^\sharp}}\tau}+{\mathcal d}I_{{k^\sharp}}.
$$

\noindent{\bf step 7} Pointwise bound on the nonlinear term. From \eqref{defnlthoth}: $$
\NL(\tilde{\rho})=(\rho_D+\rhot)^{p-1}-\rho_D^{p-1}-(p-1)\rho_D^{p-2}\rhot=\rho_D^{p-1}F\left(\frac{\rhot}{\rho_D}\right), \ \ F(v)=(1+v)^{p-1}-1-(p-1)v$$ which satisfies for $|v|\le \frac 12$:
 $$ |F^{(m)}(v)|\lesssim_m\left|\begin{array}{l} v^2\ \ \mbox{for}\ \ m=0\\ |v| \  \ \mbox{for}\ \ m=1\\ 1\ \ \mbox{for}\ \ m\ge 2.
 \end{array}\right.
 $$
 We claim with $v=\frac{\rhot}{\rho_D}$:
\be
\label{pointwiseboundnltilde}
\nabla \Delta^{K}\NL(\rhot)=F'(v)\rho_D^{p-1}\frac{\nabla \rhot_{({k^\sharp})}}{\rho_D}+O\left(\frac{\mathcal d}{\rho_D} \rho_D^{p-1}\sum_{j=0}^{{k^\sharp}}\frac{|\nabla^{j}\rhot|}{\la Z\ra^{{k^\sharp}+1-j}}\right).
\ee
 Indeed, we expand:
 \bee
 &&\nabla \Delta^{K}\NL(\rhot)= \nabla \Delta^{K}\left[\rho_D^{p-1}F(v)\right]\\
 &=& \rho_D^{p-1}\nabla\Delta ^{K} F(v)+\sum_{j_1+j_2={k^\sharp}+1, j_2\le {k^\sharp}}c_{j_1,j_2}\nabla^{j_1}(\rho_D^{p-1})\nabla^{j_2}F(v)
 \eee
 and claim:
 \be
 \label{veniennenvenevo}
 \rho_D^{p-1}\nabla\Delta ^{K} F(v) =\rho_D^{p-1} F'\frac{\nabla \rhot_{({k^\sharp})}}{\rho_D}+O\left(\frac{\mathcal d}{\rho_D} \rho_D^{p-1}\sum_{j=0}^{{k^\sharp}}\frac{|\nabla^{j}\rhot|}{\la Z\ra^{{k^\sharp}+1-j}}\right)
 \ee
 and 
 \be
 \label{cenvenvneneneo}
 \left|\sum_{j_1+j_2={k^\sharp}+1, j_2\le {k^\sharp}}c_{j_1,j_2}\nabla^{j_1}(\rho_D^{p-1})\nabla^{j_2}F(v)\right|\leq \frac{\mathcal d}{\rho_D} \rho_D^{p-1}\sum_{j=0}^{{k^\sharp}}\frac{|\nabla^{j}\rhot|}{\la Z\ra^{{k^\sharp}+1-j}}
 \ee
 which yield \eqref{pointwiseboundnltilde}.\\
 
 \noindent{\em Proof of \eqref{veniennenvenevo}}. We  recall the general Faa di Bruno formula
 $$\nabla^jF(G(x))= \sum_{m_1+2m_2+\dots+jm_j=j}c_{m_1,\dots,m_j}F^{(m_1+\dots+m_j)}(x)\Pi_{i=1}^j(\nabla^iG(x))^{m_i}.$$
For $j={k^\sharp}+1$ the highest order derivative is $m_{{k^\sharp}+1}=1$, $m_1=\dots=m_{{k^\sharp}}=0$ and hence: 
 \bea
 \label{neinveinenvoenve}
\nonumber  &&\nabla \Delta^{K}F(G(x))=F'(G(x))\nabla \Delta^{K}G(x)\\
 &+& \sum_{m_1+2m_2+\dots+{k^\sharp}m_{{k^\sharp}}={k^\sharp}+1}c_{m_1,\dots,m_{{k^\sharp}}}F^{(m_1+\dots+m_{{k^\sharp}})}(x)\Pi_{j=1}^{{k^\sharp}}(\nabla^jG)^{m_j}.
 \eea
 From Leibniz with $G=\frac{\rhot}{\rho_D}$:
 $$|\nabla^{j}G|\lesssim\sum_{j_1+j_2=j}\frac{|\nabla^{j_1}\rhot|}{\rho_D\la Z\ra^{j_2}}\lesssim\frac{1}{\rho_D} \sum_{j_1=0}^j\frac{|\nabla^{j_1}\rhot|}{\la Z\ra^{j-j_1}}.
 $$
 \noindent\underline{First term}. We compute:
 \bee
 F'(G(x))\nabla \Delta^{K}G(x)&=&F'(G(x))\left[ \frac{\nabla \rhot_{({k^\sharp})}}{\rho_D}+O\left(\frac{1}{\rho_D} \sum_{j=0}^{{k^\sharp}}\frac{|\nabla^{j}\rhot|}{\la Z\ra^{{k^\sharp}+1-j}}\right)\right]\\
 & = & F'(G(x))\frac{\nabla \rhot_{({k^\sharp})}}{\rho_D}+O\left(\frac{\mathcal d}{\rho_D} \sum_{j=0}^{{k^\sharp}}\frac{|\nabla^{j}\rhot|}{\la Z\ra^{{k^\sharp}+1-j}}\right)
 \eee
 with $\mathcal d$-smallness coming from $|F'|\le \frac {\rhot}{\rho_D}\le \mathcal d$.\\

\noindent\underline{Faa di Bruno term \eqref{neinveinenvoenve}}. We distinguish cases.\\
 If $m_{{k^\sharp}}=1$, then $m_{1}=1$ and $m_2=\dots=m_{{k^\sharp}-1}=0$ and therefore
 \bee
 &&|F^{(m_1+\dots+m_{{k^\sharp}})}(v)\Pi_{j=1}^{{k^\sharp}}(\nabla^jG)^{m_j}|=|F''(v)|\nabla^{{k^\sharp}}G||\nabla G|\leq \mathcal d \sum_{j=0}^{{k^\sharp}}\frac{|\nabla^{j}\rhot|}{\la Z\ra \rho_D\la Z\ra^{{k^\sharp}-j}}\\
 &\le& \frac{\mathcal d}{\rho_D} \sum_{j=0}^{{k^\sharp}}\frac{|\nabla^{j}\rhot|}{\la Z\ra^{{k^\sharp}+1-j}}
 \eee
 with $\mathcal d$-smallness coming from the bound for $\nabla G$.

 If $m_{{k^\sharp}}=0$, then all $j$-derivatives are of order $\le {k^\sharp}-1$. If $j\le {k^\sharp}-2$ then $$|\nabla^{j}G|\lesssim \frac{1}{\rho_D} \sum_{j_1=0}^j\frac{|\nabla^{j_1}\rhot|}{\la Z\ra^{j-j_1}}\lesssim \frac{\mathcal d}{\la Z\ra^j}.$$ Now, either there exist $i_0< j_0$ with $m_{i_0}\ge 1, m_{j_0}\ge 1$, or there exist $i_0<{k^\sharp}-2$ with $m_{i_0}\ge 2$. In the either case:
 $$|\Pi_{j=1}^{{k^\sharp}-1}(\nabla^jG)^{m_j}|\lesssim \frac{1}{\la Z\ra^{{k^\sharp}+1}}\Pi_{j=1}^{{k^\sharp}}(\la Z\ra^{j}\pa^jG)^{m_j}\leq \frac{\mathcal d }{\la Z\ra^{{k^\sharp}+1}}\sum_{j=0}^{{k^\sharp}-1}\la Z\ra^j|\nabla^{j}\rhot|
$$

The collection of above bounds concludes the proof of \eqref{veniennenvenevo}.\\

\noindent{\em Proof of \eqref{cenvenvneneneo}}. First
$$\left|\sum_{j_1+j_2={k^\sharp}+1, j_2\le {k^\sharp}}c_{j_1,j_2}\nabla^{j_1}(\rho_D^{p-1})\nabla^{j_2}F(v)\right|\lesssim \rho_D^{p-1}\left|\sum_{j=0}^{{k^\sharp}}\frac{|\nabla^{j}F(v)|}{\la Z\ra^{{k^\sharp}+1-j}}\right|.$$ Let $n\le {k^\sharp}$, then 
\bee
|\nabla^{n}F(v)|&\lesssim&  \sum_{m_1+2m_2+\dots+n m_{n}=n}|F^{(m_1+\dots+m_n)}(v)|\Pi_{j=1}^n|\nabla^jG(x)|^{m_j}\\
& \lesssim & \frac{1}{\la Z\ra^{n}}\sum_{m_1+2m_2+\dots+n m_{n}=n}|F^{(m_1+\dots+m_n)}(v)|\Pi_{j=1}^n|\la Z\ra^j\nabla^jG(x)|^{m_j}.
\eee
 Either $m_n=1$ in which case $m_1=\dots=m_{n-1}=0$ and hence $$|F^{(m_1+\dots+m_n)}(v)|\Pi_{j=1}^n|\la Z\ra^j\nabla^jG(x)|^{m_j}\leq \frac{\mathcal d}{\rho_D} \frac{|\nabla^n\rhot|}{\la Z\ra^{n-j}}
 $$
 or $m_n=0$ and there at least two terms as above:
 $$|F^{(m_1+\dots+m_n)}(v)|\Pi_{j=1}^n|\la Z\ra^j\nabla^jG(x)|^{m_j}\leq \frac{\mathcal d}{\rho_D} \sum_{j=0}^{n}\frac{|\nabla^{j}\rhot|}{\la Z\ra^{n-j}}.
 $$
 Hence, by Leibniz:
 \bee
 &&\left|\sum_{j_1+j_2={k^\sharp}+1, j_2\le {k^\sharp}}c_{j_1,j_2}\nabla^{j_1}(\rho_D^{p-1})\nabla^{j_2}F(v)\right|\lesssim \sum_{j_1+j_2={k^\sharp}+1}\frac{\rho_{D}^{p-1}}{\la Z\ra^{j_1}}\frac{\mathcal d}{\rho_D} \sum_{j=0}^{j_2}\frac{|\nabla^{j}\rhot|}{\la Z\ra^{j_2-j}}\\
 & \lesssim & \frac{\mathcal d}{\rho_D} \rho_D^{p-1}\sum_{j=0}^{{k^\sharp}}\frac{|\nabla^{j}\rhot|}{\la Z\ra^{{k^\sharp}+1-j}}
 \eee
 and  \eqref{cenvenvneneneo} is proved.\\

\noindent{\bf step 8} $\NL(\rhot)$ term. We claim 
\be
\label{enioneigoiheohie}
\mathcal J\equiv \int\rho_T^2\nabla \Delta^{K}\NL(\rhot)\cdot\nabla \Psit_{({k^\sharp})}=\frac{d}{d\tau}\left\{O(\mathcal d I_{{k^\sharp}})\right\}+O\left( e^{-c_{{k^\sharp}}\tau}+\mathcal d I_{{k^\sharp}}\right).
\ee 
Indeed, we inject \eqref{pointwiseboundnltilde} and estimate:
\bee
 \int \rho_T^2\frac{\mathcal d}{\rho_D} \rho_D^{p-1}\sum_{j=0}^{{k^\sharp}}\frac{|\nabla^{j}\rhot|}{\la Z\ra^{{k^\sharp}+1-j}}|\nabla \Psit_{({k^\sharp})}|\le e^{-c_{{k^\sharp}}\tau}+\mathcal d I_{{k^\sharp}}
\eee
Hence
$$\matchal J=\int \rho_T^2F'\left(\frac{\rhot}{\rho_D}\right)\rho_D^{p-1}\frac{\nabla \rhot_{({k^\sharp})}}{\rho_D}\cdot\nabla \Psit_{({k^\sharp})}+O\left(e^{-c_{{k^\sharp}}\tau}+\mathcal d I_{{k^\sharp}}\right).$$ We now integrate by parts:
\bee
&&\int \rho_T^2F'\left(\frac{\rhot}{\rho_D}\right)\rho_D^{p-1}\frac{\nabla \rhot_{({k^\sharp})}}{\rho_D}\cdot\nabla \Psit_{({k^\sharp})}=-\int \rhot_{({k^\sharp})}\nabla \cdot\left(F'\left(\frac{\rhot}{\rho_D}\right)\rho_D^{p-2}\rho_T^2\nabla \Psit_{({k^\sharp})}\right)\\
& = & -\int \rhot_{({k^\sharp})}\left[F'\left(\frac{\rhot}{\rho_D}\right)\rho_D^{p-2}\nabla \cdot(\rho_T^2\nabla \Psit_{({k^\sharp})})+\rho_T^2\nabla \Psit_{({k^\sharp})}\cdot\nabla\left(F'\left(\frac{\rhot}{\rho_D}\right)\rho_D^{p-2}\right)\right].
\eee
We estimate
$$\left|\nabla\left(F'\left(\frac{\rhot}{\rho_D}\right)\rho_D^{p-2}\right)\right|\lesssim \frac{\mathcal d \rho_D^{p-2}}{\la Z\ra}$$
and hence 
$$\left|\int \rhot_{({k^\sharp})}\rho_T^2\nabla \Psit_{({k^\sharp})}\cdot\nabla\left(F'\left(\frac{\rhot}{\rho_D}\right)\rho_D^{p-2}\right)\right|\lesssim \mathcal d \int \frac{\rhot_{({k^\sharp})}|\nabla \Psit_{({k^\sharp})}|\rho_D^{p-1}\rho_D}{\la Z\ra}\le \mathcal d I_{{k^\sharp}}.$$
We now insert \eqref{estqthohrkbisbis} 
\bea
\label{vnopeeopopejpjpejo}
&&-\int \rhot_{({k^\sharp})}F'\left(\frac{\rhot}{\rho_D}\right)\rho_D^{p-2}\nabla \cdot(\rho_T^2\nabla \Psit_{({k^\sharp})})\\
\nonumber & = &\int \rhot_{({k^\sharp})}F'\left(\frac{\rhot}{\rho_D}\right)\rho_D^{p-2}\rho_T\left[ \pa_\tau \rhot_{({k^\sharp})}-(\Ht_1-{k^\sharp}(\Ht_2+\Lambda \Ht_2))\rhot_{({k^\sharp})}+\Ht_2\Lambda \rhot_{({k^\sharp})}\right.\\
\nonumber &+& \left. (\Delta^{K}\rho_T)\Delta \Psit+{k^\sharp}\nabla\rho_T\cdot\nabla \Psit_{({k^\sharp})}+ 2\nabla(\Delta^{K}\rho_T)\cdot\nabla \Psit-  F_1\right]
\eea
and treat all terms in the above identity. The $\pa_\tau \rhot_{({k^\sharp})}$ is integrated by parts in time:
\bee
&&\int \rhot_{({k^\sharp})}F'\left(\frac{\rhot}{\rho_D}\right)\rho_D^{p-2}\rho_T\pa_\tau \rhot_{({k^\sharp})}=\frac12\frac{d}{d\tau}\left\{\int F'\left(\frac{\rhot}{\rho_D}\right)\rho_D^{p-2}\rho_T\rho_{({k^\sharp})}^2\right\}\\
& - & \frac 12\int \rho_{({k^\sharp})}^2\pa_\tau \left(F'\left(\frac{\rhot}{\rho_D}\right)\rho_D^{p-2}\rho_T\right).
\eee
We estimate the boundary term in time $$\left|\int F'\left(\frac{\rhot}{\rho_D}\right)\rho_D^{p-2}\rho_T\rho_{({k^\sharp})}^2\right|\le \mathcal d \int \rho_D^{p-1}\rhot_{({k^\sharp})}^2.$$
Then from \eqref{globalbeahvoirubis}:
$$\left|F'\left(\frac{\rhot}{\rho_D}\right)\pa_\tau(\rho_D^{p-2}\rho_T)\right|\leq \mathcal d \rho_T^{p-1}$$
and using \eqref{exactliearizedflowtilde}, \eqref{smallglobalboot}:
$$\left|\pa_\tau \left[F'\left(\frac{\rhot}{\rho_D}\right)\right]\right|\lesssim \frac{|\pa_\tau\rhot|}{\rho_D}+\frac{\rhot}{\rho_D}\frac{|\pa_\tau\rho_D|}{\rho_D}\le \mathcal d$$
with $\mathcal d$-smallness coming from the pointwise estimates for $\rhot$ and $F'$,
which ensures
$$\left|\int \rho_{({k^\sharp})}^2\pa_\tau \left(F'\left(\frac{\rhot}{\rho_D}\right)\rho_D^{p-2}\rho_T\right)\right|\leq \mathcal d I_{{k^\sharp}}.$$ The remaining terms in \eqref{vnopeeopopejpjpejo} are estimated by brute force. First
$$
\left|\int \rhot_{({k^\sharp})}F'\left(\frac{\rhot}{\rho_D}\right)\rho_D^{p-2}\rho_T\left[-(\Ht_1-{k^\sharp}(\Ht_2+\Lambda \Ht_2))\rhot_{({k^\sharp})}\right]\right|\leq \mathcal d I_{{k^\sharp}}$$
with $\mathcal d$-smallness coming from $F'$.
Integrating by parts,
\bee
&&\left|\int \rhot_{({k^\sharp})}F'\left(\frac{\rhot}{\rho_D}\right)\rho_D^{p-2}\rho_T\left[H_2\Lambda \rhot_{({k^\sharp})}\right]\right|\\
&=&  \left|\frac 12\int \rhot_{({k^\sharp})}^2\left[d\left(F'\left(\frac{\rhot}{\rho_D}\right)\rho_D^{p-2}\rho_T\right)+\Lambda \left(F'\left(\frac{\rhot}{\rho_D}\right)\rho_D^{p-2}\rho_T\right)\right]\right|\\
& \leq & \mathcal d I_{{k^\sharp}}
\eee
with $\mathcal d$-smallness coming from either $F'$ or the pointwise estimates for $\rhot$. Then using \eqref{veniovnineonelweroidre}:
\bee
&&\left|\int \rhot_{({k^\sharp})}F'\left(\frac{\rhot}{\rho_D}\right)\rho_D^{p-2}\rho_T \Delta^{K}\rho_T\Delta \Psit\right|\le \mathcal d \int \rhot_{({k^\sharp})}\rho^{p-1}_T\left[\frac{\rho_T|\Delta \Psit|}{\la Z\ra^{{k^\sharp}}}+|\rhot_{({k^\sharp})}|\right]\\
& \leq & \mathcal d\left[\int \rho^{p-1}_T\rhot_{({k^\sharp})}^2+\int \rho_T^2\frac{|\Delta \Psit|^2}{\la Z\ra^{2{k^\sharp}}}\right]\le \mathcal d  I_{{k^\sharp}}+e^{-c_{{k^\sharp}}\tau}.
\eee
We finally estimate
\bee
\left|\int \rhot_kF'\left(\frac{\rhot}{\rho_D}\right)\rho_D^{p-2}\rho_Tk\nabla\rho_T\cdot\nabla \Psit_k\right|\leq \mathcal d\int \rho_T^{p-1}\rhot_k\frac{\rho_T}{\la Z\ra}|\nabla \Psit_k|\le \mathcal d I_{{k^\sharp}}
\eee
and from \eqref{estfoneessentialbis}:
$$
\left|\int \rhot_kF'\left(\frac{\rhot}{\rho_D}\right)\rho_D^{p-2}\rho_TF_1\right|\leq e^{-c_{{k^\sharp}}\tau}+\mathcal d I_{{k^\sharp}}.$$
The collection of above bounds concludes the proof of \eqref{enioneigoiheohie}.\\

\noindent{\bf step 9} Conclusion for ${k^\sharp}\geq {k^\sharp}(d,r)$ large enough. The collection of above bounds yields, using also \eqref{globalbeahvoiru}, \eqref{globalbeahvoirubis},
 the differential inequality 
\bee
&&\frac 12\frac{d}{d\tau}\left\{I_{{k^\sharp}}(1+O(\mathcal d))\right\}\\
&\leq& -{k^\sharp}\left[1+O\left(\frac{1}{{k^\sharp}}\right)\right]\int( \Ht_2+\Lambda \Ht_2)\left[(p-1) \rho_D^{p-2}\rho_T\rhot_{{({k^\sharp})}}^2+\rho_T^2|\nabla \Psit_{({k^\sharp})}|^2\right]\\
& - & {k^\sharp}\int (p-1)\rho_D\pa_Z (\rho_D^{p-1})\rhot_{({k^\sharp})}\pa_Z \Psit_{({k^\sharp})}+\mathcal d I_{{k^\sharp}}+e^{-c_{{k^\sharp}}\tau}.
\eee
We now recall \eqref{enineinevnveo}: $\exists c_{d,p,r}>0$ such that uniformly $\forall Z\ge 0$,
\bea
\label{enineinevnveo'}
\nonumber 
&&(\Ht_2+\Lambda \Ht_2)\left[ (p-1)\rho_D^{p-2}\rho_T\rhot_{{({k^\sharp})}}^2+\rho_T^2|\nabla \Psit_{({k^\sharp})}|^2\right]+ (p-1)\rho_D\pa_Z (\rho_D^{p-1})\rhot_{({k^\sharp})}\pa_Z \Psit_{({k^\sharp})}\\
& \ge & c_{d,p,r}\left[ (p-1) \rho_D^{p-2}\rho_T\rhot_{{({k^\sharp})}}^2+\rho_T^2|\nabla \Psit_{({k^\sharp})}|^2\right]
\eea
which taking ${k^\sharp}>k^*(d,p)$ yields the pointwise differential inequality:
\bea\label{eq:finalODEinequalityforthehighestenergyNS}
\frac 12\frac{d}{d\tau}\left\{I_{{k^\sharp}}(1+O(\mathcal d))\right\}+\sqrt{{k^\sharp}}I_{{k^\sharp}}\le e^{-c_{{k^\sharp}}\tau}.
\eea 
Integrating in time,
we obtain \eqref{cneioneoneon}.
\end{proof}

%%%%%%%%%%%%%%%%%%%%%%%%%%

\section{The highest energy norm: the Euler case}

%%%%%%%%%%%%%%%%%%%%%%%%%%%

The Euler case in $d=2$ and $d=3$ for $\ell\leq\sqrt 3$ requires special consideration. In those cases, property (P)
of \eqref{P}, which ensures coercivity of the corresponding quadratic form in \eqref{enineinevnveo}, does not hold for $Z>Z_2$. On the other hand, \eqref{coercivityquadrcouplinginside} still gives us the required coercivity for $Z<Z_2$.
To address this we use the energy indentity \eqref{algebracienergyidnentiypouet}
\bea
\label{algebracienergyidnentiy'}
  &&\frac 12\frac{d}{d\tau}\left\{(p-1)\int \chi\rho_D^{p-2}\rho_T\rhot_{({k^\sharp})}^2+\int\chi\rho_T^2|\nabla \Psit_{({k^\sharp})}|^2\right\}\\
\nonumber& = &  \frac 12\int\left(\frac{\pa_\tau\chi}{\chi}+\frac{\pa_\tau\rho_T}{\rho_T}+(p-2)\frac{\pa_\tau\rho_D}{\rho_D}\right)(p-1)\chi\rho_D^{p-2}\rho_T\rhot_{({k^\sharp})}^2+\frac 12\int\left(\frac{\pa_\tau\chi}{\chi}+2\frac{\pa_\tau\rho_T}{\rho_T}\right)\rhot_T^2|\nabla \Psi_{({k^\sharp})}|^2\\
\nonumber& - & \int (p-1)\chi\rho_D^{p-2}\rho_T\rhot_{({k^\sharp})}^2\left[-\Ht_1+k^{\sharp}(\Ht_2+\Lambda\Ht_2)-\frac d2\Ht_2-\frac 12\Lambda \Ht_2-\frac{p-2}{2}\Ht_2\frac{\Lambda\rho_D}{\rho_D}-\frac{\Ht_2}2\frac{\Lambda\chi}{\chi}\right]\\
\nonumber& - & \int \chi\rhot_T^2|\nabla \Psit_{({k^\sharp})}|^2\left[k^{\sharp}(\Ht_2+\Lambda \Ht_2)+r-2-\frac{d-2}{2}\Ht_2+\frac 12\Lamdba \Ht_2-\frac{\Ht_2}2\frac{\Lambda \chi}{\chi}-\Ht_2\frac{\Lambda\rho_T}{\rho_T}\right]\\
\nonumber& &+\int \rhot_{({k^\sharp})}\pa_Z\Psit_{({k^\sharp})}\left[- k^{\sharp}(p-1)\chi\rho_D^{p-2}\rho_T\pa_Z\rho_T\right.\\
\nonumber &- & \left.k^{\sharp}(p-1)(p-2)\chi\rho_T^2\rho_D^{p-3}\pa_Z\rho_D+(p-1)\rho_D^{p-2}\rho^2_T\pa_Z\chi\right]\\
\nonumber & + & \int F_1\chi(p-1)\rho_D^{p-2}\rho_T\rhot_{({k^\sharp})}+\int \chi\rho^2_T\nabla F_2\cdot\nabla \Psit_{({k^\sharp})}\\
\nonumber& + & \nonumber  \int \left[-(\Delta^K\rho_T)\Delta \Psit-2\nabla(\Delta^K\rho_T)\cdot\nabla \Psit\right](p-1)\chi\rho_D^{p-2}\rho_T\rhot_{({k^\sharp})}\\
\nonumber&- & k^{\sharp}(p-1)(p-2)\int \chi\rho_T^2\pa_Z \Psi_{({k^\sharp})}\left[\pa_Z\left(\rho_D^{p-3}\pa_Z\rho_D\pa_Z\Delta^{K-1}\rhot\right)-\rho_D^{p-3}\pa_Z\rho_D\rhot_{({k^\sharp})}\right]\\
\nonumber& + & 2\int\nabla \Psit\cdot\nabla \Psit_{({k^\sharp})}\nabla\cdot(\chi\rho_T^2\nabla \Psit_{({k^\sharp})})-k^{\sharp}\int \chi\rho_T^2\Psit_{({k^\sharp})}\nabla \Psit_{({k^\sharp})}\cdot \nabla(\Ht_2+\Lambda \Ht_2).
\eea
In the previous section we used this energy inequality with $\chi\equiv 1$. This time we first choose 
\be
\chi=\left|\begin{array}{ll}1\ \ \ \ Z\leq Z_2\\
e^{-j^\sharp (Z-Z_2)}\ \ Z>Z_2, \ \ .
\end{array}\right.
\ee
with $j^\sharp\gg k^\sharp$. This guarantees the coercitivity of the quadratic form \eqref{coereuler}:
\bea
\label{coereuler'}
\nonumber 
k^\sharp \chi (\Ht_2&+&\Lambda \Ht_2)\left[ (p-1)\rho_D^{p-2}\rho_T\rhot_{{({k^\sharp})}}^2+\rho_T^2|\nabla \Psit_{({k^\sharp})}|^2\right]+ (p-1)\rho_D\pa_Z (\rho_D^{p-1})\rhot_{({k^\sharp})}\pa_Z \Psit_{({k^\sharp})}
\\&-&\frac{\Lambda \chi}{\chi}\left\{\frac{H_2}{2}\left[(p-1) Q\rhot_{({k^\sharp})}^2+\rho_P^2|\nabla \Psit_{({k^\sharp})}|^2\right]-\frac{1}{Z}(p-1)\rho_P Q\rhot_{({k^\sharp})}\pa_Z\Psit_{({k^\sharp})}\right\}\notag\\
& \ge & c_{d,p,r} k^\sharp \chi\left[ (p-1) \rho_D^{p-2}\rho_T\rhot_{{({k^\sharp})}}^2+\rho_T^2|\nabla \Psit_{({k^\sharp})}|^2\right]
\eea
We then add \eqref{algebracienergyidnentiy'} wirh $\chi=1$, {multiplied by $\delta>0$, recalling that the analog of  \eqref{enineinevnveo}    holds for $Z\leq Z_2$ and for $Z\geq Z(\mathcal{d})$ for $Z(\mathcal{d})$ large enough. The error term estimates are identical to the ones carried out in the proof of Lemma \ref{proproopr}, and we obtain the following analog of \eqref{eq:finalODEinequalityforthehighestenergyNS}
\bee
\frac 12\frac{d}{d\tau}\left\{\big(I_{{k^\sharp}, \chi}+\de I_{{k^\sharp},\chi=1}\big)(1+O(\mathcal d))\right\}+\sqrt{{k^\sharp}}\big(I_{{k^\sharp},\chi}+\de I_{{k^\sharp},\chi=1}\big)\le e^{-c_{{k^\sharp}}\tau}+\de k^\sharp I_{{k^\sharp},\chi(Z_2,Z(\mathcal d))},
\eee 
where $\chi(Z_2,Z(\mathcal d))$ denotes the characteristic function of the set $Z_2<Z<Z(\mathcal{d})$.
We now choose $\de$ such that 
\bee
\delta\ll \frac{1}{\sqrt{k^\sharp}}e^{-j^\sharp(Z(\mathcal{d})-Z_2)}
\eee
which implies 
\bee
\frac 12\frac{d}{d\tau}\left\{\big(I_{{k^\sharp}, \chi}+\de I_{{k^\sharp},\chi=1}\big)(1+O(\mathcal d))\right\}+\sqrt{{k^\sharp}}\big(I_{{k^\sharp},\chi}+\de I_{{k^\sharp},\chi=1}\big)\le e^{-c_{{k^\sharp}}\tau}. 
\eee
This yields the following lemma.}\\

\begin{lemma}[Control of the highest unweighted energy norm]
\label{proproopr'}
For some universal constant $c_{{k^\sharp}}$ ($c_{{k^\sharp}}\ll \delta_g$),
\be
\label{cneioneoneonbis}
(p-1)\int \rho_D^{p-2}\rho_T\rhot_{({k^\sharp})}^2+\int\rho_T^2|\nabla \Psit_{({k^\sharp})}|^2\le e^{-c_{{k^\sharp}}\tau}
\ee
\end{lemma}

%%%%%%%%%%%%%%%%%%%%%%%%%%%%%%%%%%%%%%%%%%%%%%%%%%%

\section{Weighted energy estimates}

%%%%%%%%%%%%%%%%%%%%%%%%%%%%%%%%%%%%%%%%%%%%%%%%%%%

We now rerun the energy estimates with suitable growing weights. This will allow us to close the bound \eqref{boundbootbound}.  Given $\sigma\in \Bbb R$, we recall the notation
$$
\|\rhot,\Psit\|_{m,\sigma}^2=\sum_{k=0}^m\int \la Z\ra^{2k-2\sigma-d+\frac{2(r-1)(p+1)}{p-1}}
\left\la \frac Z{Z^*}\right\ra^{2n_P-\frac{2(r-1)(p+1)}{p-1}+2\sigma}\left[(p-1)\rho_D^{p-2}\rho_T\rhot_k^2+\rho_T^2|\nabla \Psit_k|^2\right]
$$
We let \be
\label{defimsigma}
I_{m,\sigma}=\int \la Z\ra^{2k-2\sigma-d+\frac{2(r-1)(p+1)}{p-1}}\left\la \frac Z{Z^*}\right\ra^{2n_P-\frac{2(r-1)(p+1)}{p-1}+2\sigma}\left[(p-1)\rho_D^{p-2}\rho_T\rhot_m^2+\rho_T^2|\nabla \Psit_m|^2\right]
\ee
and claim:

\begin{lemma}[Weighted energy bounds]
\label{propinduction}
There exists $0<\sigma({k^\sharp})\ll \de_g$ such that {for $\sigma=\sigma({k^\sharp})$,} and $\tau_0\ge \tau_0({k^\sharp})\gg 1$, for all $1\le m\le {k^\sharp}$, $I_{m,\sigma}$ given by \eqref{defimsigma} satisfies the bound for all $\tau\ge\tau_0$
\be
\label{estnienonneoinduction}
I_{m,\sigma}(\tau) \le \d_0 e^{-2\sigma\tau}
\ee
where $\d_0$ is a smallness constant dependent on the data and $\tau_0$.
\end{lemma}

\begin{proof}[Proof of Lemma \ref{propinduction}] The proof is parallel to the one of Lemma \ref{proproopr} with one main difference: exponential decay on the compact set $Z\leq (Z^*)^c$ for $0<c\ll 1$ is provided by Lemma \ref{proproopr}, and we use optimal weight in \eqref{defimsigma} to propagate the {\em sharp} exponential decay. This will be essential to close the scale invariant pointwise bounds \eqref{smallglobalboot}.\\

\noindent{\bf step 1} Equation for derivatives. In this section we use 
$$\pa^k=(\pa_1^k,...,\pa_d^k) $$
and $$ \rhot_k=\pa^k\rhot, \ \ \Psit_k=\pa^k\Psit.$$ We use $$[\pa^k,\Lambda]=k\pa^k$$ to compute from \eqref{exactliearizedflowtilde}:
\bea
\label{estqthohrk}
\nonumber \pa_\tau \rhot_k&=&(\Ht_1-k\Ht_2)\rhot_k-\Ht_2\Lambda \rhot_k-(\pa^k\rho_T)\Delta \Psit-k\pa\rho_T\pa^{k-1}\Delta \Psit-\rho_T\Delta\Psit_k\\
\nonumber &-& 2\nabla(\pa^k\rho_T)\cdot\nabla \Psit-2\nabla \rho_T\cdot\nabla \Psit_k\\
& + & F_1
\eea
with
\bea
\label{formluafone}
F_1&=&-\pa^k\tilde{\mathcal E}_{P,\rho}+[\pa^k,\Ht_1]\rhot-[\pa^k,\Ht_2]\Lambda \rhot\\
\nonumber &-& \sum_{\left|\begin{array}{ll} j_1+j_2=k\\ j_1\geq 2, j_2\geq 1\end{array}\right.}c_{j_1,j_2}\pa^{j_1}\rho_T\pa^{j_2}\Delta\Psit-\sum_{\left|\begin{array}{ll}j_1+j_2=k\\
j_1,j_2\geq 1\end{array}\right.}c_{j_1,j_2}\pa^{j_1}\nabla\rho_T\cdot\pa^{j_2}\nabla\Psit.
\eea
For the second equation:
\bea
\label{nkenononenon}
\nonumber \pa_\tau\Psit_k&=&-k\Ht_2\Psit_k-\Ht_2\Lambda \Psit_k-(r-2)\Psit_k-2\nabla \Psit\cdot\nabla \Psit_k\\
&-& \left[(p-1)\rho_D^{p-2}\rhot_k+k(p-1)(p-2)\rho_D^{p-3}\pa\rho_D\pa^{k-1}\rhot\right]+F_2
\eea
with
\bea
\label{estqthohrkbis}
\nonumber F_2&=&b^2\pa^k\mathcal F -\pa^k\tilde{\mathcal E}_{P,\Psi}- [\pa^k,\Ht_2]\Lambda \Psit-(p-1)\left([\pa^k,\rho_D^{p-2}]\rhot-k(p-2)\rho_D^{p-3}\pa\rho_D\pa^{k-1}\rhot\right)\\
&-& \sum_{j_1+j_2=k,j_1,j_2\geq 1}\pa^{j_1}\nabla\Psit\cdot\pa^{j_2}\nabla\Psit-\pa^k\NL(\rhot).
\eea

\noindent{\bf step 2} Algebraic energy identity. Let $\chi$ be a smooth function $\chi=\chi(\tau,Z)$ 
and compute:
\bee
&&\frac 12\frac{d}{d\tau}\left\{(p-1)\int \chi\rho_D^{p-2}\rho_T\rhot_k^2+\int\chi\rho_T^2|\nabla \Psit_k|^2\right\}\\
&=& \frac 12\left\{(p-1)\int \pa_\tau\chi\rho_D^{p-2}\rho_T\rhot_k^2+\int\pa_\tau\chi\rho_T^2|\nabla \Psit_k|^2\right\}
\\&+&
\frac{p-1}{2}\int \chi(p-2)\pa_\tau\rho_D\rho_D^{p-3}\rho_T\rhot_k^2+ \int\chi\pa_\tau\rho_T\left[\frac{p-1}{2}\rho_D^{p-2}\rhot_k^2+\rho_T|\nabla \Psit_k|^2\right]\\
&+& \int \pa_\tau\rhot_k\left[(p-1)\chi\rho_D^{p-2}\rho_T\rhot_k\right]\\
&-&  \int\pa_\tau\Psit_k\left[2\chi\rho_T\nabla\rho_T\cdot\nabla \Psit_k+\chi\rho_T^2\Delta \Psit_k+\rho_T^2\nabla \chi\cdot\nabla \Psit_k\right],
\eee
\bee
&&\int\pa_\tau \rhot_k\left[(p-1)\chi\rho_D^{p-2}\rho_T\rhot_k\right]=  \int F_1\left[(p-1)\chi\rho_D^{p-2}\rho_T\rhot_k\right]\\
& + & \int \left[(\Ht_1-k\Ht_2)\rhot_k-\Ht_2\Lambda \rhot_k-(\pa^k\rho_T)\Delta \Psit-2\nabla(\pa^k\rho_T)\cdot\nabla \Psit\right]\left[(p-1)\chi\rho_D^{p-2}\rho_T\rhot_k\right]\\
& - & \int k\pa\rho_T\pa^{k-1}\Delta \Psit\left[(p-1)\chi\rho_D^{p-2}\rho_T\rhot_k\right]-  \int (\rho_T\Delta\Psit_k+2\nabla \rho_T\cdot\nabla \Psit_k)\left[(p-1)\chi\rho_D^{p-2}\rho_T\rhot_k\right]
\eee
and
\bee
&-&  \int\pa_\tau\Psit_k\left[2\chi\rho_T\nabla\rho_T\cdot\nabla \Psit_k+\chi\rho_T^2\Delta \Psit_k+\rho_T^2\nabla \chi\cdot\nabla \Psit_k\right]=  -\int F_2\nabla\cdot(\chi\rho_T^2\nabla \Psi_k)\\
&-& \int\Big\{-k\Ht_2\Psit_k-\Ht_2\Lambda \Psit_k-(r-2)\Psit_k-2\nabla \Psit\cdot\nabla \Psit_k\\
&-& \left[(p-1)\rho_P^{p-2}\rhot_k+k(p-1)(p-2)\rho_D^{p-3}\pa\rho_D\pa^{k-1}\rhot\right]\Big\}\left[2\chi\rho_T\nabla\rho_T\cdot\nabla \Psit_k+\chi\rho_T^2\Delta \Psit_k+\rho_T^2\nabla \chi\cdot\nabla \Psit_k\right]\\
& = & \int \chi\rho^2_T\nabla \Psi_k\cdot\nabla F_2\\
&-& \int\left[-k\Ht_2\Psit_k-\Ht_2\Lambda \Psit_k-(r-2)\Psit_k-2\nabla \Psit\cdot\nabla \Psit_k\right]\nabla\cdot(\chi\rho_T^2\nabla \Psi_k)\\
& +& \int (p-1)\rho_P^{p-2}\rhot_k\left[2\chi\rho_T\nabla\rho_T\cdot\nabla \Psit_k+\chi\rho_T^2\Delta \Psit_k+\rho_T^2\nabla \chi\cdot\nabla \Psit_k\right]\\
& + & \int k(p-1)(p-2)\rho_D^{p-3}\pa\rho_D\pa^{k-1}\rhot\nabla\cdot(\chi\rho_T^2\nabla \Psi_k)
\eee
This yields the energy identity:
\bea
\label{algebracienergyidnentiy}
&&\frac 12\frac{d}{d\tau}\left\{(p-1)\int \chi\rho_D^{p-2}\rho_T\rhot_k^2+\int\chi\rho_T^2|\nabla \Psit_k|^2\right\}\\
\nonumber & = & \frac 12\int\left(\frac{\pa_\tau\chi}{\chi}+\frac{\pa_\tau\rho_T}{\rho_T}+(p-2)\frac{\pa_\tau\rho_D}{\rho_D}\right)(p-1)\chi\rho_D^{p-2}\rho_T\rhot_k^2+\frac 12\int\left(\frac{\pa_\tau\chi}{\chi}+2\frac{\pa_\tau\rho_T}{\rho_T}\right)\rhot_T^2|\nabla \Psi_k|^2\\
\nonumber & + & \int F_1\chi(p-1)\rho_D^{p-2}\rho_T\rhot_k+\int \chi\rho^2_T\nabla F_2\cdot\nabla \Psit_k\\
\nonumber  & + & \int \left[(\Ht_1-k\Ht_2)\rhot_k-\Ht_2\Lambda \rhot_k-(\pa^k\rho_T)\Delta \Psit-2\nabla(\pa^k\rho_T)\cdot\nabla \Psit\right](p-1)\chi\rho_D^{p-2}\rho_T\rhot_k\\
\nonumber&-& \int\left[-k\Ht_2\Psit_k-\Ht_2\Lambda \Psit_k-(r-2)\Psit_k-2\nabla \Psit\cdot\nabla \Psit_k\right]\nabla\cdot(\chi\rho_T^2\nabla \Psi_k)\\
\nonumber& - & \int k\pa\rho_T\pa^{k-1}\Delta \Psit\left[(p-1)\chi\rho_D^{p-2}\rho_T\rhot_k\right]+  \int k(p-1)(p-2)\rho_D^{p-3}\pa\rho_D\pa^{k-1}\rhot\nabla\cdot(\chi\rho_T^2\nabla \Psi_k)\\
\nonumber& + &\int(p-1)\rho_P^{p-2}\rhot_k\left[\rho_T\nabla \chi\cdot\nabla \Psit_k\right]
\eea

\noindent{\bf step 3} Bootstrap bound.  We now run \eqref{algebracienergyidnentiy} with 
\be
\label{choihceofchi}
\chi(\tau,Z)=\chi_k=\la Z\ra^{2k-2\sigma-d+\frac{2(r-1)(p+1)}{p-1}}\left\la\frac {Z}{Z^*}\right\ra^{2n_P-\frac {2(r-1)(p+1)}{(p-1)}+2\sigma}, \ \  \ \ 1\leq k\leq {k^\sharp}
\ee
with $Z^*=e^\tau$
 and estimate all terms. We will use the algebra
 $$
 \frac{2(r-1)(p+1)}{p-1}=2(r-1)\left(1+\frac{2}{p-1}\right)=2(r-1)\left(1+\frac{\ell}{2}\right)=(\ell+2)(r-1)=\ell(r-1)+2(r-1)
 $$
 \be
 \label{equivalnetofmrulachigk}
 \chi_k=\la Z\ra^{2k-2\sigma-d+\ell(r-1)+2(r-1)}\left\la\frac {Z}{Z^*}\right\ra^{2n_P-{\ell(r-1)}-2(r-1)+2\sigma}
 \ee
We will use the bound for the damped profile from \eqref{dampenedprofile}, \eqref{definitionprofilewithtailchange}: 
\be
\label{estiamtionnprofile}
|Z^k\pa_Z^k\rho_D|\lesssim \frac{1}{\la Z\ra^{\frac{2(r-1)}{p-1}}}{\bf 1}_{Z\le Z^*}+\frac{1}{(Z^*)^{\frac{2(r-1)}{p-1}}}\frac{1}{\left(\frac{Z}{Z^*}\right)^{n_P}}{\bf 1}_{Z\ge Z^*}
\ee
and $$\frac{|Z^k\pa_Z^k\rho_D|}{\rho_D}\le c_k.$$
In particular,
\bea\notag
\chi_k \rho_D^2&\sim&  \la Z\ra^{2k-2\sigma-d+\ell(r-1)+2(r-1)} \rho_P^2 \left\la\frac {Z}{Z^*}\right\ra^{2\sigma-2(r-1)}\\&\sim& 
 \la Z\ra^{2k-d+2(r-1)}\la Z\ra^{-2\sigma} \left\la\frac {Z}{Z^*}\right\ra^{2\sigma-2(r-1)}\label{pow}
\eea
as well as
\bea\label{power}
\chi_k\le   \frac {\la Z\ra^{2k-d}}{\rho_D^2 \rho_P^{p-1}} \la Z\ra^{-2\sigma}\left\la\frac {Z}{Z^*}\right\ra^{2\sigma-2(r-1)}
\eea
One of our main tools below will be the following interpolation result 
\begin{lemma}\label{lem:inter}
For any $\mathcal a>0$ and any $m\le {k^\sharp}$,
\be
\label{interpolatedbound}
\|\rhot,\Psit\|^2_{m,\sigma+\mathcal a}\leq e^{-c_{\mathcal a,{k^\sharp}}\tau}.
\ee
\end{lemma}
\begin{proof}
For $0\le m\le {k^\sharp}$, on the set $Z\le Z^*_c=(Z^*)^c$, $0<c\ll1$ we control the desired norm by interpolating 
between the bootstrap bound (which controls all the lower Sobolev norms) and the bound \eqref{cneioneoneon} of Lemma
\ref{proproopr} for the highest Sobolev norm. For $Z\ge Z^*_c$ we just use the extra power of $Z$ and the bootstrap  
bound \eqref{boundbootbound} 
\be
\label{inerpo;soationcrucial}
\|\rhot,\Psit\|^2_{m,\sigma+\mathcal a}\lesssim (Z^*)^Ce^{-c_{{k^\sharp}}\tau}+\frac{1}{(Z^*_c)^{2\mathcal a}}\|\rhot,\Psit\|^2_{m,\sigma}\lesssim e^{-c_{\mathcal a, {k^\sharp}}\tau}
\ee
\end{proof}
Unlike the previously dealt with case of the highest Sobolev norms, estimates below do not require us tracking the dependence 
on the parameter $k$. Therefore, we will let $\lesssim$ to include that dependence.\\

\noindent{\bf step 4} Leading order terms.\\

\noindent\underline{\em Cross term}. We estimate the cross term:
\bee
&&k(p-1)\left|\int\chi\pa\rho_T\pa^{k-1}\Delta\Psit\rho_D^{p-2}\rho_T\rhot_k\right|\lesssim c_k\int \chi\frac{\rho_T^{p-1}}{\la Z\ra}|\rho_T\pa^{k-1}\Delta \Psit||\rhot_k|\\
&\lesssim &\int \frac{\chi}{\la Z\ra}\rho_D^{p-1}\rhot_k^2+\int\frac{\chi}{\la Z\ra}\rho_T^2|\nabla \pa^k\Psit|^2\le\|\rhot,\Psit\|^2_{k,\sigma+\frac 12}\lesssim e^{-c_{{k^\sharp}}\tau}.
\eee
The other remaining cross term is estimated using an integration by parts:
\bee
&& k(p-1)(p-2)\left|\int \nabla \cdot(\rho_T^2\nabla \Psi_k)\chi\rho_D^{p-3}\pa\rho_D\pa^{k-1}\rhot\right|\\
&\lesssim &  \int \frac{\chi}{\la Z\ra}\rho_D^{p-1}|\nabla \rhot_{k-1}|^2+ \int \frac{\chi}{\la Z\ra^3}\rho_D^{p-1}\rhot_{k-1}^2+\int\frac{\chi}{\la Z\ra}\rho_T^2|\nabla \Psit_k|^2\le
\|\rhot,\Psit\|_{k,\sigma+\frac 12}^2\\
&\lesssim& e^{-c_{{k^\sharp}}\tau}.
\eee

\noindent\underline{\em $\rho_k$ terms}. We compute using \eqref{esterrorpotentials}:
\bee
&&\int\chi(\Ht_1-k\Ht_2)\rhot_k((p-1)\rho_D^{p-2}\rho_T\rhot_k)\\
\nonumber&\le&  -\int \chi \left(k+\frac{2(r-1)}{p-1}+O\left(\frac{1}{\la Z\ra}\right)\right)(p-1)\rho_D^{p-2}\rho_T\rhot^2_k\\
\nonumber& \lesssim & e^{-c_{{k^\sharp}}\tau}-\int  \chi\left(k+\frac{2(r-1)}{p-1}\right)(p-1)\rho_D^{p-2}\rho_T\rhot^2_k.
\eee
We next recall that by definition of the norm:
\bee
\|\rhot,\Psit\|_{k,\sigma}^2\gtrsim \sum_{m=0}^k\int \chi\frac{\rho_T^2|Z^m\pa^m\nabla \Psi|^2}{\la Z\ra^{2k}}\gtrsim\sum_{m=1}^{k+1}\int\chi\frac{\rho_T^2|Z^m\pa^m \Psi|^2}{\la Z\ra^{2k+2}}
\eee
and hence
\bee
&&\left|\int \chi\left[(\pa^k\rho_D)\Delta \Psit+2\nabla(\pa^k\rho_D)\cdot\nabla \Psit\right](p-1)\rho_D^{p-2}\rho_T\rhot_k\right|\\
& \lesssim & \int \chi\frac{\rho_D^{p-2}\rho_T\rhot_k^2}{\la Z\ra}+\int\chi\rho_T^{p-1}\rho_T^2\left[\frac{|\pa^2\Psit|^2}{\la Z\ra^{2k-1}}+\frac{|\pa\Psit|^2}{\la Z\ra^{2(k+1)-1}}\right]\\
& \leq & \|\rhot,\Psit\|_{k,\sigma+\frac 12}^2\lesssim e^{-c_{{k^\sharp}}\tau}.
\eee

For the nonlinear term, we integrate by parts and use \eqref{smallglobalboot}:
$$\left|\int \chi\left[\rhot_k\Delta \Psit+2\nabla\rhot_k\cdot\nabla \Psit\right](p-1)\rho_D^{p-2}\rho_T\rhot_k\right|\lesssim \int \chi\frac{\rho_D^{p-2}\rho_T\rhot_k^2}{\la Z\ra}\lesssim e^{-c_{{k^\sharp}}\tau}.
$$
Integrating by parts and using \eqref{esterrorpotentials}:
\bee
&&-\int \chi \Ht_2\Lambda \rhot_k\left[(p-1)\rho_D^{p-2}\rho_T\rhot_k\right]+\frac{p-1}{2}\int \chi(p-2)\pa_\tau\rho_D\rho_D^{p-3}\rho_T\rhot_k^2+\frac{p-1}{2}\int \chi\pa_\tau\rho_T\rhot^{p-2}\rhot_k^2\\
&=& \frac {p-1}2\int \rhot_k^2\left[\nabla \cdot(Z\chi \Ht_2\rho_D^{p-2}\rho_T)+\chi\rho_T\pa_\tau(\rho_D^{p-2})+\chi\pa_\tau\rho_T\rho_D^{p-2}\right]\\
& = & \frac{p-1}{2}\int \chi\rho_D^{p-2}\rho_T\rhot_k^2\left[ d+ \frac{\Lambda \chi}{\chi}+(p-2)\left(\frac{\pa_\tau\rho_D+\Lambda \rho_D}{\rho_D}\right)+\frac{\pa_\tau\rho_T+\Lambda \rho_T}{\rho_T}+O\left(\frac{1}{\la Z\ra^r}\right)\right].
\eee
We now estimate from \eqref{globalbeahvoiru}, \eqref{globalbeahvoirubis}:
\bee
&&-\int \chi \Ht_2\Lambda \rhot_k\left[(p-1)\rho_D^{p-2}\rho_T\rhot_k\right]+\frac{p-1}{2}\int \chi(p-2)\pa_\tau\rho_D\rho_D^{p-3}\rho_T\rhot_k^2+\frac{p-1}{2}\int \pa_\tau\rho_T\rhot^{p-2}\rhot_k^2\\
& = &  \frac{p-1}{2}\int \chi\rho_D^{p-2}\rho_T\rhot_k^2\left[d+\frac{\Lambda \chi}{\chi}-2(r-1)+O\left(\frac{1}{\la Z\ra^r}\right)\right]\\
& = & \frac{p-1}{2}\int \chi\rho_D^{p-2}\rho_T\rhot_k^2\left[d+\frac{\Lambda \chi}{\chi}-2(r-1)\right]+O\left(e^{-c_{{k^\sharp}}\tau}\right).
\eee

\noindent\underline{\em $\Psi_k$ terms}. Integrating by parts:
$$(r-2)\int\Psi_k\nabla\cdot(\chi\rho_T^2\nabla \Psi_k)=-(r-2)\int \chi\rho_T^2|\nabla \Psit_k|^2.
$$
Similarly using \eqref{esterrorpotentials}:
\bee
&&k\int \Ht_2\Psit_k\nabla \cdot(\chi\rho_T^2\nabla \Psi_k)=  k\left[\int \chi \Ht_2\Psi_k\nabla\cdot(\rho_T^2\nabla \Psi_k)+\Ht_2\Psi_k\rho_T^2\nabla \chi\cdot\nabla \Psi_k\right]\\
& = & k\left\{-\int\rho_T^2\nabla \Psi_k\cdot\left[\chi \Ht_2\nabla \Psi_k+\Ht_2\Psi_k\nabla \chi+\chi\nabla \Ht_2\Psi_k\right]+\int \Ht_2\Psi_k\rho_T^2\nabla \chi\cdot\nabla \Psi_k\right\}\\
& = & -k\left[\int \chi \left[1+O\left(\frac{1}{\la Z\ra^{r}}\right)\right]\rho_T^2|\nabla \Psit_k|^2+O\left(\int \chi\rho_T^2\frac{|\Psit_k|^2}{\la Z\ra^{r+1}}\right)\right]\\
& = & -k\int \chi \rho_T^2|\nabla \Psit_k|^2+O\left(e^{-c_{{k^\sharp}}\tau}\right),
\eee
where we also used that $r>1$ and $k\ne 0$ (since otherwise the above term vanishes.)
Next, using \eqref{smallglobalboot}:
$$
\left|\int 2\chi\rho_T\nabla \Psit\cdot\nabla \Psit_k(2\nabla \rho_T\cdot\nabla \Psit_k)\right|\lesssim \int \chi\frac{\rho_T^2|\nabla \Psit_k|^2}{\la Z\ra}\lesssim e^{-c_{{k^\sharp}}\tau}$$
 and from \eqref{pohozaevbispouet}, \eqref{smallglobalboot}:
\bee
\left|\int 2\chi\rho_T\nabla \Psit\cdot\nabla \Psit_k(\rho_T\Delta \Psit_k)\right|&\lesssim& \int \chi|\nabla \Psit_k|^2\left(|\pa(\rho_T^2\nabla \Psit)|+\frac{|\rho_T^2\nabla \Psit|}{\la Z\ra}\right)\lesssim \int \chi\frac{\rho_T^2|\nabla \Psit_k|^2}{\la Z\ra^2}\\
&\lesssim &e^{-c_{{k^\sharp}}\tau}
\eee
We now carefully compute from \eqref{pohozaevbispouet} again:
\bea
\label{shaprpohoazev}
\nonumber &&\int \chi\rho_T \Ht_2\Lambda \Psit_k\left(2\nabla \rho_T\cdot\nabla \Psit_k+\rho_T\Delta \Psit_k\right)+\int \Ht_2\Lambda \Psi_k\rho_T^2\nabla \chi\cdot\nabla \Psi_k\\
\nonumber & = & 2\sum_{i,j}\int \chi\rho_T \Ht_2Z_j\pa_j\Psit_k\pa_i\rho_T\pa_i\Psit_k-\sum_{i,j}\int \pa_i(\chi Z_j \Ht_2\rho_T^2)\pa_i\Psit_k\pa_j\Psit_k+\frac 12\int\nabla \cdot(\chi Z \Ht_2\rho_T^2)|\nabla \Psit_k|^2\\
\nonumber & + & \sum_{i,j} \Ht_2\rho_T^2Z_j\pa_j\Psi_k\pa_i\chi\pa_i\Psi_k\\
\nonumber &= & \sum_{i,j}\Ht_2\pa_j\Psi_k\pa_i\Psi_k\left[2\chi\rho_T\pa_i\rho_TZ_j-\pa_i\chi Z_j\rho_T^2-\delta_{ij}\rho_T^2-2Z_j\rho_T\pa_i\rho_T+Z_j\rho_T^2\pa_i\chi\right]\\
\nonumber & + & \frac 12\int\chi \Ht_2\rho_T^2|\nabla \Psi_k|^2\left[d+\frac{\Lambda \chi}{\chi}+\frac{\Lambda \Ht_2}{\Ht_2}+2\frac{\Lambda \rho_T}{\rho_T}+O\left(\frac{1}{\la Z\ra}\right)\right]\\
& = & \frac 12\int \chi\rho_T^2|\nabla \Psi_k|^2\left[d-2+\frac{\Lambda \chi}{\chi}+2\frac{\Lambda \rho_T}{\rho_T}+O\left(\frac{1}{\la Z\ra}\right)\right].
\eea
Finally, recalling \eqref{globalbeahvoirubis}:
\bee
&&\int \chi\rho_T \Ht_2\Lambda \Psit_k\left(2\nabla \rho_T\cdot\nabla \Psit_k+\rho_T\Delta \Psit_k\right)+\int \Ht_2\Lamdba \Psit_k\rho_T^2\nabla \chi\cdot\nabla \Psit_k+\int\chi\pa_\tau\rho_T\rho_T|\nabla \Psit_k|^2\\
& = &  \int  \chi\rho_T^2|\nabla \Psit_k|^2\left[\frac{d-2}{2}+\frac 12\frac{\Lambda\chi}{\chi}+\frac{\Lambda \rho_T}{\rho_T}+\frac{\pa_\tau\rho_T}{\rho_T}+O\left(\frac{1}{\la Z\ra}\right)\right]\\
& = &  \int  \chi\rho_T^2|\nabla \Psit_k|^2\left[\frac{d-2}{2}+\frac 12\frac{\Lambda\chi}{\chi}-\frac{2(r-1)}{p-1}+O\left(\frac{1}{\la Z\ra}\right)\right]\\
& = & \int  \chi\rho_T^2|\nabla \Psit_k|^2\left[\frac{d-2}{2}+\frac 12\frac{\Lambda\chi}{\chi}-\frac{2(r-1)}{p-1}\right]+O(e^{-c_{{k^\sharp}}\tau}).
\eee

\noindent\underline{Conclusion for linear terms}. The collection of above bounds yields:
\bee
&&\frac 12\frac{d}{d\tau}\left\{(p-1)\int \chi\rho_D^{p-2}\rho_T\rhot_k^2+\int\chi\rho_T^2|\nabla \Psit_k|^2\right\}\leq e^{-c_{{k^\sharp}\tau}}\\
\nonumber & + & \int\chi\left[-k+\frac d2-(r-1)-\frac{2(r-1)}{p-1}+\frac12\frac{\pa_\tau\chi+\Lambda \chi}{\chi}\right]\left[(p-1)\rho_D^{p-2}\rho_T\rhot_k^2+\rho_T^2|\nabla \Psit_k|^2\right]\\
\nonumber & + & \int F_1\chi(p-1)\rho_D^{p-2}\rho_T\rhot_k+\int \chi\rho^2_T\nabla F_2\cdot\nabla \Psit_k.
\eee
We now compute from \eqref{choihceofchi}:
\bee
&&-k+\frac d2-(r-1)-\frac{2(r-1)}{p-1}+\frac12\frac{\pa_\tau\chi+\Lambda \chi}{\chi}\\
& = & -k+\frac d2-(r-1)-\frac{2(r-1)}{p-1}+\frac12\left[2k-2\sigma-d+\frac{2(p+1)(r-1)}{p-1}+O\left(\frac{1}{\la Z\ra}\right)\right]\\
&= & -\sigma+O\left(\frac{1}{\la Z\ra}\right)
\eee
and hence the first bound: $\forall m\le {k^\sharp}$
\bea
\label{firstestimate}
&&\frac 12\frac{I_{\sigma,m}}{d\tau}+\sigma  I_{\sigma,m}\leq  e^{-c_{{k^\sharp}\tau}}\\
\nonumber & + & \int F_1\chi(p-1)\rho_D^{p-2}\rho_T\rhot_k+\int \chi\rho^2_T\nabla F_2\cdot\nabla \Psit_k.
\eea

\noindent{\bf step 5} $F_1$ terms. We recall \eqref{formluafone} and claim the bound:
\be
\label{boundfone}
(p-1)\int \chi_k F_1^2\rho_D^{p-2}\rho_T\lesssim e^{-c_{{k^\sharp}}\tau}.
\ee

\noindent\underline{\em Source term induced by localization}.
Recalling \eqref{power}, \eqref{estiamtionnprofile}, \eqref{neineinneonev}:
$$\int \chi \rho_{D}^{p-2}\rho_T|\pa^k\tilde{\mathcal E}_{P,\rho}|^2\lesssim \int_{Z\geq Z^*}\frac{\rho_P^2\la Z\ra^{2k-d}}{\rho_D^2\rho_P^{p+1}}\frac{\rho_D^{p-2}\rho_D\rho_D^2}{\la Z\ra^{2k+2\delta}}Z^{d-1}dZ\lesssim \int_{Z\ge Z^*}\frac{dZ}{\la Z\ra^{2\delta+1}}\lesssim e^{-2\delta\tau}.
$$

\noindent\underline{\em $[\pa^k,\Ht_1]$ term}. From \eqref{esterrorpotentialsbis}, \eqref{interpolatedbound}:
\bee
(p-1)\int \chi_k\rho_D^{p-1}([\pa^k,\Ht_1]\rhot)^2&\lesssim &\sum_{j=0}^{k-1}\int \rho_D^{p-1}\chi_k \frac{|\pa^j\rhot|^2}{\la Z\ra^{2(r+k-j)}}\\
&\lesssim&\|\rhot,\Psit\|^2_{k-1,\sigma+r}\leq  e^{-c_{{k^\sharp}}\tau}.
\eee

\noindent\underline{\em $[\pa^k,\Ht_2]$ term}. We argue similarly using \eqref{esterrorpotentialsbis}:
\be
\label{cneoneoneone}
|[\pa^k,\Ht_2]\Lambda\rhot|\lesssim  \sum_{j=1}^{k}\frac{|\pa^j\rhot|}{\la Z\ra^{r-1+k-j}}\\
\ee
Hence, using $r>1$ and \eqref{interpolatedbound}:
\bee
(p-1)\int \chi\rho_D^{p-1}([\pa^k,\Ht_2]\Lambda \rhot)^2&\lesssim &\sum_{j=1}^{k}\int \chi \rho_D^{p-1} \frac{|\pa^j\rhot|^2}{\la Z\ra^{2(r-1+k-j)}}\\
& \lesssim &   \|\rhot,\Psit\|^2_{k,\sigma+r-1}\leq  e^{-c_{{k^\sharp}}\tau}.
\eee

\noindent\underline{\em Nonlinear term}. $$
N_{j_1,j_2}=\pa^{j_1}\rho_T\nabla^{j_2}\nabla \Psit, \ \ j_1+j_2=k+1, \ \ 2\le j_1\le k\ \ j_2\le k-1
$$
If $j_1\le {k^\sharp}-2$ then we use the pointwise bound \eqref{smallglobalboot} to estimate:
$$|\pa^{j_1}\rho_T\nabla^{j_2}\nabla \Psit|\lesssim \rho_D\frac{|\nabla^{j_2}\nabla \Psit|}{\la Z\ra^{j_1}}= \rho_D\frac{|\pa^{j_2}\nabla \Psit|}{\la Z\ra^{k+1-j_2}}$$ and hence recalling \eqref{interpolatedbound}:
\bee
\int(p-1)\chi N_{j_1,j_2}^2\rho_D^{p-2}\rho_T&\lesssim& \int\chi_k\frac{\rho_T^2|\pa^{j_2}\nabla \Psit|^2}{\la Z\ra^{2(k+1-j_2)+2(r-1)}}\\
& \lesssim &  \int \frac{\chi_{j_2}}{\la Z\ra^{2r}}\rho_T^2|\nabla^{j_2}\nabla \Psit|^2\leq e^{-c_{{k^\sharp}}\tau}
\eee
In the other case, when $j_2\le {k^\sharp}-2$, we use the pointwise bound \eqref{smallglobalboot} for $\nabla\Psi$ instead
and estimate similarly.

\noindent{\bf step 6} Dissipation. [{\it Calculations below and specification $d=3$ are only needed in the Navier-Stokes case}.]\\
We now compute carefully the  dissipation term in \eqref{algebracienergyidnentiy}:
$${\rm Diss}_k=\int \chi_k\rho^2_T\nabla(b^2\pa^k\mathcal F)\cdot\nabla \Psit_k.$$ 
 Indeed, recalling \eqref{defmthjacfl}:
$$\nabla \mathcal F(u_T,\rho_T)=(\mu+\mu')\frac{\Delta u_T}{\rho_T^2}$$ yields
$${\rm Diss}_k=(\mu+\mu') b^2\int \chi_k\rho^2_T\pa^k\left(\frac{\Delta u_T}{\rho_T^2}\right)\cdot \ut_k$$
\noindent\underline{case $k=0$}. We conclude:
$${\rm Diss}_0=(\mu+\mu') b^2\int \chi_0\rho^2_T\frac{\Delta {u_T}}{\rho_T^2}\cdot \ut={(\mu+\mu')}b^2\int \chi_0\Delta u_T\cdot \ut.$$
For $Z\le 10 Z^*$, we use the bootstrap bound $|\la Z\ra^k \pa^ku_{T,D}|\lesssim \frac{1}{\la Z\ra^{r-1}}$ 
as well as that $u_D$ is supported in $Z\le 10 Z^*$ to estimate, recalling \eqref{equivalnetofmrulachigk},\eqref{pow}, 
\bee
\label{enionveinvoenoen}
\nonumber b^2\int\chi_0|\Delta u_T\cdot \ut|&\lesssim &b^2\int\chi_0|\Delta u_T\cdot u_T|+b^2\int\chi_0|\Delta u_T\cdot u_D|\\
&\lesssim&
b^2 \int \la Z\ra^{-2-2\sigma-d+2(r-1)} \left\la \frac Z{Z^*}\right\ra^{-2(r-1)+2\sigma}
\frac{u_T^2+|Z^2\Delta u_T|^2}{\rho_D^2}
 \\ &+& \frac{1}{(Z^*)^{\ell(r-1)+r-2}}\int_{Z\le 10Z^*}\frac{\la Z\ra^{\ell(r-1)}}{\la Z\ra ^{2\sigma+3}}dZ\\&\leq& b^2 \int \la Z\ra^{-2-2\sigma-d+2(r-1)} \left\la \frac Z{Z^*}\right\ra^{-2(r-1)+2\sigma}
\frac{u_T^2+|Z^2\Delta u_T|^2}{\rho_D^2}+e^{-c_e\tau}
\eee
with 
$$
c_e=\min\{\ell(r-1)+r-2,r\}>0
$$
Exactly the same bounds apply to 
$$
b^2\int \chi_0 |\nabla \ut|^2\le  b^2\int \la Z\ra^{-2-2\sigma-d+2(r-1)} \left\la \frac Z{Z^*}\right\ra^{-2(r-1)+2\sigma} 
\frac{u_T^2+|Z^2\Delta u_T|^2}{\rho_D^2}+e^{-c_e\tau}.
$$
Therefore,
\bee
{\rm Diss}_0&\le& -(\mu+\mu') b^2\int\chi_0|\nabla \ut|^2\\ &+&(\mu+\mu') b^2 \int \la Z\ra^{-2-2\sigma-d+2(r-1)} \left\la \frac Z{Z^*}\right\ra^{-2(r-1)+2\sigma} 
\frac{u_T^2+|Z^2\Delta u_T|^2}{\rho_D^2}+e^{-c_e\tau}
\eee

\noindent\underline{case $k\ge 1$}:
\bee
{\rm Diss}_k=(\mu+\mu') b^2\int \chi_k\rho^2_T\pa^k\left(\frac{\Delta u_T}{\rho_T^2}\right)\cdot \ut_k=
(\mu+\mu')\sum_{k_1+k_2=k}b^2\int\chi_k \rho_T^2\pa^{k_1}\Delta u_T\pa^{k_2}\left(\frac1{\rho_T^2}\right)\ut_k
\eee
For $k_2=0$ we decompose $u_T=u_D+\ut$
 and estimate, using that
$u_D$ localized for $Z\le 10 Z^*$
\bee
&&b^2\int \chi_k\pa^{k}{\Delta u_D}\cdot \ut_k\le b^{2\delta}\int \chi_k\rho_T^2|\ut_k|^2\\
&+& b^{4-2\delta}\int_{Z\le 10Z^*}\frac{\la Z\ra^{2k-2\sigma-d+\ell(r-1)+2(r-1)}
}{\la Z\ra^{2(r-1+2+k)}}\la Z\ra^{\ell(r-1)} Z^{d-1}dZ\\
& \le & e^{-c_{{k^\sharp}}\tau}+ b^{4-2\delta}\int _{Z\le 10Z^*}\frac{\la Z\ra^{2\ell(r-1)}dZ}{\la Z\ra^{2\sigma+5}}\\
& \le & e^{-c_{{k^\sharp}}\tau}+ b^{4-2\delta}\lesssim e^{-c_{{k^\sharp}}\tau}
\eee
since the condition 
\bea\label{eq:conditionrminusoneisstriclylessthan2ford=3}
\ell(r-1)<2\ \ \textrm{ holds for }\ \ d=3
\eea 
in view of 
$$
r^*(d,\ell)<r_+(d,\ell)=1+\frac{d-1}{(1+\sqrt \ell)^2}
$$
The main dissipation term is
\bee
&&(\mu+\mu') b^2\int \chi_k\Delta \ut_k\cdot\ut_k=-(\mu+\mu') b^2\left[\int \chi_k|\nabla \ut_k|^2-\frac 12\int \Delta  \chi_k|\ut_k|^2\right]\\
& \leq & -(\mu+\mu') b^2\int \chi_k|\nabla \ut_k|^2+O(b^2)\int \frac{\chi_k}{\la Z\ra^2}|\ut_k|^2\\
&\leq& -{(\mu+\mu')}b^2\int \chi_k|\nabla \ut_k|^2+(\mu+\mu') O(b^2)\sum_{j=0}^{k-1}\chi_j|\nabla \ut_j|^2
\eee
If $1\le k_2\le {k^\sharp}-2$, we estimate from \eqref{estimtineti} and Leibniz (similarly to the above, every term below 
should have a factor of $(\mu+\mu')$, which we suppress):
\bee
&&b^2\int \chi_k\rho^2_T\sum_{k_1+k_2=k,k_1\le k-1}\left|\pa^{k_1}\Delta u_T\pa^{k_2}\left(\frac1{\rho_T^2}\right)\right||\ut_k| \lesssim  b^2\int \chi_k\rho_T^2|\ut_k|\sum_{k_1+k_2=k+1,k_1\le k}\frac{|\nabla \pa^{k_1} u_T|}{\rho_D^2\la Z\ra^{k_2}}\\
& \lesssim & b^2\int \frac{\sqrt{\chi_k}|\ut_k|}{\la Z\ra}\sum_{k_1=0}^k\sqrt{\chi_{k_1}}|\nabla \pa^{k_1}u_T|\leq\frac{b^2}{10}\int \chi_k|\pa^k\nabla u_T|^2+C_kb^2\sum_{j=0}^{k-1}\int\chi_j\left(|\nabla \ut_j|^2+|\nabla \pa^j u_T|^2\right).
\eee

For $k_2={k^\sharp}-1$, $k_1\le 1$, we integrate by parts once and use \eqref{estimtineti} to estimate
\bee
&&b^2\left|\int \chi_k\rho_T^2 \pa^{k_1}\Delta u_T\pa^{k_2}\left(\frac{1}{\rho_T^2}\right)\cdot \tilde{u}_k\right|\\&\lesssim& b^2\int\frac{\chi_k\rho_T^2}{\rho_T^2\la Z\ra^{k_2-1}}\left[|\pa^{k_1+1}\Delta u_T||\ut_k|+\frac{|\pa^{k_1}\Delta u_T||\ut_k|}{\la Z\ra}+|\pa^{k_1}\Delta \ut||\nabla \ut_k|\right]\\
& \leq & \frac{b^2}{10}\int \chi_k|\nabla \ut_k|^2+C_kb^2\sum_{j=0}^{k-1}\chi_j \left(|\nabla \ut_j|^2+|\nabla \pa^ju_T|^2\right).
\eee
For $k_2={k^\sharp}$, $k_1=0$ and $k={k^\sharp}$, we integrate by parts once and use \eqref{estimtineti} to estimate (the highest 
derivative term)
\bee
&&b^2\left|\int  \chi_k\rho_T^2 \Delta u_T\pa^{k_2}\left(\frac{1}{\rho_T^2}\right)\cdot \tilde{u}_k\right|\lesssim b^2\int\frac{\chi_k|\pa^{k-1}\rho_T|}{\rho_T}\left[|\pa\Delta u_T||\ut_k|+\frac{|\Delta u_T||\ut_k|}{\la Z\ra}+|\Delta u_T||\nabla \ut_k|\right].
\eee
(Lower derivative terms are easier to estimate. We omit the details.)
Estimates for the three terms are similar but for the first two we can use estimates from the step $k={k^\sharp}-1$. We therefore
 will only explicitly treat the term
\bee
b^2\int\frac{\chi_k|\pa^{k-1}\rho_T|}{\rho_T} |\Delta u_T||\nabla \ut_k|
\eee
First,
\bee
b^2\int_{Z\le 12Z^*} &&\frac{\chi_k|\pa^{k-1}\rho_T|}{\rho_T} |\Delta u_T||\nabla \ut_k|\lesssim
\frac {b^2}{10} \int \chi_k |\nabla \ut_k|^2 + C b^2 \int_{Z\le 12Z^*} 
\frac{\chi_k|\pa^{k-1}\rho_T|^2}{\la Z\ra^{4+2(r-1)}\rho_T^2}\\ &&\le \frac {b^2}{10} \int \chi_k |\nabla \ut_k|^2 + C b^2 \int_{Z\le 12Z^*} \la Z\ra^{2(r-1)\frac{p+1}{p-1}-2-2(r-1)}\chi_k \rho_D^{p-1}
\frac{|\pa^{k-1}\rho_T|^2}{\la Z\ra^2}\\ &&\le \frac {b^2}{10} \int \chi_k |\nabla \ut_k|^2 + C \frac 1{(Z^*)^{\ell(r-1)+r-2}} \int_{Z\le 12Z^*} \la Z\ra^{\ell(r-1)-2}\chi_k \rho_D^{p-1}
\frac{|\pa^{k-1}\rho_T|^2}{\la Z\ra^2}\\ &&\le\frac {b^2}{10} \int \chi_k |\nabla \ut_k|^2 + C  \int_{Z\le 12Z^*} \la Z\ra^{-r}\chi_k \rho_D^{p-1}
\frac{|\pa^{k-1}\rhot|^2}{\la Z\ra^2} \\ &&+C \frac 1{(Z^*)^{\ell(r-1)+r-2}} \int_{Z\le 12Z^*} \la Z\ra^{\ell(r-1)-2}\chi_k \rho_D^{p-1}
\frac{|\pa^{k-1}\rho_D|^2}{\la Z\ra^2} \\ &&\le \frac {b^2}{10} \int \chi_k |\nabla \ut_k|^2 + e^{-c_{{k^\sharp}}\tau}\\&&+ 
C\frac 1{(Z^*)^{\ell(r-1)+r-2}} \int_{Z\le 12Z^*} \la Z\ra^{\ell(r-1)-2-d+2k+\ell(r-1)-2\sigma}
\frac{Z^{d-1}}{\la Z\ra^{2+2(k-1)+\ell(r-1)}} dZ\\ &&\le\frac {b^2}{10} \int \chi_k |\nabla \ut_k|^2 + e^{-c_{{k^\sharp}}\tau}+C\frac 1{(Z^*)^{\ell(r-1)+r-2}}\int_{Z\le 12Z^*} \frac{\la Z\ra^{\ell(r-1)}dZ}{\la Z\ra^{2\sigma+3}}\\ &&\le\frac {b^2}{10} \int \chi_k |\nabla \ut_k|^2 + e^{-c_{{k^\sharp}}\tau}
\eee
where we used the condition that $\ell(r-1)+r-2>0$, see \eqref{conditione}, as well as \eqref{eq:conditionrminusoneisstriclylessthan2ford=3}.
We now estimate 
$$
b^2\int_{Z\ge 12 Z^*}\frac{\chi_k|\pa^{k-1}\rho_T|}{\rho_T} |\Delta u_T||\nabla \ut_k|
$$
We first decompose $\rho_T=\rho_D+\rhot$.
\bee
b^2\int_{Z\ge 12 Z^*}\frac{\chi_k|\pa^{k-1}\rho_D|}{\rho_T} |\Delta u_T||\nabla \ut_k|&&\lesssim
b^2\int_{Z\ge 12 Z^*}\chi_k \la Z\ra^{-k+1} |\Delta u_T||\nabla \ut_k|\\
&&\le \frac {b^2}{10}\int\chi_k |\nabla \ut_k|^2+ Cb^2\int_{Z\geq 12Z_*}\la Z\ra^{-2k+2} \chi_k|\Delta u_T|^2\\
&&\le \frac {b^2}{10}\int\chi_k |\nabla \ut_k|^2+Cb^2\sum_{j=0}^{k-1} \int \chi_j |\nabla \ut_j|^2
\eee
where we used that $\chi_k=\la Z\ra^{2k-2}\chi_1$ and that $u_D$ is supported in $Z\le 10Z^*$.

Integrating from infinity and using Cauchy-Schwarz,
$$
Z^{d-1}\frac {\chi_2}{Z} |\Delta\ut|^2(Z)\lesssim \int_Z^\infty Z^{d-1} \chi_2 |\nabla \Delta \ut|^2 dZ + \int_Z^\infty Z^{d-1} \frac{\chi_2}{Z^2} |\Delta \ut|^2 dZ \lesssim\int \chi_2 |\nabla \ut_2|^2+\int \chi_1|\nabla \ut_1|^2
$$
Using that $u_D$ is supported in $Z\le 10Z^*$ and that $\chi_k=\la Z\ra^{2k-4} \chi_2$ we then obtain
\bee
b^2\int_{Z\ge 12 Z^*}\frac{\chi_k|\pa^{k-1}\rhot|}{\rho_T} |\Delta u_T||\nabla \ut_k|&&\le \frac {b^2}{10}\int \chi_k |\nabla \ut_k|^2\\&&+ Cb^2 \left(\sum_{j=0}^{k-1} \int \chi_j |\nabla \ut_j|^2\right) \int_{Z\ge 12 Z^*}Z^{-d}\frac{|\la Z\ra^{k-1}\pa^{k-1}\rhot|^2}{\rho_T^2} 
\eee
We now use the estimate \eqref{eq:strrhor}
$$
\int_{Z\ge 12Z^*}\la Z\ra ^{-d+2k} \left\la \frac Z{Z^*}\right\ra^{\mu-2\sigma}
\left|\frac{\nabla^k\rhot}{\rho_D}\right|^2\le \delta
$$
which holds for any $k\le {k^\sharp}-1$ and positive $\mu=\min\{1,2(r-1)\}$, to conclude that
$$
b^2\int_{Z\ge 12 Z^*}\frac{\chi_k|\pa^{k-1}\rhot|}{\rho_T} |\Delta u_T||\nabla \ut_k|\le \frac {b^2}{10}\int\chi_k |\nabla \ut_k|^2+Cb^2\sum_{j=0}^{k-1} \int \chi_j |\nabla \ut_j|^2
$$
We now set
$$
J:=(\mu+\mu') b^2 \int \la Z\ra^{-2-2\sigma-d+2(r-1)} \left\la \frac Z{Z^*}\right\ra^{-2(r-1)+2\sigma} 
\frac{u_T^2+|Z^2\Delta u_T|^2}{\rho_D^2}
$$
Choose a decreasing sequence of positive constants $C_m$ and sum the above inequalities to obtain 
for 
\be\label{eye}
I:= \sum_{m=0}^{{k^\sharp}} C_m I_{m,\sigma}
\ee
\be\label{eq:interm}
\frac 12\frac {dI}{d\tau}  + \sigma  I+
\frac 12 b^2  \sum_{m=0}^{{k^\sharp}} C_m \int \chi_m |\nabla\ut_m|^2\le C J +  \sum_{m=0}^{{k^\sharp}} C_m
\int \chi_m\rho^2_T\nabla \tilde F^m_2\cdot\nabla \Psit_m+ e^{-c_{{k^\sharp}}\tau},
\ee
where 
$$
\tilde F_2=F_2^m-b^2\pa^m\mathcal F^m
$$
denotes the $F_2$ terms minus the contribution from the dissipative terms.

\noindent{\bf step 7} $\tilde F_2$ terms. We claim:
\be
\label{estfoneessentialftwo}
\sum_{m=0}^{{k^\sharp}} C_m
\int \chi_m\rho^2_T\nabla \tilde F^m_2\cdot\nabla \Psit_m\le O( e^{-c_{{k^\sharp}}\tau})-\sigma C_{{k^\sharp}}\mathcal  K - \frac 12 \frac {d}{d\tau} C_{{k^\sharp}}\mathcal K 
\ee
with
$$
|\mathcal K|\le \mathcal d I_{{k^\sharp},\sigma}
$$
\noindent\underline{\em Source term induced by localization}. Recall \eqref{profileequationtilde}:
$$\tilde{\mathcal E}_{P,\Psi}=\pa_\tau \Psi_D+\left[|\nabla \Psi_D|^2+\rho_D^{p-1}+(r-2)\Psi_D+\Lambda \Psi_D\right]$$
which yields
$$\pa_Z\tilde{\mathcal E}_{P,\Psi}=\pa_\tau u_D+\left[2u_D\pa_Zu_D+(p-1)\rho_D^{p-1}\pa_Z\rho_D+(r-1)u_D+\Lambda u_D\right].$$ In view of the exact profile equation for $u_P$ and the fact that $u_P$ coincides with $u_D$ for $Z\le Z^*$,
$\pa_Z\tilde{\mathcal E}_{P,\Psi}$ is supported in $Z\ge Z^*$. Furthermore,
from \eqref{definitionprofilewithtailchange}:
$$u_D(\tau,Z)=\zeta(\l Z)u_P(Z)$$ and hence 
\bee
&&\pa_\tau u_D+\Lambda u_D+(r-1)u_D=-\Lambda \zeta(x) u_P(Z)+\Lambda \zeta(x) u_P(Z)+\zeta(x)\Lambda u_P(Z)+(r-1)\zeta(x)u_P(Z)\\
&=& \zeta(x)\left[(r-1)u_P+\Lambda u_P\right](Z)=O\left(\frac{{\bf 1}_{Z^*\le Z\le 10Z^*}}{Z^{r-1+\delta}}\right).
\eee
Using that $|u_D|+\rho_D^{\frac{p-1}2}\lesssim \la Z\ra^{-(r-1)}$, with the inequality becoming $\sim$ in the region 
$Z^*\le Z\le 10 Z^*$ and that $u_D$ vanishes for $Z\ge 10 Z^*$, 
we infer 
$$|\pa_Z\tilde{\mathcal E}_{P,\Psi}|\lesssim \frac{{\bf 1}_{Z\ge Z^*}}{\la Z\ra^{\delta}}\rho_D^{\frac{p-1}{2}}$$
with a similar statement holding for higher derivatives
$$|\nabla \pa^k\tilde{\mathcal E}_{P,\Psi}|\lesssim \frac{{\bf 1}_{Z\ge Z^*}}{\la Z\ra^{k+\delta}}\rho_D^{\frac{p-1}{2}}$$
 Then, using \eqref{power},
\bee
\int \chi\rho_T^2|\nabla \pa^k\tilde{\mathcal E}_{P,\Psi}|^2\lesssim \int_{Z\ge Z^*}\frac{Z^{d-1} Z^{2k}}{Z^{d}\rho_P^{p-1}\rho_D^2}\frac{\rho_T^2\rho_D^{p-1}}{\la Z\ra^{2k+2\delta}}dZ\leq e^{-2\delta\tau}.
\eee

\noindent\underline{\em $ [\pa^k,\tilde{H}_2]\Lambda \Psi$ term}. From \eqref{esterrorpotentialsbis}:
\bee
|\nabla([\pa^k,\tilde H_2]\Lambda\Psi)|\lesssim \sum_{j=1}^{k+1}\frac{|\pa_j\Psit|}{\la Z\ra^{r+1+k-j}}\lesssim  \sum_{j=0}^{k}\frac{|\nabla \pa^j\Psit|}{\la Z\ra^{r+k-j}}
\eee
and hence:
$$
\int\chi_k\rho_T^2|\nabla([\pa^k,\tilde H_2]\Lambda\Psi)|^2\lesssim \sum_{j=0}^{k}\int\chi_k\rho_T^2\frac{|\nabla \pa^j\Psit|^2}{\la Z\ra^{2(k-j)+2r}}\lesssim e^{-c_{{k^\sharp}}\tau}.
$$

\noindent\underline{$[\pa^k,\rho_D^{p-2}]$ term}. By Leibniz and \eqref{boundnonlienalpha}:
$$\left|\left[[\pa^k,\rho_D^{p-2}]\rhot-k(p-2)\rho_D^{p-3}\pa\rho_D\pa^{k-1}\rhot\right]\right|\lesssim \sum_{j=0}^{k-2}\frac{|\pa^j\rhot|}{\la Z\ra^{k-j}}\rho_D^{p-2}$$
and hence taking a derivative and using \eqref{interpolatedbound}:
\bee
&&\int \chi_k\rho_T^2\left|\nabla \left[[\pa^k,\rho_D^{p-2}]\rhot-k(p-2)\rho_D^{p-3}\pa\rho_D\pa^{k-1}\rhot\right]\right|^2\lesssim\sum_{j=0}^{k-1}\int\chi_k\rho_D^{2(p-2)+2}\frac{|\pa^j\rhot|^2}{\la Z\ra^{2(k-j)+2}}\\
& \lesssim &e^{-c_{{k^\sharp}}\tau}
\eee
since
$2(p-2)+2=2(p-1)>p-1$.

\noindent\underline{Nonlinear $\Psi$ term}. Let $$\pa N_{j_1,j_2}=\pa^{j_1}\nabla\Psi\pa^{j_2}\nabla\Psi, \ \ j_1+j_2=k+1, \ \ j_1\leq j_2, \ \ j_1,j_2\ge 1.$$ We have $j_1\le \frac{{k^\sharp}}{2}$ and hence the $L^\infty $ smallness \eqref{smallglobalboot} yields:
\bee
\int \chi_k\rho_T^2|\pa^{j_1}\nabla\Psi\pa^{j_2}\nabla\Psi|^2&\leq &\mathcal d \int_{Z\le Z^*}\chi_k \rho_T^2\frac{|\pa^{j_2}\nabla\Psi|^2}{\la Z\ra^{2(k-j_2)+2(r-1)}}\\ &+& \mathcal d e^{-2(r-1)\tau}\int_{Z\ge Z^*}\chi_k \rho_T^2\frac{|\pa^{j_2}\nabla\Psi|^2}{\la Z\ra^{2(k-j_2)}}\\ &\le& \|\rhot,\Psit\|^2_{j_2,\sigma+r-1}+e^{-2(r-1)\tau}\|\rhot,\Psit\|^2_{j_2,\sigma}\le e^{-c_{{k^\sharp}}\tau}.
\eee

\noindent{\bf step 8} $\NL(\rhot)$ term. Arguing as for the proof of \eqref{pointwiseboundnltilde} yields:
\be
\label{vnienvonnenovlveporke}
\nabla \pa^{k}\NL(\rhot)=F'\left(\frac{\rhot}{\rho_D}\right)\rho_D^{p-1}\frac{\nabla \rhot_k}{\rho_D}+O\left(\frac{\delta}{\rho_D} \rho_D^{p-1}\sum_{j=0}^{k}\frac{|\pa^{j}\rhot|}{\la Z\ra^{k+1-j}}\right).
\ee
We recall that
$$
F(v)=(1+v)^{p-1}-1-(p-1) v,\qquad F'(v)=(p-1)\left((1+v)^{p-2}-1\right)
$$

We need to estimate going back to \eqref{algebracienergyidnentiy}
$$\mathcal J_k=-\int \chi\rho^2_T\nabla \pa^{k}\NL(\rhot)\cdot\nabla \Psit_k$$ and claim 
\be
\label{vevnenveineneoinvoe}
|\mathcal J_{k}|\leq e^{-c_{{k^\sharp}}\tau}\ \ \mbox{for}\ \ k\le {k^\sharp}-1
\ee
and for $k={k^\sharp}$:
\bea
\label{nnveneonoenvenvon}
\nonumber \mathcal J_{k}&\le & -\frac12\frac{d}{d\tau}\left\{\int \chi_kF'\left(\frac{\rhot}{\rho_D}\right)\rho_D^{p-2}\rho_T\rho_k^2\right\}-\sigma \int \chi_kF'\left(\frac{\rhot}{\rho_D}\right)\rho_D^{p-2}\rho_T\rho_k^2\\
&+& O(e^{-c_{{k^\sharp}}\tau}).
\eea
Indeed, we estimate:
\bee
&&\mathcal J_k=-\int \chi_k\rho^2_T\left[F'\left(\frac{\rhot}{\rho_D}\right)\rho_D^{p-1}\frac{\nabla \rhot_k}{\rho_D}+O\left(\frac{\delta}{\rho_D} \rho_D^{p-1}\sum_{j=0}^{k}\frac{|\pa^{j}\rhot|}{\la Z\ra^{k+1-j}}\right)\right]\cdot\nabla \Psit_k\\
& = & -\int \chi_k\rho_T^2F'\left(\frac{\rhot}{\rho_D}\right)\rho_D^{p-1}\frac{\nabla \rhot_k}{\rho_D}\cdot\nabla \Psit_k+O(e^{-c_{{k^\sharp}}\tau})
\eee
and we now integrate by parts:
\bee
&&-\int \chi_k\rho_T^2F'\left(\frac{\rhot}{\rho_D}\right)\rho_D^{p-1}\frac{\nabla \rhot_k}{\rho_D}\cdot\nabla \Psit_k=\int \rhot_k\nabla \cdot\left(\chi_kF'\left(\frac{\rhot}{\rho_D}\right)\rho_D^{p-2}\rho_T^2\nabla \Psit_k\right)\\
& = & \int \rhot_k\left[\chi_kF'\left(\frac{\rhot}{\rho_D}\right)\rho_D^{p-2}\nabla \cdot(\rho_T^2\nabla \Psit_k)+\rho_T^2\nabla \Psit_k\cdot\nabla\left(\chi_kF'\left(\frac{\rhot}{\rho_D}\right)\rho_D^{p-2}\right)\right].
\eee
We estimate
$$\left|\nabla\left(F'\left(\frac{\rhot}{\rho_D}\right)\rho_D^{p-2}\right)\right|\lesssim \frac{\delta \rho_D^{p-2}}{\la Z\ra}$$
and hence 
$$\left|\int \rhot_k\rho_T^2\nabla \Psit_k\cdot\nabla\left(\chi_kF'\left(\frac{\rhot}{\rho_D}\right)\rho_D^{p-2}\right)\right|\lesssim \delta \int \frac{\chi_k\rhot_k|\nabla \Psit_k|\rho_D^{p-1}\rho_D}{\la Z\ra}\le e^{-c_{{k^\sharp}}\tau}.$$ For $k\le {k^\sharp}-1$, we estimate directly
\bee
&&\left|\int \rhot_k\chi_kF'\left(\frac{\rhot}{\rho_D}\right)\rho_D^{p-2}\nabla \cdot(\rho_T^2\nabla \Psit_k)\right|\lesssim \int \chi_k|\rhot_k|\rho_D^{p-2}\left[\frac{|\rho^2_D|}{\la Z\ra}|\nabla \Psit_k|+\rho_D^2|\Delta \Psit_k|\right]\\
& \lesssim & I_{{k^\sharp},\sigma+1}\leq e^{-c_{{k^\sharp}}\tau}
\eee
and \eqref{vevnenveineneoinvoe} is proved. We now let $k={k^\sharp}$ and insert \eqref{estqthohrkbisbis} 
\bee
&&\int \chi_k\rhot_kF'\left(\frac{\rhot}{\rho_D}\right)\rho_D^{p-2}\nabla \cdot(\rho_T^2\nabla \Psit_k)\\
\nonumber & = -&\int \chi_k\rhot_kF'\left(\frac{\rhot}{\rho_D}\right)\rho_D^{p-2}\rho_T\left[ \pa_\tau \rhot_k-(\Ht_1-k(\Ht_2+\Lambda \Ht_2)\rhot_k+\Ht_2\Lambda \rhot_k\right.\\
\nonumber &+& \left. (\Delta^{K}\rho_T)\Delta \Psit+k\nabla\rho_T\cdot\nabla \Psit_k+ 2\nabla(\Delta^{K}\rho_T)\cdot\nabla \Psit-  F_1\right]
\eee
and treat all terms in the above identity. The $\pa_\tau \rhot_k$ is integrated by parts in time:
\bee
&&-\int \chi_k\rhot_kF'\left(\frac{\rhot}{\rho_D}\right)\rho_D^{p-2}\rho_T\pa_\tau \rhot_k=-\frac12\frac{d}{d\tau}\left\{\int \chi_kF'\left(\frac{\rhot}{\rho_D}\right)\rho_D^{p-2}\rho_T\rho_k^2\right\}\\
& +& \frac 12\int \rhot_k^2\pa_\tau \left(\chi_k F'\left(\frac{\rhot}{\rho_D}\right)\rho_D^{p-2}\rho_T\right).
\eee
We now recall the identity 
$$\int \rhot_k G\Lambda \rhot_k=-\frac 12\int \rhot_k^2(dG +\Lambda G)$$ Therefore,
\bee
&&\frac 12\int \rhot_k^2\pa_\tau \left(\chi_k F'\left(\frac{\rhot}{\rho_D}\right)\rho_D^{p-2}\rho_T\right)\\
&-& \int \chi_k\rhot_kF'\left(\frac{\rhot}{\rho_D}\right)\rho_D^{p-2}\rho_T\left[-[\Ht_1-k(\Ht_2+\Lambda \Ht_2)]\rhot_k+\Ht_2\Lambda \rhot_k\right]=  \int \rhot_k^2 \mathcal A
\eee
with
\bee
\mathcal A&=&\frac 12\pa_\tau \left(\chi_k F'\left(\frac{\rhot}{\rho_D}\right)\rho_D^{p-2}\rho_T\right)+\chi_k F'\left(\frac{\rhot}{\rho_D}\right)\rho_D^{p-2}\rho_T\left[\Ht_1-k\Ht_2-k\Lambda \Ht_2\right]\\
& + & \frac d2\chi_kF'\left(\frac{\rhot}{\rho_D}\right)\rho_D^{p-2}\rho_T+\frac 12\Lambda  \left(\chi_k \Ht_2F'\left(\frac{\rhot}{\rho_D}\right)\rho_D^{p-2}\rho_T\right)
\eee
\noindent\underline{leading order term}. 
We claim
\be
\label{boundpotienntial}
\mathcal A\le \chi_k F'\left(\frac{\rhot}{\rho_D}\right)\rho_D^{p-2}\rho_T\left[-\sigma+O\left(\frac{1}{\la Z\ra^r}\right)\right]
\ee
which ensures 
\bee
&&-\int \chi_k\rhot_kF'\left(\frac{\rhot}{\rho_D}\right)\rho_D^{p-2}\rho_T\pa_\tau \rhot_k\\
&\le  & -\frac12\frac{d}{d\tau}\left\{\int \chi_kF'\left(\frac{\rhot}{\rho_D}\right)\rho_D^{p-2}\rho_T\rho_k^2\right\}-\sigma \int \chi_kF'\left(\frac{\rhot}{\rho_D}\right)\rho_D^{p-2}\rho_T\rho_k^2+O(e^{-c_{{k^\sharp}}\tau}).
\eee
Therefore, see \eqref{estfoneessentialftwo},
$$
\mathcal K=-\int \chi_kF'\left(\frac{\rhot}{\rho_D}\right)\rho_D^{p-2}\rho_T\rhot_k^2\
$$
and 
\be
\label{venivnioenovlmevmepmenev}
|\mathcal K|\le \delta \int\chi_k \rho_D^{p-1}\rhot_k^2.
\ee
\noindent{\em Proof of \eqref{boundpotienntial}}: First
$$\left|F'\left(\frac{\rhot}{\rho_D}\right)\Lambda \Ht_2\right|\lesssim \frac{C}{\la Z\ra^{r}}.$$
 Then from \eqref{choihceofchi}, \eqref{globalbeahvoiru}, \eqref{globalbeahvoirubis}, \eqref{esterrorpotentials}:
 \bee
&&\frac 12\frac{\pa_\tau(\rho_D^{p-2}\rho_T)+\Ht_2\Lambda (\rho_D^{p-2}\rho_T)}{\rho_D^{p-2}\rho_T}+\Ht_1-k\Ht_2+\frac d2+\frac 12\frac{\pa_\tau\chi_k+\Lambda \chi_k\Ht_2}{\chi_k}\\
& = & \frac12\left[(p-2)\left(-\frac{2(r-1)}{p-1}\right)-\frac{2(r-1)}{p-1}\right]-\frac{2(r-1)}{p-1}-k+
\frac d2\\
& + & \frac{2k-2\sigma-d+\frac{4(r-1)}{p-1}+2(r-1)}{2}+O\left(\frac{1}{\la Z\ra^{r}}\right)= -\sigma+O\left(\frac{1}{\la Z\ra^{r}}\right).
\eee
We then estimate from \eqref{exactliearizedflowtilde}, \eqref{smallglobalboot}, \eqref{esterrorpotentials}, \eqref{neineinneonev}:
\bee
&&\left(\pa_\tau+\Lambda\right)\left[F'\left(\frac{\rhot}{\rho_D}\right)\right]=F''\left(\frac{\rhot}{\rho_D}\right)\left\{\frac{(\pa_\tau+\Lambda) \rhot}{\rho_D}-\frac{\rhot}{\rho_D}\frac{(\pa_\tau+\Lambda) \rho_D}{\rho_D}\right\}\\
& = & F''\left(\frac{\rhot}{\rho_D}\right)\left\{-\frac{2(r-1)}{p-1}+\frac{2(r-1)}{p-1}+O\left(\frac{1}{\la Z\ra^{r}}\right)\right\}=O\left(\frac{1}{\la Z\ra^{r}}\right).
\eee
and \eqref{boundpotienntial} is proved. \\
\noindent\underline{lower order terms}. Using \eqref{veniovnineonelweroidre}, \eqref{smallglobalboot}:
\bee
&&\left|\int\chi_k \rhot_kF'\left(\frac{\rhot}{\rho_D}\right)\rho_D^{p-2}\rho_T \Delta^{K}\rho_T\Delta \Psit\right|\le \delta \int \chi_k\rhot_k\rho^{p-1}_T|\Delta \Psit|\left[\frac{\rho_D}{\la Z\ra^{{k^\sharp}}}+|\rhot_k|\right]\\
& \leq & e^{-c_{{k^\sharp}}\tau}.
\eee
The term 
$$
\left|\int \chi_k\rhot_kF'\left(\frac{\rhot}{\rho_D}\right)\rho_D^{p-2}\rho_T\nabla\Delta^{K}\rho_T\cdot\nabla \Psit\right|
$$
is treated similarly after integrating by parts once.
Furthermore,
\bee
\left|\int \chi_k\rhot_kF'\left(\frac{\rhot}{\rho_D}\right)\rho_D^{p-2}\rho_T\nabla\rho_T\cdot\nabla \Psit_k\right|\leq \delta\int \chi_k\rho_T^{p-1}\rhot_k\frac{\rho_T}{\la Z\ra}|\nabla \Psit_k|\le e^{-c_{{k^\sharp}}\tau}
\eee
Finally, from \eqref{boundfone}:
$$
\left|\int \chi_k\rhot_kF'\left(\frac{\rhot}{\rho_D}\right)\rho_D^{p-2}\rho_TF_1\right|\leq e^{-c_{{k^\sharp}}\tau}.$$
The collection of above bounds concludes the proof of \eqref{nnveneonoenvenvon}.\\

\noindent{\bf step 9} Conclusion. 
Going back to \eqref{eq:interm} we obtain
\be\label{eq:bull}
\frac 12\frac {d(I+C_{{k^\sharp}}\mathcal K)}{d\tau}  + \sigma  (I+C_{{k^\sharp}}\mathcal K)+
\frac 12 (\mu+\mu') b^2  \sum_{m=0}^{{k^\sharp}} C_m \int \chi_m |\nabla\ut_m|^2\le C J + e^{-c_{{k^\sharp}}\tau}.
\ee
We integrate in time and use 
\eqref{venivnioenovlmevmepmenev} to obtain for $\sigma<c_{{k^\sharp}}$
$$
I(\tau)\le e^{-2\sigma(\tau-\tau_0)}I(\tau_0) + C e^{-2\sigma\tau} \int_{\tau_0}^\tau e^{2\sigma\tau'} J + e^{-c_{{k^\sharp}}\tau}
$$
We now recall \eqref{eq:stup}, choose a small constant $\mathcal b>0$ (which will depend only on the constants $r$ and $n_P$,)
 let $Z^*_{\mathcal b}=(Z^*)^{1+\mathcal b}$ 
and estimate 
\bee
\int_{\tau_0}^\tau e^{2\sigma\tau'} J &=& (\mu+\mu')\int_{\tau_0}^\tau b^2 (Z^*)^{2\sigma}\int \la Z\ra^{2(r-1)-2-2\sigma-d} \left\la\frac{Z}
{Z^*}\right\ra^{2\sigma-2(r-1)}
\frac{u_T^2+(Z^2\Delta u_T)^2}{\rho_D^2}\\ &=&(\mu+\mu')\int_{\tau_0}^\tau b^2   (Z^*)^{2\sigma}\int_{Z\le Z^*_{\mathcal b}} \la Z\ra^{2(r-1)-2-2\sigma-d} \left\la\frac{Z}
{Z^*}\right\ra^{2\sigma-2(r-1)}
\frac{u_T^2+(Z^2\Delta u_T)^2}{\rho_D^2}\\ &+&(\mu+\mu')\int_{\tau_0}^\tau b^2 (Z^*)^{2(r-1)}\int_{Z\ge Z^*_{\mathcal b}} 
\la Z\ra^{-2-d} 
\frac{u_T^2+(Z^2\Delta u_T)^2}{\rho_D^2}
\eee
We first obtain
\bee
(\mu+\mu')\int_{\tau_0}^\tau &&b^2 (Z^*)^{2\sigma}\int_{Z\le Z^*_{\mathcal b}} \la Z\ra^{2(r-1)-2-2\sigma-d} \left\la\frac{Z}
{Z^*}\right\ra^{2\sigma-2(r-1)}
\frac{u_T^2+(Z^2\Delta u_T)^2}{\rho_D^2}\\ &\lesssim&(\mu+\mu')\int_{\tau_0}^\tau (Z^*)^{-\ell(r-1) -r+2+2\sigma}\int_{Z\le Z^*_{\mathcal b}} \la Z\ra^{-2-2\sigma-1} \la Z\ra^{\ell(r-1)+2\mathcal b n_P}dZ\\ &&\lesssim e^{-(r-2\mathcal bn_P)\tau_0}+e^{-(\ell(r-1) +r-2-2\sigma)\tau_0}\le e^{-\delta\tau_0}
\eee
as long as $\mathcal b$ has been chosen small enough, so that $r\gg \mathcal b n_P$ and $\sigma$ is small enough so that
$2\sigma<\ell(r-1)+r-2$.

To control the second integral we use the global bound \eqref{eq:globald}
\bee
(\mu+\mu')\int_{\tau_0}^\tau&& b^2 (Z^*)^{2(r-1)}\int_{Z\ge Z^*_{\mathcal b}} \la Z\ra^{-2-d} 
\frac{u_T^2+(Z^2\Delta u_T)^2}{\rho_D^2}\\&=&(\mu+\mu')\int_{\tau_0}^\tau b^2 (Z^*)^{-d+ 2r}\int_{Z\ge Z^*_{\mathcal b}} \left(\frac {Z^*}{Z}\right)^{d-2} 
\frac{u_T^2+(Z^2\Delta u_T)^2}{\la Z\ra^{4}\rho_D^2}\\ &\le& (\mu+\mu')e^{-\mathcal b(d-2)\tau_0}\int_{\tau_0}^\tau b^2 (Z^*)^{-d+ 2r}\int
 \frac{u_T^2+(Z^2\Delta u_T)^2}{\la Z\ra^{4}\rho_D^2}\le \mathcal D e^{-\mathcal b(d-2)\tau_0}\le e^{-\delta \tau_0},
\eee
where the penultimate and last inequalities hold since\footnote{Once again, dimensional restriction arises in the treatment 
of the dissipative term. It is not needed in the Euler case.} $d= 3$.
This concludes the proof of \eqref{estnienonneoinduction}.
\end{proof}

 %%%%%%%%%%%%%%%%%%%%%%%%%%%%%%%%%%%%%%%%

\section{$L^\infty$ bounds}

 %%%%%%%%%%%%%%%%%%%%%%%%%%%%%%%%%%%%%%%%

We are now in position to improve the bound  \eqref{smallglobalboot}.

\begin{lemma}[Improved $L^\infty$ bounds]
\label{lemmainprovedlinftysindaie}
For all $0\le k\le {k^\sharp}-2$, 
\be
\label{imprvedpoitnwibounds}
\left\|\frac{\la Z\ra^{k}\nabla^k\rhot}{\rho_D}\right\|_{L^\infty}+\left\|
\la Z\ra ^{k+(r-1)} \nabla^k\ut\right\|_{L^\infty(Z\leq Z^*)}\leq \mathcal d_0
\ee
and for all $0\le k\le {k^\sharp}-1$, 
\be
\label{imprvedpoitnwibounds:bisrepetita}
\left\|
\la Z\ra ^{k+(r-1)}\left\la \frac Z{Z^*}\right\ra^{-2(r-1)} \nabla^k\ut\right\|_{L^\infty(Z\geq 1)}\leq \mathcal d_0
\ee
\end{lemma}

\begin{proof}[Proof of Lemma \ref{lemmainprovedlinftysindaie}] For any spherically symmetric function vanishing at infinity
\be\label{triv}
|f|^2(Z)\le \int_{Z}^\infty Z^{-d} {|Z\pa_Z f|^2} Z^{d-1} dZ+ \int_{Z}^\infty Z^{-d} {|f|^2} Z^{d-1} dZ
\ee
We apply this to $f^2=\la Z\ra^{d}\chi_k \rho_D^2|\nabla^k\ut|^2$ with $\chi_k$ from \eqref{equivalnetofmrulachigk}. For $Z\ge 1$ we then obtain
$$
\la Z\ra^{d}\chi_k \rho_D^2|\nabla^k\ut|^2(Z) \lesssim \int  \chi_{k+1} \rho_D^2{|\nabla^{k+1}\ut|^2}+ \int
\chi_k \rho_D^2{|\nabla^k \ut|^2}\le e^{-2\sigma\tau}\mathcal d_0
$$
We now observe that from \eqref{pow}
$$
\la Z\ra^{d}\chi_k \rho_D^2\sim  \la Z\ra^{2k+2(r-1)-2\sigma}\left\la \frac Z{Z^*}\right\ra^{2\sigma-2(r-1)}. 
$$
The estimate \eqref{imprvedpoitnwibounds}
for $\nabla^k \ut(Z)$ with $Z\ge 1$ and $k\le {k^\sharp}-1$ follows immediately. For $Z\le 1$ the estimates for both 
$\nabla^k\rhot$ and $\nabla^k\ut$ for $k\le {k^\sharp}-2$ follow from the boundedness of the Sobolev norm 
$\|\rhot,\Psit\|_{{k^\sharp}}$ in dimension $d\le 3$.

The exterior estimates for $\rhot$ have been already established in \eqref{eq:strinfr}
$$
\left\|\left\la\frac Z{Z^*}\right\ra^{\frac \mu 2-\sigma}\frac{\la Z\ra^{k}\nabla^k\rhot}{\rho_D}\right\|_{L^\infty(Z\ge 12Z^*)}\le \mathcal d_0
$$
for any $0\le k\le {k^\sharp}-2$ and $\mu=\min\{1,2(r-1)\}$. It remains to prove \eqref{imprvedpoitnwibounds} for $\rhot$ for 
$1\le Z\le 12Z^*$. We again use \eqref{triv} but integrating from $12Z^*$ instead. Setting 
$f^2=\la Z\ra^{d}\chi_k \rho_D^{p-1}|\nabla^k\rhot|^2$, we obtain
\bee
\la Z\ra^{d}\chi_k \rho_D^{p-1}|\nabla^k\rhot|^2 &\lesssim& \la Z\ra^{d}\chi_k \rho_D^{p-1}|\nabla^k\rhot|^2|_{Z=12Z^*}+
\int_{Z\leq 12 Z^*}  \chi_{k+1} \rho_D^{p-1}{|\nabla^{k+1}\rhot|^2}\\
&&+ \int_{Z\leq 12 Z^*} 
\chi_k \rho_D^{p-1}{|\nabla^k \rhot|^2}
\eee
We now observe that for $Z\le 12 Z^*$ 
$$
\la Z\ra^{d}\chi_k \rho_D^{p-1}\sim \la Z\ra^{2k-2\sigma+\ell(r-1)}\sim  \frac{\la Z\ra^{2k-2\sigma}}{\rho_D^2},
$$
which implies 
$$
\la Z\ra^{-2\sigma} \frac{|\nabla^k\rhot|^2}{\rho_D^2}\le (Z^*)^{-2\sigma} \mathcal d_0+e^{-2\sigma\tau} \mathcal d_0
$$
The result now follows immediately.
\end{proof}

\section{Control of low Sobolev norms and proof of Theorem \ref{thmmain}}
\label{low}
 %%%%%%%%%%%%%%%%%%%%%%%%%%%%%%%%%%%%%%%%
  %%%%%%%%%%%%%%%%%%%%%%%%%%%%%%%%%%%%%%%%

Our aim in this section is to control weighted low Sobolev norms in the interior region $|x|\le 1$ which in renormalized variables corresponds to $Z\le Z^*$. On our way we will conclude the proof of the bootstrap Proposition \ref{propboot}. Theorem \ref{thmmain} will then follow from a classical topological argument. In this section all of the analysis will take
place in the region $Z\le 5Z^*$ where $\rho_D=\rho_P$ and $\Psi_D=\Psi_P$. We recall the decomposition \eqref{eq:lk}
$$
\rho_T=\bar\rho+\rho_P,\ \ \Psi_T=\Psi_P+\bar\Psi,\ \ \Phi=\rho_P\bar\Psi 
$$
and note that $(\bar\rho, \bar\Psi)=(\rhot,\Psit)$ for $Z\le 5 Z^*$.
 %%%%%%%%%%%%%%%%%%%%%%%%%%%%%%%%%%%%%%%%

\subsection{Exponential decay slightly beyond the light cone}

 %%%%%%%%%%%%%%%%%%%%%%%%%%%%%%%%%%%%%%%%

We use the exponential decay estimate \eqref{stabiliteexpo} for a linear problem to prove exponential decay 
for the nonlinear evolution in the region slightly past the light cone.
We recall the notations of Section 3, in particular $Z_a$ of Lemma \ref{shiftoight}.
\begin{lemma}[Exponential decay slightly past the light cone]
\label{lemmalightcone}
Let $$\tilde{Z_a}=\frac{Z_2+Z_a}2.$$ Then
\be
\label{firstbound}
 \|\nabla \Phi\|_{H^{2k_{\flat}}(Z\leq\tilde{Z_a})}+\|\bar\rho\|_{H^{2k_{\flat}}(Z\leq \tilde{Z_a})}\lesssim e^{-\frac{\delta_g}{2}\tau}.
  \ee
\end{lemma}

\begin{proof} The proof relies on the spectral theory beyond the light cone and an elementary finite speed propagation like 
argument in renormalized variables, related to \cite{MZwave}.\\

\noindent{\bf step 1} Semigroup decay in $X$ variables. Recall the definition \eqref{notatinotphi} of $X=(\Phi,\T)$ 
 \be\label{formulat}
  \left|\begin{array}{ll}
  \Phi=\rho_P\bar\Psi \\
   \T=\pa_\tau\Phi+aH_2\Lambda \Phi=-(p-1)\qx\bar\rhox-H_2\Lambda \Phix+(H_1-e)\Phix+G_\Phi+aH_2\Lambda \Phi
   \end{array}\right.
  \ee
with $G_\Phi$ given by \eqref{defgrho}, the scalar product \eqref{defscalarproduct} and the definitions
\eqref{nvknneknengno}, \eqref{eq:decomp}: $$\left|\begin{array}{l}
\Lambda_0=\{\l \in \Bbb C, \ \ \Re(\l)\ge 0\} \cap \{\l\ \ \mbox{is an eigenvalue of}\ \ \mathcal M\}=(\l_i)_{1\le i\le N}\\
V=\cup_{1\leq i\leq N}\mbox{\rm ker} (\mathcal M-\l_i I)^{k_{\l_i}}
\end{array}\right.
$$
the projection $P$ associated with $V$, the decay estimate \eqref{stabiliteexpo} on the range of $(I-P)$ and 
the results of Lemma \ref{browerset}. Relative to the $X$ variables our equations take the form 
$$
\pa_\tau X=\mathcal M X + G,
$$ 
which are considered on the time interval $\tau\ge \tau_0\gg 1$ and the space interval $Z\in [0,Z_a]$ (no boundary conditions
at $Z_a$.) We consider evolution in the Hilbert space ${\Bbb H_{2k_{\flat}}}$ with initial data such that 
\be\label{data}
\|(I-P) X(\tau_0)\|_{\Bbb H_{2k_{\flat}}}\le e^{-\frac{\delta_g}2\tau_0},\qquad \|P X(\tau_0)\|_{\Bbb H_{2k_{\flat}}}\le e^{-\frac {3\delta_g}5\tau_0}.
\ee
According to the bootstrap assumption \eqref{eq:unstX}
\be\label{eq:unstX'}
\|PX(\tau)\|_{\Bbb H_{2k_{\flat}}}\le e^{-\frac {\delta_g}{2}\tau}, \qquad \forall\tau\in [\tau_0,\tau^*]
\ee
Lemma \ref{browerset} shows that as long as 
\be\label{eq:G}
\|G\|_{\Bbb H_{2k_{\flat}}}\le e^{-\frac{2\delta_g}3\tau}, \qquad \tau\ge \tau_0
\ee
there exists $\Gamma$, which can
 be made as large as we want with a choice of $\tau_0$, such that 
  \be\label{eq:Xfgrow}
 \|PX(\tau)\|_{\Bbb H_{2k_{\flat}}}\lesssim e^{-\frac{\delta_g}{2}\tau},\qquad \tau_0\le \tau\le \tau_0+\Gamma.
 \ee
 This will allow us to show eventually that if we can verify \eqref{eq:G}, the bootstrap time $\tau^*\ge \tau_0+\Gamma$.
 
Moreover,  as long as \eqref{eq:G} holds, the decay estimate \eqref{stabiliteexpo} implies that
\bea
 \label{kmdommepeo}
 \nonumber &&\|(I-P)X(\tau)\|_{\Bbb H_{2k_{\flat}}}\lesssim e^{-\frac{\delta_g}{2}(\tau -\tau_0)}\|X(\tau_0)\|_{\Bbb H_{2k_{\flat}}}+\int_{\tau_0}^{\tau}e^{-\frac{\delta_g}{2}(\tau-\sigma)}\|G(\sigma)\|_{\Bbb H_{2k_{\flat}}}d\sigma\\
 & \lesssim & e^{\frac{-\delta_g}{2}\tau}\left[e^{\frac{\delta_g}{2}\tau_0}\|X(\tau_0)\|_{\Bbb H_{2k_{\flat}}}+\int_{\tau_0}^{+\infty}e^{-\frac{\delta_g}{6}\tau}d\tau\right]\le e^{-\frac{\delta_g}2\tau}.
 \eea
 As a result,
 \be\label{eq:Xf}
 \|X(\tau)\|_{\Bbb H_{2k_{\flat}}}\lesssim e^{-\frac{\delta_g}{2}\tau},\qquad \tau_0\le \tau\le \tau^*
 \ee
 Below we will verify \eqref{eq:G} $\forall \tau\in[\tau_0,\tau^*]$ under the assumption \eqref{kmdommepeo}, closing
 both. Once again,  this will allow us to show eventually that the length of the bootstrap interval $\tau^*-\tau_0\ge \Gamma$ is sufficiently large.\\ 
 
 Recall from \eqref{defgt}, \eqref{newlinearflow}, \eqref{defscalarproduct}:
 \be
 \label{vnebonennnevno}
 \|G\|^2_{\Bbb H_{2k_{\flat}}}\lesssim \int_{Z\le Z_a} |\nabla\Delta^{k_{\flat}}G_\T|^2gZ^{d-1}dZ+\int_{Z\le Z_a} G_\T^2Z^{d-1}dZ
 \ee
 with
 $$\left|\begin{array}{lll}
G_\T=\pa_\tau G_\Phi-\left(H_1+H_2\frac{\Lambda Q}{Q}\right)G_\Phi+H_2\Lambda G_\Phi-(p-1)QG_\rho\\
G_\rho=-\bar\rho\Delta \bar\Psi-2\nabla\bar\rho\cdot\nabla \bar\Psix\\
G_\Phi=-\bar\rho_P(|\nabla \bar\Psi|^2+\NL(\rho))+b^2\rho_P\mathcal F(u_T,\rho_T).
\end{array}\right.
$$
\noindent{\bf step 2} Semigroup decay for $(\bar\rho,\bar\Psi)$.  
 We now translate the $X$ bound to the bounds for $\bar\rho$ and $\bar\Psi$ and then verify \eqref{eq:G}.
 We recall \eqref{formulat} and obtain for any $\hat Z>Z_2$
  \bee
  \|\T\|_{H^{2k_{\flat}}(Z\le \hat{Z})}+\|\Phi\|_{H^{2k_{\flat}+1}(Z\le \hat{Z})}&\lesssim&\|\bar\rho\|_{H^{2k_{\flat}}(Z\le \hat{Z})}+\|\bar\Psi\|_{H^{2k_{\flat}+1}(Z\le \hat{Z})} +\|G_\Phi\|_{H^{2k_{\flat}}(Z\le \hat{Z})}\\&\lesssim& \|T\|_{H^{2k_{\flat}}(Z\le \hat{Z})}+\|\Phi\|_{H^{2k_{\flat}+1}(Z\le \hat{Z})}+\|G_\Phi\|_{H^{2k_{\flat}}(Z\le \hat{Z})}
 \eee
  and claim:
\be
\label{boundgphi}
\|G_{\Phi}\|_{H^{2k_{\flat}}(Z\le \hat{Z})}\lesssim \|\nabla \bar\Psi\|_{H^{2k_{\flat}}(Z\le \hat{Z})}^2+\|\bar\rho\|_{H^{2k_{\flat}}(Z\le \hat{Z})}^2+e^{-{\delta_g}\tau}.
\ee
Indeed, since $H^{2k_{\flat}}(Z\le \hat{Z})$ is an algebra for $k_{\flat}$ large enough:
  $$\|\rho_P(|\nabla \bar\Psi|^2+\NL(\rho))\|_{H^{2k_{\flat}}(Z\le \hat{Z})}\lesssim \|\nabla \bar\Psi\|_{H^{2k_{\flat}}(Z\le \hat{Z})}^2+\|\bar\rho\|_{H^{2k_{\flat}}(Z\le \hat{Z})}^2.$$
    The remaining term, see \eqref{def-f}, is treated using the pointwise bound \eqref{smallglobalboot} and the smallness of $b$ which imply: $$\|b^2\rho_P\mathcal F(u_T,\rho_T)\|_{H^{2k_{\flat}}(Z\leq \hat{Z})}\lesssim (Z_0)^Cb^2\leq  e^{-\delta_g\tau}$$
  provided $\delta_g>0$ has been chosen small enough, and \eqref{boundgphi} is proved. 
  Choosing $\hat Z>Z_2$, this implies from \eqref{formulat} and the initial bound \eqref{improvedsobolevlowinit}:
  \bea
  \label{boundonthedata}
  \nonumber \|X(\tau_0)\|_{\Bbb H^{2k_{\flat}}}&\lesssim&  \|\bar\Psi(\tau_0)\|_{H^{2k_{\flat}+1}(Z\le \hat{Z})}+\|\bar\rho(\tau_0)\|_{H^{2k_{\flat}}(Z\le \hat{Z})}+e^{-\delta_g\tau_0}\\
  &\lesssim & e^{-\frac{\delta_g\tau_0}2}.
  \eea
  This verifies \eqref{data}. On the other hand, choosing $\hat Z=\tilde Z_a$ with 
$$\tilde{Z_a}=\frac{Z_2+Z_a}2,$$ we also obtain from \eqref{eq:Xf}
  \be\label{eq:rhopsi}
  \|\bar\Psi(\tau)\|_{H^{2k_{\flat}+1}(Z\le \tilde{Z}_a)}+\|\bar\rho(\tau)\|_{H^{2k_{\flat}}(Z\le \tilde{Z}_a)}
  \lesssim \|X(\tau)\|_{\Bbb H^{2k_{\flat}}}+ e^{-\delta_g\tau}\lesssim  e^{-\frac{\delta_g\tau}2}.
  \ee
 The estimate \eqref{firstbound} follows.\\
  
 \noindent{\bf step 3} Estimate for $G$. Proof of \eqref{eq:G}. We recall \eqref{vnebonennnevno}.
 On a fixed compact domain $Z\le Z_0$ with $Z_0>Z_2$,
 we can interpolate the bootstrap bound \eqref{eq:bootdecay} with the global energy bound \eqref{cneioneoneon} and obtain for ${k^\sharp}$ large enough and $b_0<b_0({k^\sharp})$ small enough:
 \be
 \label{cneoineonenvoen}
 \|\bar\rho\|_{H^{2k_{\flat}+10}(Z\le Z_0)}+\|\bar\Psi\|_{H^{2k_{\flat}+10}(Z\le Z_0)}\leq C_Ke^{-\left[\frac{3}{8}-\frac{1}{100}\right]\delta_g\tau}\le e^{-\left[\frac{3}{8}-\frac{1}{50}\right]\delta_g\tau}
 \ee
 and since $H^{2k_{\flat}}$ is an algebra and all terms are either quadratic or with a $b$ term, \eqref{cneoineonenvoen} implies
\bea
 \label{estimatheG}
 \nonumber 
 &&\|G_\T\|_{H^{2k_{\flat}+5}(Z\leq Z_0)}+\|G_\rho\|_{H^{2k_{\flat}+5}(Z\leq Z_0)}+\|G_\Phi\|_{H^{2k_{\flat}+5}(Z\leq Z_0)}\\
 &\leq &e^{-\left(\frac 34-\frac1{20}\right)\delta_g\tau}\le e^{-\frac{2\delta_g}{3}\tau}
 \eea
 which in particular using \eqref{vnebonennnevno} implies \eqref{eq:G}.
\end{proof}

 %%%%%%%%%%%%%%%%%%%%%%%%%%%%%%%%%%%%%%%%

\subsection{Weighted decay for $m\le 2k_{\flat}$ derivatives}

 %%%%%%%%%%%%%%%%%%%%%%%%%%%%%%%%%%%%%%%%

We recall the notation \eqref{defnewvariablephi}. We now transform the exponential decay \eqref{firstbound} from just past the light cone into weighted decay estimate. It is {\em essential} for this argument that the decay \eqref{firstbound} has been 
shown in the region strictly including the light cone $Z=Z_2$. The estimates in the lemma below close the remaining bootstrap 
bound \eqref{eq:bootdecay}.

\begin{lemma}[Weighted Sobolev bound for $m\leq 2k_{\flat}$]
\label{lemmabootlocalnorms}
 Let $m\leq 2k_{\flat}$ and $\nu_0=\frac{\delta_g}{2}-\frac {2(r-1)}{p-1}$, recall \eqref{defchimunu}
 $$
 \xi_{\nu_0,m}=\frac{1}{\la Z\ra^{d-2(r-1)+2(\nu_0-m)}}\zeta\left(\frac{Z}{Z^*}\right), \ \ \zeta(Z)=\left|\begin{array}{ll}1\ \ \mbox{for}\ \ Z\leq 2\\ 0\ \ \mbox{for}\ \ Z\ge 3,
 \end{array}\right.
 $$
 then:
 \be
 \label{improvedsobolevlowbetter}
 \sum_{m=0}^{2k_{\flat}}\int (p-1)Q(\pa^m\bar\rho)^2\xi_{\nu_0,m}+|\nabla \pa^m\Phi|^2\xi_{\nu_0,m}\leq C e^{-\frac{4\delta_g}{5}\tau}.
 \ee
 \end{lemma}

\begin{proof}[Proof of Lemma \ref{lemmabootlocalnorms}]  The proof relies on a sharp energy estimate with time dependent localization of $(\bar\rho,\Phi)$. This is  a renormalized version of the finite speed of propagation. (Remember: this part of the 
argument treats the dissipative Navier-Stokes term as perturbation and, at the expense of loosing derivatives, relies on the structure of the compressible Euler equations.)\\

\noindent{\bf step 1} $\dot{H}^{m}$ localized energy identity. Pick a smooth well localized spherically symmetric function $\chi(\tau,Z)$. For  integer $m$ let
$$\bar\rho_m=\pa^m\bar\rho, \ \ \Phi_m=\pa^m\Phi.$$ 
We recall the Emden transform formulas \eqref{defhtwohunbis}:
$$
\left|\begin{array}{ll}
H_2=\mu(1-w)\\
H_1=\frac{\mu\ell}{2}(1-w)\left[1+\frac{\Lambda \sigma}{\sigma}\right]\\
H_3=\frac{\Delta \rho_P}{\rho_P}
\end{array}\right.
$$
which yield the bounds using \eqref{limitprofilesbsi}, \eqref{decayprofile}:
 \be
\label{esterrorpotentials}
\left|\begin{array}{llll}
H_2=1+O\left(\frac{1}{\la Z\ra^r}\right), \ \ H_1=-\frac{2(r-1)}{p-1}+O\left(\frac{1}{\la Z\ra^r}\right)\\
 |\la Z\ra^j\pa_Z^j H_1|+|\la Z\ra^j\pa_Z^jH_2|\lesssim \frac 1{\la Z\ra^{r}}, \ \ j\ge 1\\
 |\la Z\ra^j\pa_Z^j H_3|\lesssim \frac{1}{\la Z\ra^2}\\
  \frac{1}{\la Z\ra^{2(r-1)}}\left[1+O\left(\frac{1}{\la Z\ra^{r}}\right)\right]\lesssim_j |\la Z\ra^j\pa_Z^jQ|\lesssim_j \frac{1}{\la Z\ra^{2(r-1)}}
 \end{array}\right.
 \ee
and the commutator bounds:
$$
\left|\begin{array}{lllll}
|[\pa_i^m,H_1]\bar\rho|\lesssim  \sum_{j=0}^{m-1}\frac{|\pa_Z^j\bar\rho|}{\la Z\ra^{r+m-j}}\\
|\nabla\left([\pa_i^m,H_1]\bar\rho\right)|\lesssim \sum_{j=0}^{m}\frac{|\pa_Z^j\bar\rho|}{\la Z\ra^{m-j+r+1}}\\
|[\pa_i^m,Q]\bar\rho|\lesssim Q\sum_{j=0}^{m-1}\frac{|\pa_Z^j\bar\rho|}{\la Z\ra^{m-j}}\\
|[\pa_i^m,H_2]\Lambda\bar\rho|\lesssim \sum_{j=1}^{m}\frac{|\pa_Z^j\bar\rho|}{\la Z\ra^{r+m-j}}\\
|\nabla\left([\pa_i^m,H_2]\Lambda\Phi\right)|\lesssim \sum_{j=1}^{m+1}\frac{|\pa^j_Z\Phi|}{\la Z\ra^{r+1+m-j}}.
\end{array}\right.
$$
Commuting \eqref{nekoneneon} with $\pa_i^m$:
$$
\left|\begin{array}{ll}
\pa_\tau\bar\rho_m=H_1\bar\rho_m-H_2(m+\Lambda)\bar\rho_m-\Delta \Phi_m+\pa_i^mG_\rho+E_{m,\rho}\\
\pa_\tau\Phi_m=-(p-1)Q\bar\rho_m-H_2(m+\Lambda) \Phi_m+(H_1-(r-2))\Phi_m+\pa^m_iG_\Phi+E_{m,\Phi}
\end{array}\right.
$$
with the bounds $$\left|\begin{array}{ll} |E_{m,\rho}|\lesssim \sum_{j=0}^{m}\frac{|\pa_Z^j\bar\rho|}{\la Z\ra^{r-1+m-j}}+\sum_{j=0}^{m}\frac{|\pa_Z^j\Phi|}{\la Z\ra^{m-j+2}}\\  |\nabla E_{m,\Phi}|\lesssim Q\sum_{j=0}^{m}\frac{|\pa_Z^j\bar\rho|}{\la Z\ra^{m+1-j}}+ \sum_{j=0}^{m+1}\frac{|\pa^j_Z\Phi|}{\la Z\ra^{r+m-j}}.
\end{array}\right.$$

We derive  the corresponding energy identity:
\bee
&&\frac 12\frac{d}{d\tau}\left\{\int (p-1)Q\bar\rho_m^2\chi+|\nabla \Phi_m|^2\chi\right\}=\frac 12\int \pa_\tau\chi\left[(p-1)Q\bar\rho_m^2+|\nabla \Phi_m|^2\right]\\
& +& \int(p-1)Q\bar\rho_m\chi\left[H_1\bar\rho_m-H_2(m+\Lambda)\bar\rho_m-\Delta \Phi_m+\pa_i^mG_\rho+E_{m,\rho}\right]\\
& + & \int \chi\nabla \Phi_m\cdot\nabla\left[-(p-1)Q\bar\rho_m-H_2(m+\Lambda \Phi_m)+(H_1-(r-2))\Phi_m+\pa^i_mG_\Phi+E_{m,\Phi}\right]\\
& = & \frac 12\int \pa_\tau\chi\left[(p-1)Q\bar\rho_m^2+|\nabla \Phi_m|^2\right]\\
& + & \int(p-1)Q\bar\rho_m\chi\left[H_1\bar\rho_m-H_2(m+\Lambda)\bar\rho_m+\pa_i^mG_\rho+E_{m,\rho}\right]+\int(p-1)Q\bar\rho_m\nabla\chi\cdot\nabla \Phi_m\\
&+& \int \chi\nabla \Phi_m\cdot\nabla\left[-H_2(m+\Lambda) \Phi_m+(H_1-(r-2))\Phi_m+\pa^m_iG_\Phi+E_{m,\Phi}\right].\\
\eee
In what follows we will use $\omega>0$ as a small universal constant to denote the power of tails of the error terms. 
In most cases, the power is in fact $r>1$ which we do not need.\\
\noindent\underline{$\bar\rho_m$ terms}. From the asymptotic behavior of $Q$ \eqref{decayprofile} and \eqref{esterrorpotentials}:
\bee
&&-\int(p-1)Q\bar\rho_m\chi H_2\Lambda \bar\rho_m=\frac{p-1}2\int \bar\rho_m^2 \chi QH_2\left[d+\frac{\Lambda Q}{Q}+\frac{\Lambda H_2}{H_2}+\frac{\Lambda \chi}{\chi}\right]\\
&= & \int \bar\rho_m^2 (p-1)\chi Q\left[\frac d2-(r-1)+O\left(\frac{1}{\la Z\ra^\omega}\right)\right]+\frac 12\int (p-1)QH_2\Lambda \chi \bar\rho_m^2
\eee
\noindent\underline{$\Phi_m$ terms}. We first estimate recalling \eqref{esterrorpotentials}:
\bee
&&\int \chi\nabla \Phi_m\cdot\nabla\left[(-mH_2+H_1-(r-2))\Phi_m\right]\\
&=&\int(-mH_2+H_1-(r-2))\chi|\nabla\Phi_m|^2+O\left(\int \frac{\chi}{\la Z\ra^{r}}|\nabla\Phi_m||\Phi_m|\right)\\
& = & -\left[(m+r-2)+\frac{2(r-1)}{p-1}\right]\int \chi|\nabla\Phi_m|^2+O\left(\int \frac{\chi}{\la Z\ra^\omega}\left[|\nabla \Phi_m|^2+\frac{\Phi_m^2}{\la Z\ra^2}\right]\right)
\eee
We recall Pohozhaev identity for spherically symmetric functions
\bee
\nonumber \int_{\Bbb R^d} f\Delta g \pa_rgdx&=&c_d\int_{\Bbb R^+}\frac{f}{r^{d-1}} \pa_r(r^{d-1}\pa_rg)r^{d-1}\pa_rgdr\\
&=&-\frac 12\int_{\Bbb R^d}|\pa_rg|^2\left[f'-\frac{d-1}{r}f\right]dx
\eee
and for general functions
\bea
\label{pohozaevbispouet}
\nonumber \int\Delta g F\cdot\nabla gdx&=&\sum_{i,j=1}^d \int \pa_i^2 g F_j\pa_jgdx=-\sum_{i,j=1}^d \int\pa_ig(\pa_iF_j\pa_jg+F_j\pa^2_{i,j}g)\\
&=&-\sum_{i,j=1}^d \int\pa_iF_j\pa_ig\pa_jg+\frac 12\int |\nabla g|^2\nabla \cdot F.
\eea
Now, taking $F=\chi H_2(Z_1,\dots,Z_d)$ in the above:
\bee
&&-\int \chi\nabla \Phi_m\cdot\nabla (H_2\Lambda \Phi_m)=\int H_2\Lambda \Phi_m[ \chi \Delta \Phi_m+\nabla \chi\cdot\nabla \Phi_m]\\
& = & -\sum_{i,j=1}^d \int\pa_iF_j\pa_i\Phi_m\pa_j\Phi_m+\frac 12\int |\nabla \Phi_m|^2\nabla \cdot F+\int H_2\Lambda \Phi_m\nabla \chi\cdot\nabla \Phi_m\\
& = & \sum_{i,j=1}^d\pa_i\Phi_m\pa_j\Phi_m\left[-\pa_i(\chi H_2 Z_j)+H_2Z_j\pa_i\chi\right]+\frac 12\int |\nabla \Phi_m|^2\chi H_2\left[d+\frac{\Lambda\chi}{\chi}+\frac{\Lambda H_2}{H_2}\right]\\
& = &\frac{(d-2)}{2}\int\chi|\nabla \Phi_m|^2+\frac 12 \int H_2\Lambda \chi|\nabla \Phi_m|^2+O\left(\int \frac{\chi}{\la Z\ra^\omega}|\nabla \Phi_m|^2\right)
\eee

The collection of above bounds yields for some universal constant $\omega>0$ the weighted energy identity:
\bea
\label{energyidentittyk}
&&\frac 12\frac{d}{d\tau}\left\{\int (p-1)Q\bar\rho_m^2\chi+|\nabla \Phi_m|^2\chi\right\}\\
\nonumber & = &-\int\chi\left[(p-1)Q\bar\rho_m^2+|\nabla\Phi_m|^2\right]\left[\left(m-\frac d2+r-1\right)+\frac{2(r-1)}{p-1}+O\left(\frac{1}{\la Z\ra^\omega}\right)\right] \\
\nonumber& + & \frac 12\int (p-1)Q\bar\rho_m^2\left[\pa_\tau \chi+H_2\Lambda \chi\right]+\frac 12\int |\nabla \Phi_m|^2\left[\pa_\tau \chi+H_2\Lambda \chi\right]+\int(p-1)Q\bar\rho_m\nabla\chi\cdot\nabla \Phi_m\\
\nonumber& + & O\left(\int \chi\left[\sum_{j=0}^{m+1}\frac{|\pa_Z^j\Phi|^2}{\la Z\ra^{2(m+1-j)+\omega}}+\sum_{j=0}^{m}\frac{Q|\pa_Z^j\bar\rho|^2}{\la Z\ra^{2(m-j)+\omega}}\right]\right)\\
\nonumber&+ & O\left(\int \chi|\nabla \Phi_m||\nabla \pa^mG_\Phi|+\int\chi Q|\bar\rho_m||\pa^mG_\rho|\right)
\eea

\noindent{\bf step 2} Nonlinear and source terms. We claim the bound for $\chi=\xi_{\nu_0,m}$:
\bea
\label{cenoeneonoenevne}
\nonumber &&\sum_{m=0}^{2k_{\flat}} \sum_{i=1}^d\int \xi_{\nu_0,m}|\nabla \pa^mG_\Phi|^2+\int (p-1)Q\xi_{\nu_0,m}|\pa^mG_\rho|^2\\
& \lesssim & \left(\sum_{m=0}^{2k_{\flat}}\sum_{i=1}^d\int Q\bar\rho_m^2\xi_{\nu_0+1,m}+|\nabla \Phi_m|^2\xi_{\nu_0+1,m}\right)+e^{-c_g\tau}
\eea
for some positive $c_g>0$. 
\begin{remark}
\label{rem:cg}
Crucially, the constant $c_g$ can be chosen to be such that $c_g> \delta_g$. More accurately, the constant $c_g$
will be computed to explicitly depend on the speed $\mathcal e=\ell(r-1)+r-2$, $r$ and $\delta_g$. It will be clear
that adjusting $\delta_g$ while keeping all the other universal constants ($\ell, r$) fixed we can satisfy the inequality
$c_g> \delta_g$.
\end{remark}
\noindent\underline{$G_\rho$ term}. Recall \eqref {defgrho}
$$G_\rho=-\bar\rho\Delta \bar\Psi-2\nabla\bar\rho\cdot\nabla \bar\Psix, $$ then by Leibniz:
$$
|\pa^mG_\rho|^2\lesssim \sum_{j_1+j_2=m+2, j_2\geq 1}|\pa^{j_1}\bar\rho|^2|\pa^{j_2}\bar\Psi|^2.$$ 
We recall the pointwise bounds \eqref{smallglobalboot} for $Z\leq 3Z^*$,
\bee
|\pa^{j_1}\bar\rho|\le \frac{C_K}{\la Z\ra^{j_1+\frac{2(r-1)}{p-1}}}, \ \ |\pa^{j_2}\bar\Psi|&\leq& \frac{C_K}{\la Z\ra^{j_2+r-2}}.
\eee
This yields, recalling \eqref{beubeibebiev}, for $j_1\le 2k_{\flat}$:
 \bee
&&\int \xi_{\nu_0,m}Q|\pa^{j_1}\bar\rho|^2|\pa^{j_2}\bar\Psi|^2\lesssim \int Q\zeta\left(\frac{Z}{Z^*}\right) \frac{|\pa^{j_1}\bar\rho|^2}{Z^{2(j_2-m)+d-2(r-1)+2(r-2)+2\nu_0}}\\
& \lesssim & \int\zeta\left(\frac{Z}{Z^*}\right)Q\frac{|\pa^{j_1}\bar\rho|^2}{\la Z\ra^{d-2(r-1)+2(\nu_0-j_1)+2}}\lesssim  \sum_{j=0}^{j_1}\int \xi_{\nu_0+1,j_1}Q|\pa_Z^{j}\bar\rho|^2\\
& \lesssim & \sum_{m=0}^{2k_{\flat}}\sum_{i=1}^d\int Q\bar\rho_m^2\xi_{\nu_0+1,m}+|\nabla \Phi_m|^2\xi_{\nu_0+1,m}.
\eee
For $j_1=m+1$, $j_2=1$, we use the other variable:
\bee
&&\int \xi_{\nu_0,m}Q|\pa^{j_1}\bar\rho|^2|\pa^{j_2}\bar\Psi|^2\lesssim \int Q\zeta\left(\frac{Z}{Z^*}\right)\frac{|\pa^{j_2}\bar\Psi|^2}{Z^{2(j_1-m)+d-2(r-1)+\frac{4(r-1)}{p-1}+2\nu_0}}\\
& \lesssim & \int\zeta\left(\frac{Z}{Z^*}\right)\frac{\rho_P^2|\pa^{j_2}\bar\Psi|^2}{\la Z\ra^{d-2(r-1)+2(\nu_0-j_2)+2}}\lesssim \sum_{j=0}^{j_2}\int\zeta\left(\frac{Z}{Z^*}\right)\frac{|\pa_Z^{j}\Phi|^2}{\la Z\ra^{d-2(r-1)+2(\nu_0-j)+2}}\\
&\lesssim &  \sum_{j=0}^{j_2}\int \xi_{\nu_0+1,j}|\pa_Z^{j}\Phi|^2\lesssim  \sum_{m=0}^{2k_{\flat}}\sum_{i=1}^d\int Q\bar\rho_m^2\xi_{\nu_0+1,m}+|\nabla \Phi_m|^2\xi_{\nu_0+1,m}
\eee
and \eqref{cenoeneonoenevne} follows for $G_\rho$ by summation on $0\leq m\leq 2k_{\flat}$ .

\noindent\underline{$G_\Phi$ term}. Recall \eqref{defgrho}
$$G_\Phi=-\rho_P(|\nabla \bar\Psi|^2+\NL(\rho))+{b^2}\rho_P\mathcal F(u_T,\rho_T).$$
We estimate using the pointwise bounds \eqref{smallglobalboot} for $j_3\le 2k_{\flat}$:
\bee
&&|\nabla \pa^m(\rho_P|\nabla \bar\Psi|^2)|\lesssim \sum_{j_1+j_2+j_3=m+1,j_2\le j_3}\frac{\rho_P}{\la Z\ra^{j_1}}|\pa^{j_2+1}\bar\Psi\pa^{j_3+1}\bar\Psi|\\
&\lesssim&  \sum_{j_1+j_2+j_3=m+1,j_2\le j_3}\frac{1}{\la Z\ra^{\frac{2(r-1)}{p-1}+j_1+r-2+j_2+1}}|\pa^{j_3+1}\bar\Psi|\lesssim \sum_{j_3=0}^{2k_{\flat}}\frac{|\pa^{j_3+1}\Phi|}{\la Z\ra^{r+m-j_3}}
\eee
and  since $r>1$:
\bee
\sum_{j_3=0}^{2k_{\flat}}\int\xi_{\nu_0,m}\frac{|\pa^{j_3+1}\Phi|^2}{\la Z\ra^{2(r+m-j_3)}}\lesssim \sum_{j_3=0}^{2k_{\flat}}\int\xi_{\nu_0+1,j_3}|\nabla\Phi_{j_3}|^2.
\eee
For $j_3=2k_{\flat}+1$, we use the other variable and the conclusion follows similarly.

The dissipative term is estimated using the pointwise bounds \eqref{smallglobalboot}:
\bee
\int \xi_{\nu_0,m}&&\left|\nabla \pa^m\left(b^2\rho_P\mathcal F(u_T,\rho_T)\right)\right|^2\lesssim   b^4 \int_{Z\le 3Z^*} \xi_{\nu_0.m}\sum_{j=0}^{m+1}\rho_P^2\frac{|\pa^j\mathcal F(u_T,\rho_T)|^2}{\la Z\ra^{2(m+1-j)}}\\
&\lesssim&   b^4C_K \int_{Z\le 3Z^*}\frac{1}{\la Z\ra^{d+\delta_g-2(r-1)-\frac{4(r-1)}{p-1}-2m}}
\sum_{j=0}^{m+1}\frac{1}{\la Z\ra^{\frac{4(r-1)}{p-1}}}\frac{|\pa^j\mathcal F(u_T,\rho_T)|^2}{\la Z\ra^{2(m+1-j)}}\\
& \lesssim & b^4\int_{Z\le 3Z^*}\frac{\la Z\ra^{2(r-1)}}{\la Z\ra^{d+\delta_g+2}}\sum_{j=0}^{m+1}|\la Z\ra^j\pa^j\mathcal F(u_T,\rho_T)|^2
\eee
For $j\ge 1$, we estimate pointwise from \eqref{smallglobalboot}:
\bee
&&\la Z\ra^j|\pa^j\mathcal F(u_T,\rho_T)|\lesssim {(\mu+\mu')}\la Z\ra^j\left|\pa^{j-1}\left(\frac{\Delta u_T}{\rho_T^2}\right)\right|
\lesssim {(\mu+\mu')}\la Z\ra^j\sum_{j_1+j_2=j-1}\frac{|\pa^{j_1}\Delta u_T|}{\rho_T^2\la Z\ra^{j_2}}\\
&\lesssim&  {(\mu+\mu')}\la Z\ra^{j+\frac{4(r-1)}{p-1}}\sum_{j_1+j_2=j-1}\frac{1}{\la Z\ra^{r-1+j_1+2+j_2}}\lesssim {(\mu+\mu')}\frac{1}{\la Z\ra^{r-\frac{4(r-1)}{p-1}}}={(\mu+\mu')}\frac{\la Z\ra^{\ell(r-1)}}{\la Z\ra^{r}}
\eee
Therefore, recalling \eqref{definitioncontstas1}:
\bee
&&b^4\int_{Z\le 3Z^*}\frac{\la Z\ra^{2(r-1)}}{\la Z\ra^{d+\delta_g+2}}\sum_{j=1}^{m+1}|\la Z\ra^j\pa^j\mathcal F(u_T,\rho_T)|^2\\
&\lesssim& {(\mu+\mu')}\frac{1}{\la Z^*\ra^{2[\ell(r-1)+r-2]}}\int_{Z\leq 3Z^*}\frac{\la Z\ra^{2(r-2)}}{\la Z\ra^{1+\delta_g}}\frac{\la Z\ra^{2\ell(r-1)}}{\la Z\ra^{2r}}dZ\lesssim  e^{-c_g\tau},
\eee
where $c_g=\min\{2[\ell(r-1)+r-2], \delta_g+2r\}>0$.
For $j=0$, we have the bound:
$${(\mu+\mu')}|\mathcal F(u_T,\rho_T)|\lesssim {(\mu+\mu')}\int_0^Z\frac{dz}{\la z\ra^{r-1+2-\ell(r-1)}}={(\mu+\mu')}\int_0^Z\frac{dz}{\la z\ra^{1+r-\ell(r-1)}}$$ We observe at $r^*(3,\ell)$: 
\bee
&&r^*(\ell)-\ell(r^*(\ell)-1)>0\Leftrightarrow \ell(r^*(\ell)-1)<r^*(\ell)\Leftrightarrow \ell\left(\frac{\ell+3}{\ell+\sqrt{3}}-1\right)<\frac{\ell+3}{\ell+\sqrt{3}}\\
& \Leftrightarrow& (3-\sqrt{3})\ell<\ell+3\Leftrightarrow \ell<\frac{3}{2-\sqrt{3}}
\eee
which holds since $\ell<3<\frac{3}{2-\sqrt{3}}$.  
Therefore, in the case $r\sim r^*$, we have the estimate $|\mathcal F(u_T,\rho_T)|\lesssim {(\mu+\mu')}$, which yields the contribution:
 $${(\mu+\mu')}b^4\int_{Z\le 3Z^*} \frac{\la Z\ra^{2(r-1)}}{\la Z\ra^{d+\delta_g+2}}Z^{d-1}dZ\leq \frac{{(\mu+\mu')}}{\la Z^*\ra^{2[\ell(r-1)+r-2]}}\left(1+(Z^*)^{2(r-2)-\delta_g}\right)\le e^{-c_g\tau},$$ 
 where  $c_g=\min\{2[\ell(r-1)+r-2], \delta_g+2\ell(r-1)\}>0$.
In the case of $r\sim r_+$, we have either  $|\mathcal F(u_T,\rho_T)|\lesssim {(\mu+\mu')}$ in which case we obtain the bound as above, or $|\mathcal F(u_T,\rho_T)|\lesssim {(\mu+\mu')}Z^{\ell(r-1)-r}$. Then, we obtain 
$${(\mu+\mu')}b^4\int_{Z\le 3Z^*} \frac{\la Z\ra^{(r-2)+\ell(r-1)}}{\la Z\ra^{d+\delta_g+2}}Z^{d-1}dZ\lesssim \frac{{(\mu+\mu')}}{\la Z^*\ra^{[\ell(r-1)+r-2]}}\int_{Z\le 3Z^*}  \frac{dZ}{\la Z\ra^{1+\delta_g+2}}\lesssim e^{-2\mathcal{e}\tau}.$$  
This concludes the proof of \eqref{cenoeneonoenevne}.\\ 

\noindent{\bf step 2} Initialization and lower bound on the bootstrap time $\tau^*$.\\
 Fix a large enough $Z_0$ and pick a small enough universal constant $\omega_0$ such that 
\be
 \label{positiiviiefecorrection}
 \forall Z\ge 0, \ \ -\omega_0+H_2\ge \frac{\omega_0}{2}>0
 \ee
 and let $\Gamma=\Gamma(Z_0)$ such that  
\be
\label{definitionna}
\frac{Z_0}{2\hat{Z_a}}e^{-\omega_0\Gamma}=1.
\ee
We claim that provided ${\tau_0}$ has been chosen sufficiently large, the bootstrap time $\tau^*$ of Proposition \ref{propboot} satisfies the lower bound
\be
\label{vneiovnevneonvein}
\tau^*\ge \tau_0+\Gamma.
\ee
Indeed, in view of the results in sections 7 and 8 there remains to control the bound \eqref{eq:bootdecay} on $[\tau_0,\tau_0+\Gamma]$. By \eqref{eq:Xfgrow}, the desired bounds already hold for $Z\le \tilde Z_a$ on $[\tau_0,\tau_0+\Gamma]$.\\

We now run the energy estimate \eqref{energyidentittyk} with $\chi=\xi_{\nu_0,m}$ and obtain from \eqref{energyidentittyk}, \eqref{cenoeneonoenevne} and the Remark \ref{rem:cg} the rough bound on $[\tau_0,\tau^*]$:
$$\frac{d}{d\tau}\left\{\int (p-1)Q\bar\rho_m^2\xi_{\nu_0,m}+|\nabla \Phi_m|^2\xi_{\nu_0,m}\right\}\leq C\int (p-1)Q\bar\rho_m^2\xi_{\nu_0,m}+|\nabla \Phi_m|^2\xi_{\nu_0,m}+e^{-\delta_g\tau}.
$$
which yields using \eqref{improvedsobolevlowinit}:
\bee
&&\int (p-1)Q\bar\rho_m^2\xi_{\nu_0,m}+|\nabla \Phi_m|^2\xi_{\nu_0,m}\le e^{C(\tau-\tau_0)}\int (p-1)Q(\bar\rho_m(0))^2\xi_{\nu_0,m}+|\nabla \Phi_m(0)|^2\xi_{\nu_0,m}\\
& + & e^{C\tau}\int_{\tau_0}^{\tau}e^{-(C+\delta_g)\sigma}d\sigma\leq e^{C\Gamma}\left[C_0e^{-\delta_g\tau_0}+e^{-\delta_g\tau_0}\right]\leq 2e^{C\Gamma}C_0e^{-\delta_g\tau_0}
\eee
and hence
\bee
&&e^{\frac{4\delta_g}{5}\tau}\left[\int (p-1)Q\bar\rho_m^2\xi_{\nu_0,m}+|\nabla \Phi_m|^2\xi_{\nu_0,m}\right]\\
&\le&  e^{2C\Gamma}C_0e^{-\delta_g\tau_0}e^{\frac{4\delta_g}{5}\tau_0}\le e^{2C\Gamma}e^{-\frac{\delta_g}{10}\tau_0}\le 1
\eee
which concludes the proof of \eqref{vneiovnevneonvein} and \eqref{improvedsobolevlowbetter} for $\tau\in[\tau_0,\tau_0+\Gamma]$.\\

\noindent{\bf step 3} Finite speed of propagation.
We now pick a time $\tau_f\in[\tau_0+\Gamma,\tau^*]$ and $Z_a<Z_0<\infty$ and propagate the bound \eqref{firstbound} to the compact set $Z\leq {Z_0}$ using a finite speed of propagation argument. We claim:
\be
\label{lowsobolev}
\|\bar\rho\|^2_{H^{2k_{\flat}}(Z\leq \frac{Z_0}{2})}+\|\nabla \bar\Psi\|^2_{H^{2k_{\flat}}(Z\le \frac{Z_0}{2})}\le Ce^{-{\delta_g}\tau}.
\ee
 Here the key is that \eqref{firstbound} controls a norm on the set {\em strictly including} the light cone $Z\le Z_2$. 
 Let $$\hat{Z_a}=\frac{\tilde{Z_a}+Z_2}{2}$$ and note that we may, without loss of generality by taking $a>0$ small enough, assume:
\be
\label{estnisneione}
\frac{\tilde{Z}_a}{\hat{Z_a}}\leq 2.
\ee
Recall that $\Gamma=\Gamma(Z_0)$ is parametrized by \eqref{definitionna}. We define $$\chi(\tau,Z)=\zeta\left(\frac{Z}{\nu(\tau)}\right), \ \ \nu(\tau)=\frac{Z_0}{2\hat{Z}_a}e^{-\omega_0(\tau_f-\tau)}$$ with $\omega_0>0$ 
defined in \eqref{positiiviiefecorrection} and \eqref{definitionna} and a fixed spherically symmetric non-increasing cut off function 
\be
\label{defzetavneneov}
\zeta(Z)=\left|\begin{array}{ll} 1\ \ \mbox{for}\ \ 0\leq Z\leq \hat{Z_a}\\ 0\ \ \mbox{for}\ \ Z\geq \tilde{Z}_a.
\end{array}\right., \ \ \zeta'\leq 0
\ee 
We define $$\tau_\Gamma=\tau_f-\Gamma$$ so that from \eqref{definitionna}:
\be
\label{defstopingtime}
\left|\begin{array}{l}
 \tau_0\le \tau_\Gamma\le \tau^*\\
\nu(\tau_\Gamma)=\frac{Z_0}{2\hat{Z}_a}e^{-\omega_0(\tau_f-\tau_\Gamma)}=\frac{Z_0}{2\hat{Z_a}}e^{-\omega_0\Gamma}=1.
\end{array}\right.
\ee

We pick $$0\leq m\leq 2k_{\flat}$$ 
then  \eqref{defzetavneneov}, \eqref{defstopingtime} ensure ${\rm Supp}(\chi(\tau_\Gamma,\cdot))\subset \{Z\le \tilde{Z}_a\}$ and hence from \eqref{firstbound}:
\be
\label{initiazliationtauagtronwall}
\left(\int (p-1)Q\bar\rho_m^2\chi+|\nabla \Phi_m|^2\chi\right)(\tau_\Gamma)\lesssim e^{-\delta_g \tau_\Gamma}.
\ee
This estimate implies that we can integrate energy identity \eqref{energyidentittyk} {\it just} on the interval 
$[\tau_\Gamma,\tau_f]$. 
We now estimate all terms in  \eqref{energyidentittyk}.\\

\noindent\underline{Boundary terms}. We compute the quadratic terms involving $\Lambda \chi$ which should be thought of as boundary terms. First $$\pa_\tau\chi(\tau,Z)=-\frac{\pa_\tau \nu}{\nu} \frac{Z}{\nu}\pa_Z\zeta\left(\frac{Z}{\nu}\right)=-\omega_0\Lambda \chi.$$ 
We now assume that $\omega_0$ has been chosen small enough so that \eqref{positiiviiefecorrection} holds, and hence the lower bound on the full boundary quadratic form using $\Lambda \chi\le 0$:
\bee
 &&\frac 12\int (p-1)Q\bar\rho_m^2\left[\pa_\tau \chi+H_2\Lambda \chi\right]+\frac 12\int |\nabla\Phi_m|^2\left[\pa_\tau \chi+H_2\Lambda \chi\right]+\int(p-1)Q\bar\rho_m\nabla\chi\cdot\nabla \Phi_m\\
 & = & \int\left\{\frac 12  (p-1)Q\bar\rho_m^2\left[-\omega_0+H_2\right]+\frac 12|\nabla\Phi_m|^2\left[-\omega_0+H_2\right]+(p-1)\frac{Q}{Z}\pa_Z \Phi_m\bar\rho_m\right\}\Lambda \chi.
 \eee

The discriminant of the above quadratic form is given by the following expression in the variables of Emden
transform 
\bee
&& \left[(p-1)\frac{Q}Z\right]^2-(-\omega_0+H_2)^2(p-1)Q=(p-1)Q\left[\frac{(p-1)Q}{Z^2}-(-\omega_0+H_2)^2\right]\\
 &= & (p-1)Q\left[\sigma^2-\left(-\omega_0+1-w\right)^2\right]= (p-1)Q\left[-D(Z)+O(\omega_0)\right].
 \eee
 where $D(Z)=(1-w)^2-\sigma^2$, see Lemma \ref{shiftoight}.

 We then observe by definition of $\chi$ that for $\tau\geq \tau_\Gamma$: $$Z\in {\rm Supp}\Lambda \chi\Leftrightarrow \hat{Z}_a\leq \frac{Z}{\nu(\tau)}\leq \tilde{Z}_a \Rightarrow  Z\geq \nu(\tau) \hat{Z}_a\geq \nu(\tau_\Gamma) \hat{Z}_a=\hat{Z}_a$$ from which since $\hat{Z}_a>Z_2$: $$Z\in {\rm Supp}\Lambda \chi \Rightarrow -D(Z)+O(\omega_0)<0$$ provided $0<\omega_0\ll 1$ has been chosen small enough. 

 Together with \eqref{positiiviiefecorrection} and $\Lambda\chi<0$, this ensures: $\forall \tau\in [\tau_\Gamma,\tau^*]$, 
\bea
\label{signfnonon}
\nonumber &&\frac 12\int (p-1)Q\bar\rho_m^2\left[\pa_\tau \chi+H_2\Lambda \chi\right]+\frac 12\int |\nabla\Phi_m|^2\left[\pa_\tau \chi+H_2\Lambda \chi\right]+\int(p-1)Q\bar\rho_m\nabla\chi\cdot\nabla \Phi_m\\
&<&0
\eea

\noindent\underline{Nonlinear terms}. From \eqref{defzetavneneov}, \eqref{estnisneione} for $\tau\leq \tau_f$:
$$\mbox{Supp}\chi\subset\{Z\leq \nu(\tau)\tilde{Z}_a\}\subset \{Z\leq \nu(\tau_f)\tilde{Z}_a\}=\left\{Z\leq \frac{Z_0}{2}\frac{\tilde{Z}_a}{\hat{Z}_a}\right\}\subset\{Z\leq Z_0\},$$ and hence from \eqref{estimatheG}:
$$\int \chi|\nabla \pa^mG_\Phi|^2+\int (p-1)Q\chi|\pa^mG_\rho|^2\lesssim  \|\nabla G_\Phi\|^2_{H^{2k_{\flat}}(Z\leq Z_0)}+\|\pa^mG_\rho\|^2_{H^{2k_{\flat}}(Z\leq Z_0)}\leq  e^{-\frac{4\delta_g}{3}\tau}.
$$
\noindent\underline{Conclusion}. Inserting the collection of above bounds into \eqref{energyidentittyk} and summing over $m\in[0,2k_{\flat}]$ yields the crude bound: $\forall \tau\in [\tau_\Gamma,\tau_f]$,
$$
\frac{d}{d\tau}\left\{\sum_{m=0}^{2k_{\flat}}\int (p-1)Q\bar\rho_m^2\chi+|\nabla \Phi_m|^2\chi\right\}\leq C \sum_{m=0}^{2k_{\flat}}\int (p-1)Q\bar\rho_m^2\chi+|\nabla \Phi_m|^2\chi+e^{-\frac{4\delta_g}{3}\tau}.
$$
We integrate the above  on $[\tau_\Gamma,\tau_f]$
and conclude using $$\chi(\tau_f,Z)=\zeta\left(\frac{Z}{\nu(\tau_f)}\right)=\zeta\left(\frac{Z}{\frac{Z_0}{2\hat{Z}_a}}\right)=1\ \ \mbox{for}\ \ Z\leq {Z_0}$$ and the initial data \eqref{initiazliationtauagtronwall}:
\bee
&&\left[\|\bar\rho\|^2_{H^{2k_{\flat}}(Z\leq {Z_0})}+\|\nabla \bar\Psi\|^2_{H^{2k_{\flat}}(Z\leq {Z_0})}\right](\tau_f)\\
& \lesssim & e^{C(\tau_f-\tau_\Gamma)}e^{-\delta_g\tau_\Gamma}+\int_{\tau_\Gamma}^{\tau_f}e^{C(\tau_f-\sigma)}e^{-\frac{4\delta_g}{3}\sigma}d\sigma\lesssim C(\Gamma)e^{-\delta_g\tau_f}=C(Z_0)e^{-\delta_g\tau_f}.
\eee
Since the time $\tau_f$ is arbitrary in $[\tau_0+\Gamma,\tau^*]$, the bound \eqref{lowsobolev} follows.\\

\noindent{\bf step 4} Proof of \eqref{improvedsobolevlowbetter}. We run the energy identity \eqref{energyidentittyk} with $\xi_{\nu_0,m}$  and estimate each term.\\

 \noindent\underline{\em terms $\frac{Z_0}{3}\leq Z\leq \frac{Z_0}{2}$}. In this zone, we have by construction $$\bar\rho=\rhot$$ and hence the bootstrap bounds \eqref{boundbootbound} imply $$\|\bar\rho\|_{H^{{k^\sharp}}(Z\leq \frac{Z_0}{2})}+\|\nabla \bar\Psi\|_{H^{{k^\sharp}}(Z\leq \frac{Z_0}{2})}\lesssim 1$$ and hence interpolating with \eqref{lowsobolev} for ${k^\sharp} $ large enough:
 \bea
 \label{inteprolsaiotnnormone}
 \nonumber \|\bar\rho\|_{H^{m}(\frac{Z_0}{3}\leq Z\leq \frac{Z_0}{2})}&\lesssim& \|\bar\rho\|^{\frac{m}{{k^\sharp}}}_{H^{{k^\sharp}}(\frac{Z_0}{3}\leq Z\leq \frac{Z_0}{2})}\|\bar\rho\|^{1-\frac{m}{{k^\sharp}}}_{L^2(\frac{Z_0}{3}\leq Z\leq \frac{Z_0}{2})}\lesssim e^{-\frac{\delta_g}{2}\left(1-\frac{m}{{k^\sharp}}\right)}\\
 &\leq& e^{-\frac{4\delta_g}{10}}
 \eea
  and similarly 
  \be
 \label{inteprolsaiotnnormbisone}
 \|\nabla\bar\Psi\|_{H^{m}(\frac{Z_0}{3}\leq Z\leq \frac{Z_0}{2})}\lesssim e^{-\frac{\delta_g}{2}\left(1-\frac{m}{{k^\sharp}}\right)}\leq e^{-\frac{4\delta_g}{10}}.
 \ee
 
 \noindent\underline{\em Linear term}. We observe the cancellation using \eqref{esterrorpotentials}, \eqref{renormalization}:
  \bea
 \label{vnnenoenoene}
 \nonumber &&\pa_\tau\xi_{\nu_0,m}+H_2\Lambda \xi_{\nu_0,m}=\frac{1}{\la Z\ra^{d-2(r-1)+2(\nu_0-m)}}\left[-\Lambda \zeta\left(\frac{Z}{Z^*}\right)\right]\\
 \nonumber &+& (1-w)\left[\frac{1}{\la Z\ra^{d-2(r-1)+2(\nu_0-m)}}\Lambda \zeta\left(\frac{Z}{Z^*}\right)+\Lambda \left(\frac{1}{\la Z\ra^{d-2(r-1)+2(\nu_0-m)}}\right)\zeta\left(\frac{Z}{Z^*}\right)\right]\\
 & = &-\left[d-2(r-1)+2(\nu_0-m)\right]\xi_{\nu_0,m}+O\left(\frac{1}{\la Z\ra^{d-2(r-1)+2(\nu_0-m)+\omega}}\right)
\eea
for some universal constant $\omega>0$. We now estimate the norm for $2Z^*\le Z\le 3Z^*$.
Using spherical symmetry for $Z\ge 1$ and $m\ge 1$:
\be
\label{beubeibebiev}
|Z^{m}\pa^m\bar\rho|\lesssim \sum_{j=1}^mZ^{m}\frac{|\pa_Z^j\bar\rho|}{Z^{m-j}}\lesssim  \sum_{j=1}^m Z^{j}|\pa_Z^j\bar\rho|
\ee
and hence using the outer $L^\infty$ bound \eqref{smallglobalboot}:
\bea
\label{beibeibeibiebebv}
\nonumber &&\int_{2Z^*\leq Z\leq 3Z^*}\frac{(p-1)Q|\pa^m\bar\rho|^2+|\pa^m\nabla \Phi|^2}{\la Z\ra^{d-2(r-1)+2(\nu_0-m)+\omega}}\\
\nonumber&\lesssim&\int_{2Z^*\leq Z\leq 3Z^*}\left[\sum_{j=0}^m \left|\frac{Z^{j}\pa_Z^j\bar\rho}{\la Z\ra^{\frac d2+\nu_0+\frac{\omega}2}}\right|^2+\sum_{j=1}^{m+1}\left|\frac{Z^{j}\pa_Z^j\Phi}{(Z^*)^{\nu_0+\frac d2-(r-1)+1+\frac{\omega}{2}}}\right|^2\right]\\
\nonumber& \lesssim & \int_{2Z^*\leq Z\leq 3Z^*}\left[\sum_{j=0}^m \left|\frac{Z^{j}\pa_Z^j\bar\rho}{\rho_P\la Z\ra^{\frac d2+\nu_0+\frac{2(r-1)}{p-1}+\frac{\omega}2}}\right|^2+\sum_{j=1}^{m+1}\left| \la Z\ra^{r-2}\frac{Z^{j}\pa_Z^j\bar\Psi}{\la Z\ra^{\nu_0+\frac{2(r-1)}{p-1}+\frac d2+\frac{\omega}{2}}}\right|^2\right]\\
& \lesssim & \frac{1}{(Z^*)^{\omega+2\left[
\nu_0+\frac{2(r-1)}{p-1}\right]}}\leq e^{-\delta_g\tau}
\eea
using 
the explicit choice from \eqref{venovnoenneneo}: $$2\left(\nu_0+\frac{2(r-1)}{p-1}\right)=\delta_g$$
\noindent\underline{Conclusion} Inserting the above bounds into \eqref{energyidentittyk} yields:
\bee
&&\frac 12\frac{d}{d\tau}\left\{\int (p-1)Q\bar\rho_m^2\xi_{\nu_0,m}+|\nabla \Phi_m|^2\xi_{\nu_0,m}\right\}\\
\nonumber & = &-\int\xi_{\nu_0,m}\left[(p-1)Q\bar\rho_m^2+|\nabla\Phi_m|^2\right]\left[\nu_0+\frac{2(r-1)}{p-1}\right] \\
\nonumber& + & O\left(\int_{Z_0\leq Z\leq 2Z^*} \xi_{\nu_0,m}\left[\sum_{m=0}^{m+1}\frac{|\pa_Z^j\Phi|^2}{\la Z\ra^{2(m+1-j)+2\omega}}+\sum_{j=0}^{m}\frac{Q|\pa_Z^j\bar\rho|^2}{\la Z\ra^{2(m-j)+2\omega}}\right]+e^{-\frac{4\delta_g}{5}\tau}\right)\\
\nonumber&+ &  O\left(\int \xi_{\nu_0,m}|\nabla \Phi_m||\nabla \pa^mG_\Phi|+\int\xi_{\nu_0,m}Q|\bar\rho_m||\pa^mG_\rho|\right)\eee
and hence after summing over $m$:
\bee
&&\frac 12\frac{d}{d\tau}\left\{\sum_{m=0}^{2k_{\flat}}\int (p-1)Q\bar\rho_m^2\xi_{\nu_0,m}+|\nabla \Phi_m|^2\xi_{\nu_0,m}\right\}\\
\nonumber& = & -\left[\nu_0+\frac{2(r-1)}{p-1}\right]\sum_{m=0}^{2k_{\flat}} \int\xi_{\nu_0,m}\left[(p-1)Q\bar\rho_m^2+|\nabla\Phi_m|^2\right]\\
& + & O\left(e^{-\frac{4\delta_g}{5}\tau}+\sum_{m=0}^{2k_{\flat}}\int (p-1)Q\bar\rho_m^2\xi_{\nu_0+\omega,m}+|\nabla \Phi_m|^2\xi_{\nu_0+\omega,m}\right)\\
\nonumber&+ &  \sum_{m=0}^{2k_{\flat}}O\left(\int \xi_{\nu_0,m}|\nabla \Phi_m||\nabla \pa^mG_\Phi|+\int\xi_{\nu_0,m}Q|\bar\rho_m||\pa^mG_\rho|\right)
\eee
Using \eqref{lowsobolev} we conclude 
\bea
\label{estigneononeo}
&&\frac 12\frac{d}{d\tau}\left\{\sum_{m=0}^{2k_{\flat}}\int (p-1)Q\bar\rho_m^2\xi_{\nu_0,m}+|\nabla \Phi_m|^2\xi_{\nu_0,m}\right\}\\
\nonumber& = & -\left[\nu_0+\frac{2(r-1)}{p-1}+O\left(\frac{1}{Z_0^C}\right)\right]\sum_{m=0}^{2k_{\flat}} \int\xi_{\nu_0,m}\left[(p-1)Q\bar\rho_m^2+|\nabla\Phi_m|^2\right]\\
\nonumber&+ & O\left(e^{-\frac{4\delta_g}{5}}+\sum_{m=0}^{2k_{\flat}} \int \xi_{\nu_0,m}|\nabla \pa^mG_\Phi|^2+\int (p-1)Q\xi_{\nu_0,m}|\pa^mG_\rho|^2\right).
\eea
Therefore, using also \eqref{cenoeneonoenevne}, for $Z_0$ large enough and universal and 
$$2\left(\nu_0+\frac{2(r-1)}{p-1}\right)=\delta_g,$$ 
there holds
 \bee
 &&\frac{d}{d\tau}\left\{\sum_{m=0}^{2k_{\flat}}\int (p-1)Q\bar\rho_m^2\xi_{\nu_0,m}+|\nabla \Phi_m|^2\xi_{\nu_0,m}\right\}\\
\nonumber& \le & -\frac{9}{10}\delta_g\sum_{m=0}^{2k_{\flat}} \int\xi_{\nu_0,m}\left[(p-1)Q\bar\rho_m^2+|\nabla\Phi_m|^2\right]+C e^{-\frac{4\delta_g\tau}{5}}.
\eee
Integrating in time and using \eqref{improvedsobolevlowinit} yields \eqref{improvedsobolevlowbetter}.
 \end{proof}

%%%%%%%%%%%%%%%%%%%%%%%%%%%%%%%%%%%%%%

\subsection{Closing the bootstrap and proof of Theorem \ref{thmmain}}
\label{proofthmmamin}
%%%%%%%%%%%%%%%%%%%%%%%%%%%%%%%%%%%%%%%%

At this point all the required bounds of the bootstrap Proposition \ref{propboot} have been improved. This now will immediately imply Theorem \ref{thmmain}.

 \begin{proof}[Proof of Theorem \ref{thmmain}] 
We conclude the proof with a classical topological argument \`a la Brouwer. The bounds of sections 5,6,7,8 have been  shown to hold for all initial data on the time 
 interval $[\tau_0, \tau_0+\Gamma]$ with $\Gamma$ large. Moreover, as explained in the proof of Lemma \ref{lemmalightcone},
 they can be immediately propagated to any time $\tau^*$ after a choice of projection of initial data on the subspace of unstable 
 modes $PX(\tau_0)$. This choice is dictated by Lemma \ref{browerset}. A continuity argument implies $\tau^*=\infty$ for this data, and the conclusions of Theorem \ref{thmmain} follow.
 \end{proof}

%%%%%%%%%%%%%%%%%%%%%%%%%%%%%%%%%%%%%%%%%%%%%%%%%%%%%%%%%%%
%%%%%%%%%%%%%%%%%%%%%%%%%%%%%%%%%%%%%%%%%%%%%%%%%%%%%%%%%%%

\begin{appendix}

%%%%%%%%%%%%%%%%%%%%%%%%%%%%%%%%%%%%%%%%%%%%%%%%%%%%%%%%%%%
%%%%%%%%%%%%%%%%%%%%%%%%%%%%%%%%%%%%%%%%%%%%%%%%%%%%%%%%%%%

%%%%%%%%%%%%%%%%%%%%%%%%%%%%%%%%%%%%%%%%%%%%%%%%%%%%%%%%%%%
%%%%%%%%%%%%%%%%%%%%%%%%%%%%%%%%%%%%%%%%%%%%%%%%%%%%%%%%%%%

\section{Hardy inequality}

%%%%%%%%%%%%%%%%%%%%%%%%%%%%%%%%%%%%%%%%%%%%%%%%%%%%%%%%%%%
%%%%%%%%%%%%%%%%%%%%%%%%%%%%%%%%%%%%%%%%%%%%%%%%%%%%%%%%%%%

\begin{lemma}[Hardy inequality]
There holds for $\alpha\neq 2-d$ and $r_0>0$:
\be
\label{hardyinetatiliy}
\int_{|x|\ge r_0}|x|^{\alpha-2}|u|^2dx\le c_{r_0,\alpha}\|u\|_{L^\infty(r=r_0)}^2+\frac{4}{(d-2+\alpha)^2}\int_{|x|\ge r_0}|x|^{\alpha}|\nabla u|^2dx.
\ee
\end{lemma}

\begin{proof} 
We compute $$\nabla \cdot (r^{\alpha-1}e_r)=\frac{1}{r^{d-1}}\pa_r(r^{d-1+\alpha-1})=(d-2+\alpha)r^{\alpha-2}$$ and hence
\bee
&&\int_{|x|\ge r_0}|x|^{\alpha-2}u^2dx=\frac{1}{d-2+\alpha}\int_{|x|\ge r_0}u^2\nabla \cdot (r^{\alpha-1}e_r)dx\\
&=&\frac{1}{d-2+\alpha}\int_{|x|=r_0}r^{\alpha-1}u^2d\sigma -\frac{2}{d-2+\alpha}\int_{|x|\ge r_0}r^{\alpha-1}\pa_ru udx\\
& \leq & c_{r_0}\|u\|_{L^\infty(r=r_0)}^2+\frac{2}{|d-2+\alpha|}\left(\int_{|x|\ge r_0}|x|^{\alpha-2}u^2dx\right)^{\frac 12}\left(\int_{|x|\ge r_0}|x|^{\alpha}|\nabla u|^2dx\right)^{\frac 12}
\eee
and \eqref{hardyinetatiliy} is proved using H\"older and optimizing the constant.
\end{proof}

%%%%%%%%%%%%%%%%%%%%%%%%%%%%%%%%%%%%%%%%%%%%%%%%%%%%%%%%%%%
%%%%%%%%%%%%%%%%%%%%%%%%%%%%%%%%%%%%%%%%%%%%%%%%%%%%%%%%%%%

\section{Commutator for $\Delta^k$}

%%%%%%%%%%%%%%%%%%%%%%%%%%%%%%%%%%%%%%%%%%%%%%%%%%%%%%%%%%%
%%%%%%%%%%%%%%%%%%%%%%%%%%%%%%%%%%%%%%%%%%%%%%%%%%%%%%%%%%%

\begin{lemma}[Commutator for $\Delta^k$]
Let $k\ge 1$, then for any two smooth function $V,\Phi$, there holds:
\be
\label{estimatecommutatorvlkeveln}
[\Delta^k,V]\Phi-2k\nabla V\cdot\nabla \Delta^{k-1}\Phi=\sum_{|\alpha|+|\beta|=2k,|\beta|\le 2k-2}c_{k,\alpha,\beta}\nabla^\alpha V\nabla^\beta\Phi.
\ee
where $\nabla^\alpha=\pa_1^{\alpha_1}\dots\pa^{\alpha_d}_d$, $|\alpha|=\alpha_1+\dots+\alpha_d.$
\end{lemma}

\begin{proof} 
We argue by induction on $k$. For $k=1$:
$$\Delta(V\Phi)-V\Delta \Phi=2\nabla V\cdot\nabla \Phi.$$
We assume \eqref{estimatecommutatorvlkeveln} for $k$ and prove $k+1$. Indeed,
\bee
&&\Delta^{k+1}(V\Phi)=\Delta([\Delta^k,V]\Phi+V\Delta^k\Phi)\\
&=&\Delta\left(2k\nabla V\cdot\nabla\Delta^{k-1} \Phi+\sum_{|\alpha|+|\beta|=2k, |\beta|\le 2k-2}c_{k,\alpha,\beta}\nabla^\alpha V\nabla^\beta\Phi+V\Delta^k\Phi\right)=  2k\nabla V\cdot\nabla \Delta^{k}\Phi\\
&+&\sum_{|\alpha|+|\beta|=2k+2, |\alpha|\ge 1}\tilde{c}_{k,\alpha,\beta}\nabla^\alpha V\nabla^\beta\Phi+2k\nabla V\cdot\nabla \Delta^k\Phi+V\Delta^{k+1}\Phi+2\nabla V\cdot\nabla \Delta^k\Phi\\
& = & V\Delta^{k+1}\Phi+2(k+1)\nabla V\cdot\nabla \Delta^{k}\Phi+\sum_{|\alpha|+|\beta|=2k+2, |\alpha|\ge 1}c_{k+1,\alpha,\beta}\nabla^\alpha V\nabla^\beta\Phi
\eee
and \eqref{estimatecommutatorvlkeveln} is proved.
\end{proof}

\end{appendix}

\end{document}